\documentclass[11pt]{amsart}

\usepackage{amsthm,amssymb,stmaryrd}
\usepackage[unicode]{hyperref}
\usepackage{xcolor}
\usepackage{tikz-cd}
\usepackage{tikzit}

\usepackage[style=numeric-comp,giveninits=true,url=false,maxbibnames=100,backend=bibtex]{biblatex}
\bibliography{references.bib}

\author{Yakov Berchenko-Kogan and Evan S. Gawlik}

\title[Blow-up Whitney forms]{Blow-up Whitney forms, shadow forms, and Poisson processes}

\newtheorem{theorem}{Theorem}[section]
\newtheorem{proposition}[theorem]{Proposition}
\newtheorem{corollary}[theorem]{Corollary}
\newtheorem{lemma}[theorem]{Lemma}
\newtheorem{conjecture}[theorem]{Conjecture}
\theoremstyle{definition}
\newtheorem{definition}[theorem]{Definition}
\newtheorem{notation}[theorem]{Notation}
\newtheorem{example}[theorem]{Example}
\newtheorem{remark}[theorem]{Remark}
\newtheorem{intuition}[theorem]{Intuition}

\newcommand\abs[1]{\left\lvert{#1}\right\rvert}

\newcommand\pp[2][]{\frac{\partial{#1}}{\partial{#2}}}

\renewcommand\d{\mathop{}\!d}
\newcommand\bR{\mathbb R}

\newcommand\mr{\mathring}

\newcommand\brho{{\boldsymbol\rho}}
\newcommand\btheta{{\boldsymbol\theta}}
\newcommand\blambda{{\boldsymbol\lambda}}
\newcommand\bl{\boldsymbol{l}}
\newcommand\br{\boldsymbol{r}}
\newcommand\rr[1]{\lambda_{#1}}
\newcommand\tl\tilde
\newcommand\cP{\mathcal P}

\renewcommand\epsilon\varepsilon
\renewcommand\phi\varphi

\DeclareMathOperator\id{id}
\DeclareMathOperator\Lk{Lk}

\colorlet{darkgreen}{green!50!black}

\begin{document}

\maketitle

\begin{abstract}
  The Whitney forms on a simplex $T$ admit high-order generalizations that have received a great deal of attention in numerical analysis.  Less well-known are the \emph{shadow forms} of Brasselet, Goresky, and MacPherson. These forms generalize the Whitney forms, but have rational coefficients, allowing singularities near the faces of $T$. Motivated by numerical problems that exhibit these kinds of singularities, we introduce degrees of freedom for the shadow $k$-forms that are well-suited for finite element implementations. In particular, we show that the degrees of freedom for the shadow forms are given by integration over the $k$-dimensional faces of the \emph{blow-up} $\tilde T$ of the simplex $T$. Consequently, we obtain an isomorphism between the cohomology of the complex of shadow forms and the cellular cohomology of $\tilde T$, which vanishes except in degree zero. Additionally, we discover a surprising probabilistic interpretation of shadow forms in terms of Poisson processes. This perspective simplifies several proofs and gives a way of computing bases for the shadow forms using a straightforward combinatorial calculation.  
\end{abstract}

\section{Introduction}

Since their introduction in the 1950s, the Whitney forms~\cite{whitney1957geometric} have had widespread impact on geometry, topology, and computational mathematics~\cite{dodziuk1978riemannian,muller1978analytic,bossavit1988whitney}.  These differential $k$-forms are piecewise affine forms on a simplicial triangulation and are in duality with integration over $k$-chains. As such, Whitney forms serve as the simplest example of a finite element space of differential forms arising in finite element exterior calculus~\cite{arnold2006finite,arnold2010finite}, a framework for analyzing finite element methods for partial differential equations.  They admit piecewise polynomial generalizations~\cite{raviart2006mixed,nedelec1980mixed,hiptmair2001higher,rapetti2009whitney} that are known to be well-suited for discretizing the Hodge Laplacian~\cite{arnold2006finite,arnold2010finite}. 

In this paper, we study another generalization of the Whitney forms introduced by Brasselet, Goresky, and MacPherson~\cite{brasselet1991simplicial}: the so-called \emph{shadow forms}.  These differential $k$-forms on an $n$-simplex $T$ have rational, as opposed to polynomial, coefficients, and they exhibit singular behavior near the boundary of $T$.  Taken together, the spaces of shadow $k$-forms of degree $k=0,1,\dots,n$ form an exact sequence that contains the classical Whitney forms as a subcomplex.  As we argue below, the shadow forms appear to be well-suited for constructing certain novel finite element spaces, like tangentially- and normally-continuous vector fields on triangulated surfaces.

We also remark here that, although the shadow forms are integrable, not all of them are square-integrable, so the complex of shadow forms is a \emph{Banach} complex, not a Hilbert complex. While analytic results are outside the scope of this paper, we feel that it is important to note that the problem of $L^p$ cohomology for $p\neq2$ received attention in Brasselet, Goresky, and MacPherson's paper \cite{brasselet1991simplicial}, as well as more recently; see for example \cite{guillaume2021}. It is likely that these analytic tools will be key to the eventual convergence analysis for these new finite element spaces. Perhaps they may even yield more direct proofs of existing finite element exterior calculus theorems that currently rely on smoothing, such as the bounded projection operators \cite{arnold2006finite}.

For reasons that we explain below, we will refer to the shadow $k$-forms as \emph{blow-up Whitney $k$-forms} in this paper.  This change in nomenclature highlights an important perspective we adopt when constructing degrees of freedom for this space and when proving exactness.  Following the notational conventions of~\cite{arnold2006finite,arnold2010finite}, we denote the space of blow-up Whitney $k$-forms on $T$ by $b\mathcal{P}_1^-\Lambda^k(T)$ to emphasize that it is a superspace of $\mathcal{P}_1^-\Lambda^k(T)$, the space of classical Whitney $k$-forms on $T$.

We make three main contributions.  First, we introduce a dual basis for the blow-up Whitney $k$-forms that is well-suited for finite element implementations.  The members of this dual basis are given by integration over $k$-dimensional faces of the \emph{blow-up} of $T$.  Briefly, the blow-up of $T$ is the manifold obtained from $T$ by blowing up its subsimplices in the sense of~\cite{gr01,melrose1996differential}, which is similar to the notion of blow-up in algebraic geometry.  For example, in two dimensions, we blow up the vertices of a triangle $T$ to obtain its blow-up $\tilde T$; the result is illustrated in Figure~\ref{fig:blowuptri}: It has 6 faces of dimension 1 that one can loosely think of as the 3 edges of $T$ and 3 infinitesimal arcs at the vertices of $T$. Meanwhile, in three dimensions, to obtain the blow-up $\widetilde{T}$ of a tetrahedron $T$, we first blow up its vertices, and then we blow up its edges; the result is illustrated in Figure~\ref{fig:blowuptet}.

Our second main contribution is to relate the definition of the blow-up Whitney forms, which is classically given by an integral formula, to a certain probability associated with Poisson processes.  This link with Poisson processes allows us to do several things.  It allows us to write down explicit formulas for the blow-up Whitney $k$-forms in any dimension $n$ using a straightforward combinatorial calculation.  (See~\cite{bendiffalah1995shadow} for a different explicit formula involving derivatives of a rational function of barycentric coordinates.)  It also leads to a simple proof that the classical Whitney $k$-forms are contained in the space of blow-up Whitney $k$-forms.  And it underpins our proof that the exterior derivative sends the blow-up Whitney $k$-forms to blow-up Whitney $(k+1)$-forms.  Although the latter two results are classical, we include proofs to highlight the utility of the Poisson process perspective. 

Our third main contribution is to give a new proof that the sequence of blow-up Whitney forms on a simplex $T$ is exact using the blow-up construction discussed above.  Our dual basis plays a key role in this proof, allowing us to relate the cohomology of the complex of blow-up Whitney forms to the cellular cohomology of the blow-up of $T$.

We believe that the blow-up construction and the Poisson process perspective will end up playing a role in developing higher-order generalizations of the blow-up Whitney forms. The blow-up construction suggests a natural choice for degrees of freedom in the higher-order setting: one considers higher-order moments of the form over faces of the blow-up of $T$.  Meanwhile, the Poisson process perspective provides ways of generating basis forms of higher degree.  In the lowest-order setting considered here, one generates basis forms by asking for the probability of observing a certain sequence of arrival times for particles produced by an ensemble of Poisson processes whose rates are related to the barycentric coordinates of a point in $T$.  We expect that this construction admits a generalization with more particles that would generate higher-order basis forms. At the end of this paper, we present some preliminary results on the higher-order spaces for blow-up scalar fields, but the full story of higher-order blow-up $k$-forms is still to be realized.

\subsection*{Motivation.} Our motivation for developing the blow-up forms comes from two sources.  First, there is growing evidence that tensor fields that exhibit singular behavior play an important role in the study of triangulated manifolds equipped with nonsmooth Riemannian metrics.  Specifically, when the metric is piecewise smooth and possesses single-valued tangential-tangential components along codimension-1 faces, various curvatures like the scalar curvature and the Riemann curvature tensor can be given meaning in distributional sense, often on the basis of heuristic arguments~\cite{strichartz2020defining,berchenko2023finite,gopalakrishnan2022analysis,gawlik2023finite,gawlik2023finite2,gopalakrishnan2023analysis}.  (See, however,~\cite{christiansen2024definition} for a systematic treatment of distributional scalar curvature in the piecewise flat setting.)  It turns out that one can arrive at these definitions in a systematic way (in any dimension, piecewise flat or not) by considering the action of the squared covariant exterior derivative, interpreted in a distributional sense, on vector fields with appropriate regularity.  With the right regularity hypotheses, one can perform an integration by parts calculation involving the covariant exterior derivative to arrive at the definition of distributional Riemann curvature proposed in~\cite{gopalakrishnan2023analysis} (various traces of which yield other distributional curvatures).  This integration by parts calculation takes place on the blow-up $\widetilde{T}$ of each $n$-simplex $T$, leading to integrals over components of $\partial\widetilde{T}$, just like the degrees of freedom for the blow-up forms.  The details behind this calculation are a subject of our ongoing work.

\subsection*{More motivation.} Our second source of motivation comes from a desire to construct a discretization of the vector Laplacian on triangulated surfaces that overcomes the following well-known difficulty: any piecewise smooth vector field tangent to a triangulated surface that has single-valued tangential and normal components along edges must vanish at the vertices, except in degenerate situations, like when the triangles meeting at a given vertex lie in a common plane~\cite{demlow2023tangential}.  This peculiarity is a well-known obstruction to constructing a conforming finite element method for the surface vector Laplacian, a differential operator that has received considerable attention in numerical analysis~\cite{demlow2023tangential,bonito2020divergence,olshanskii2022tangential,jankuhn2021error,lederer2020divergence,reusken2022analysis,hansbo2020analysis} owing to its role in the dynamics of fluids on surfaces~\cite{azencot2015functional}.   This operator differs from the Hodge Laplacian on one-forms, which can be discretized with standard spaces from finite element exterior calculus~\cite{arnold2006finite,arnold2010finite}. The obstruction described above is intimately related to the fact that the angles incident at a vertex $P$ on a triangulated surface generally do not sum to $2\pi$; instead, their sum generally deviates from $2\pi$ by some ``angle defect'' $\Theta \neq 0$.  Hence, any vector field $v$ with the aforementioned properties must experience a rotation by $\Theta$ under parallel transport around $P$, and this contradicts the continuity of the vector field unless it vanishes at $P$.  A seldom-used idea for sidestepping this obstruction is to relax the piecewise smoothness constraint on $v$: within each triangle incident at $P$, we allow $v$ to rapidly rotate in the vicinity of $P$, so that its value ``at'' $P$ differs along different rays emanating from $P$.  In other words, in polar coordinates $(r,\theta)$ centered at $P$, we allow $v$ to vary with $\theta$, even at $r=0$.  

\begin{figure}
  \centering
  \includegraphics[scale=0.4,clip=true,trim=0in 0in 2.2in 0in]{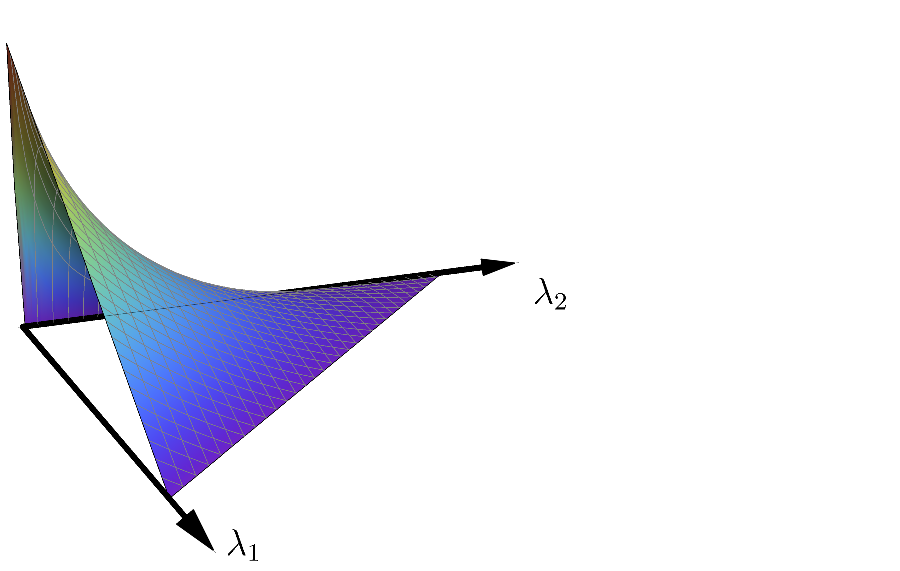}
  \includegraphics[scale=0.4,clip=true,trim=0in 0in 2.2in 0in]{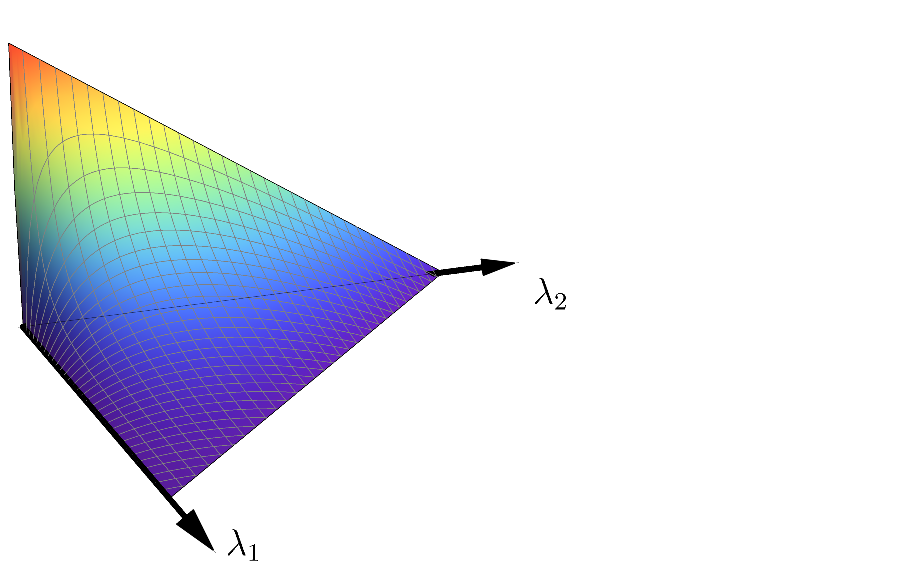}
  \includegraphics[scale=0.4,clip=true,trim=0in 0in 2.2in 0in]{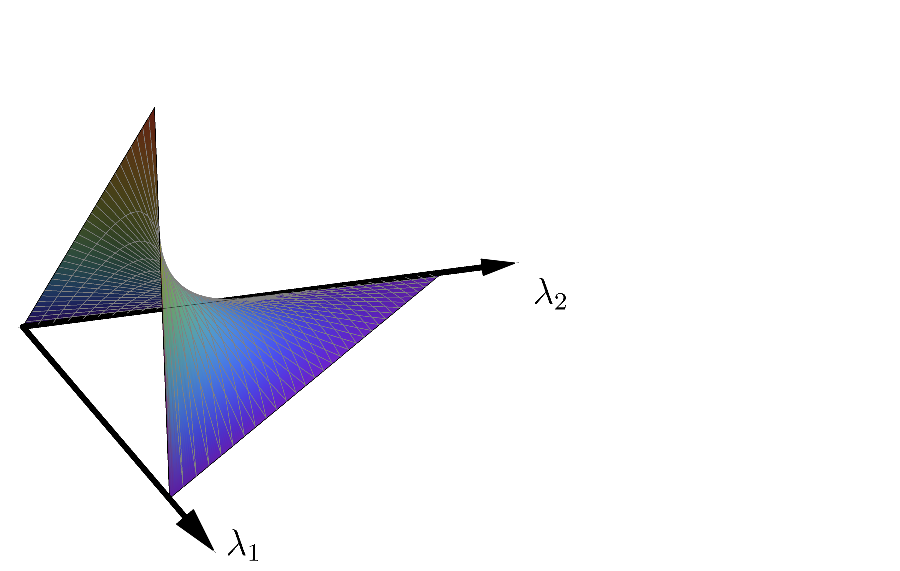} \\
  \includegraphics[scale=0.4,clip=true,trim=0in 0in 2.2in 0in]{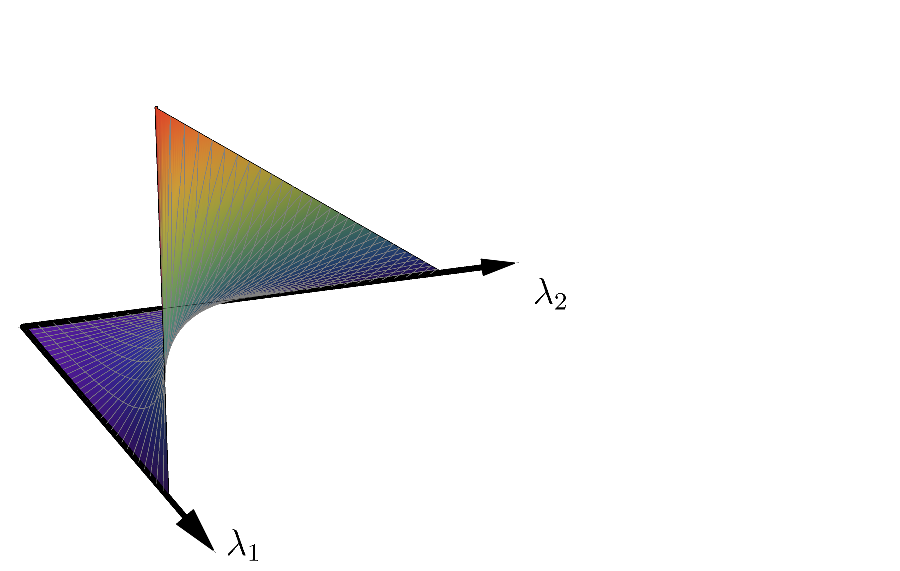}
  \includegraphics[scale=0.4,clip=true,trim=0in 0in 2.2in 0in]{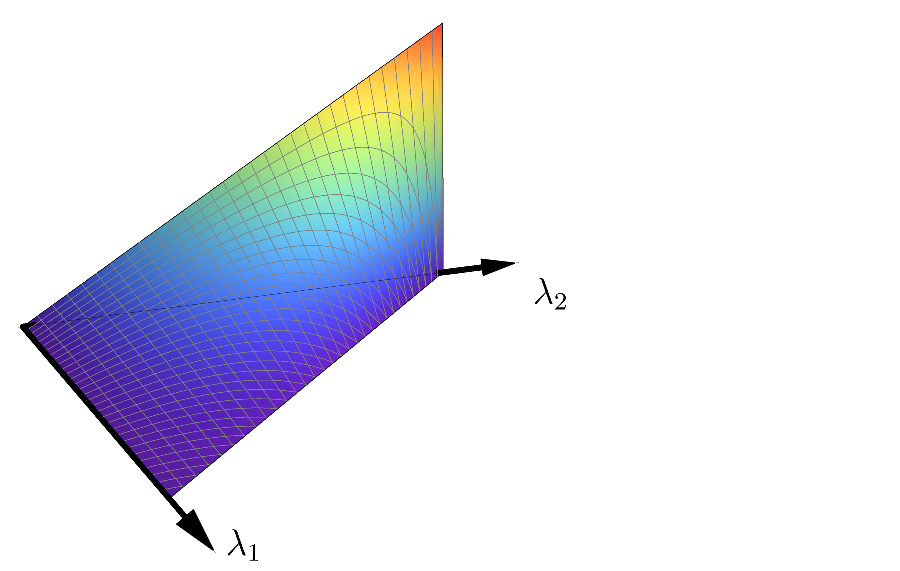}
  \includegraphics[scale=0.4,clip=true,trim=0in 0in 2.2in 0in]{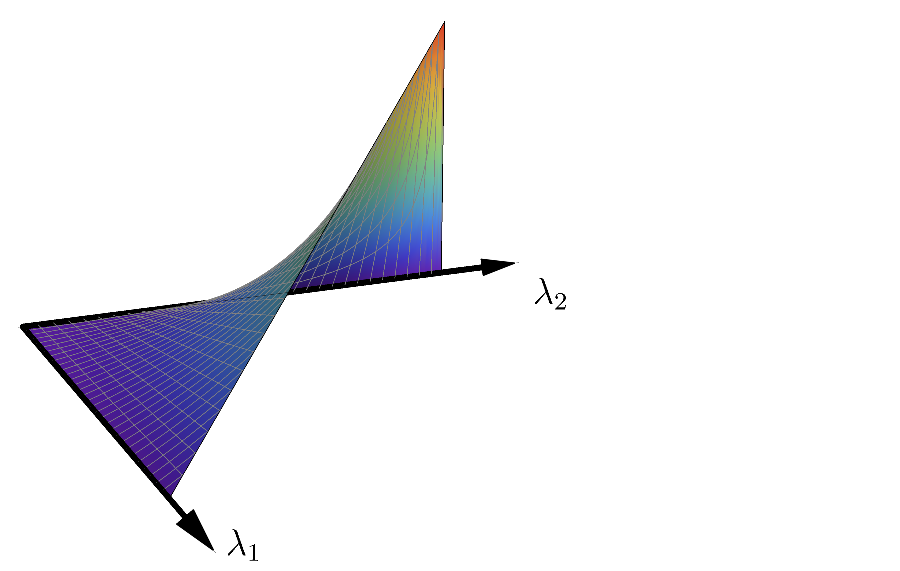} \\
  \caption{The 6 blow-up Whitney 0-forms in dimension $n=2$.  Top row: $\psi_{012},\psi_{021},\psi_{102}$.  Bottom row: $\psi_{120},\psi_{201},\psi_{210}$.}
  \label{fig:plot0form}
  \includegraphics[scale=0.5,clip=true,trim=0in 0in 0in 0in]{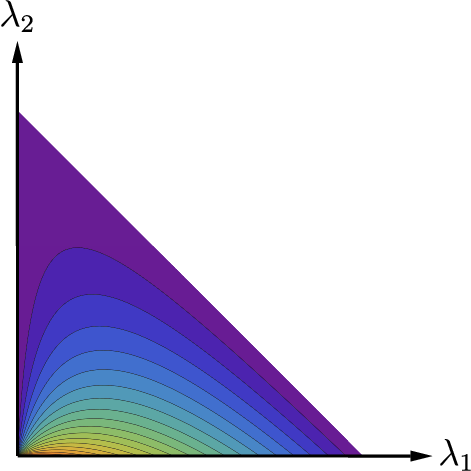}
  \includegraphics[scale=0.5,clip=true,trim=0in 0in 0in 0in]{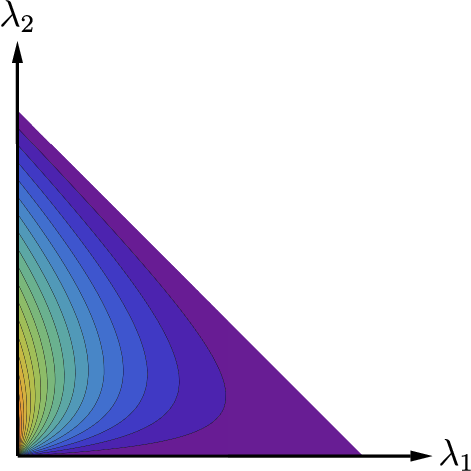}
  \includegraphics[scale=0.5,clip=true,trim=0in 0in 0in 0in]{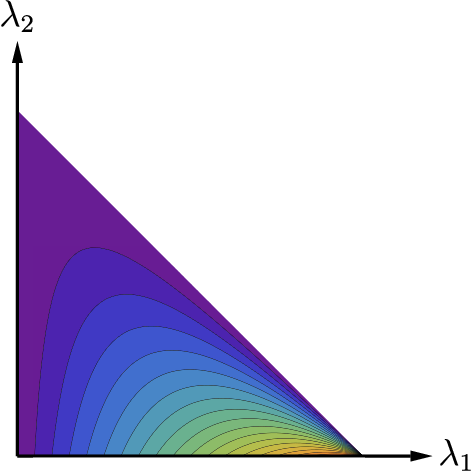} \\
  \includegraphics[scale=0.5,clip=true,trim=0in 0in 0in 0in]{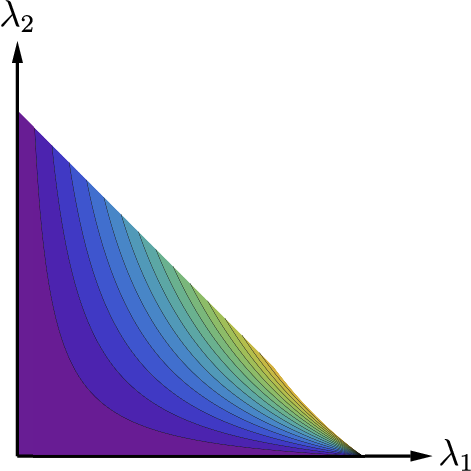}
  \includegraphics[scale=0.5,clip=true,trim=0in 0in 0in 0in]{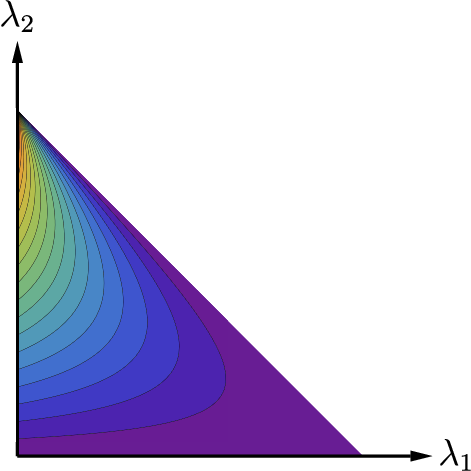}
  \includegraphics[scale=0.5,clip=true,trim=0in 0in 0in 0in]{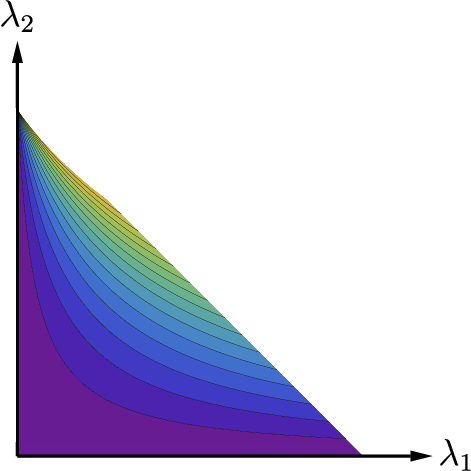}
  \caption{Contour plots of the same blow-up Whitney $0$-forms as above.}
\end{figure}

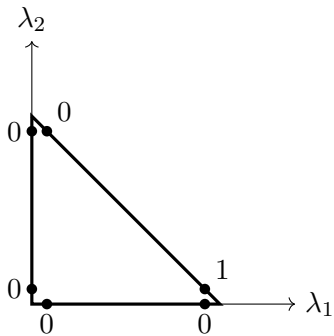
\begin{figure}
  \begin{center}
    \begin{tikzpicture}
      \draw[->] (0,0) -- (xyz cs:x=3.5) node[right] {$\lambda_1$};
      \draw[->] (0,0) -- (xyz cs:y=3.5) node[above] {$\lambda_2$};
      \draw[very thick] (0,0) node[label=$$]{$$}
      -- (2.5,0) node[anchor=north]{$$}
      -- (0,2.5) node[anchor=south]{$$}
      -- cycle;
      \fill[black] (0.2,0) circle (2pt) node[below] {0};
      \fill[black] (2.3,0) circle (2pt) node[below] {0}; 
      \fill[black] (2.3,0.2) circle (2pt) node[above right] {1};
      \fill[black] (0.2,2.3) circle (2pt) node[above right] {0}; 
      \fill[black] (0,2.3) circle (2pt) node[left] {0};
      \fill[black] (0,0.2) circle (2pt) node[left] {0}; 
    \end{tikzpicture}
  \end{center}
  \caption{Degrees of freedom for the function $\psi_{120} = \frac{\lambda_1 \lambda_2}{\lambda_2+\lambda_0} \in b\mathcal{P}_1^-\Lambda^0(T)$.  Each labelled point represents the limiting value of the function as one approaches the indicated vertex along the indicated edge.}
  \label{fig:dof}
\end{figure}

The intuition above motivates the following idea.  Let $T$ be a triangle with barycentric coordinates $\lambda_0,\lambda_1,\lambda_2$.  Consider the functions
\begin{align*}
  \psi_{012} &= \frac{\lambda_0 \lambda_1}{\lambda_1+\lambda_2}, & \psi_{021} &= \frac{\lambda_0 \lambda_2}{\lambda_2+\lambda_1}, & \psi_{102} &= \frac{\lambda_1 \lambda_0}{\lambda_0+\lambda_2}, \\ \psi_{120} &= \frac{\lambda_1 \lambda_2}{\lambda_2+\lambda_0}, & \psi_{201} &= \frac{\lambda_2 \lambda_0}{\lambda_0+\lambda_1}, & \psi_{210} &= \frac{\lambda_2 \lambda_1}{\lambda_1+\lambda_0}, 
\end{align*}
which are plotted in Figure~\ref{fig:plot0form}.  Each of these functions ``rapidly varies'' in the vicinity of one vertex.  For example, notice that the function $\psi_{120}$ vanishes on the edge $\lambda_1=0$, vanishes on the edge $\lambda_2=0$, and varies linearly along the edge $\lambda_0=0$: $\psi_{120}\big|_{\lambda_0=0} = \lambda_1$.  In particular, it is multivalued at vertex 1, in the sense that the limit of $\psi_{120}$ as we approach vertex 1 along the edge $\lambda_2=0$ is 0, but its limit as we approach vertex 1 along the edge $\lambda_0=0$ is 1.  In this sense, $\psi_{120}$ ``rapidly varies'' in the vicinity of vertex 1.  We think of $\psi_{120}$ as a basis function associated with a specific endpoint of a specific edge, namely vertex 1 of edge $(1,2)$ ($\lambda_0=0$); this is the location where $\psi_{120}$ evaluates to 1 (in a limiting sense).

The space spanned by the $\psi_{ijk}$'s is the space $b\mathcal{P}_1^-\Lambda^0(T)$ of blow-up Whitney 0-forms in dimension $n=2$. Note that it contains all affine functions, since 
\begin{equation} \label{affine}
  \psi_{012} + \psi_{021} = \lambda_0,\qquad \psi_{120}+\psi_{102} = \lambda_1,\qquad \psi_{201}+\psi_{210}=\lambda_2.
\end{equation}
Thus, $b\mathcal{P}_1^-\Lambda^0(T)$ is a superset of $\mathcal{P}_1^-\Lambda^0(T)$, the space of classical Whitney 0-forms.  The $b$ stands for ``blow-up'' and signals that these functions are not smooth on $T$, but rather they are smooth on the manifold obtained from $T$ by blowing up the vertices as we alluded to above.

To summarize: $b\mathcal{P}_1^-\Lambda^0(T)$ is a richer space than $\mathcal{P}_1^-\Lambda^0(T)$ that accommodates rapid variations near the vertices.  By considering a vector-valued version of this space, we obtain a space of vector fields that accommodates rapid rotations near the vertices, and by gluing these local spaces together appropriately on a triangulated surface $\mathcal{T}$, we obtain a global space that admits nontrivial vector fields that have single-valued normal and tangential components along edges.  Our numerical experiments suggest that this space is well-suited for discretizing the vector Laplacian on surfaces and, more generally, the Bochner Laplacian on 2-manifolds.  In particular, we observe in numerical experiments that it can be used to compute the spectrum of the Bochner Laplacian with accuracy $O(h^2)$ on triangulations with maximum element diameter $h$.  We are pursuing this in a separate paper.

A key feature of the discretization described above is that it is intrinsic, in the sense that it can be implemented without specifying coordinates for the vertices of the triangulation; only edge lengths and connectivity information are needed.  In contrast, the finite element discretizations of the surface vector Laplacian that we are aware of, such as the ones studied in in~\cite{demlow2023tangential,bonito2020divergence,olshanskii2022tangential,jankuhn2021error,lederer2020divergence,reusken2022analysis,hansbo2020analysis}, are extrinsic. That is, the discretizations rely on vertex coordinates and generally change when we move the vertices without changing the edge lengths.  This is undesirable since the surface vector Laplacian---and more generally, the Bochner Laplacian---is an intrinsic object; it does not change under isometric deformations.

The reader may notice that the blow-up Whitney forms have some commonalities with the extended finite element method (XFEM) and other enriched finite element methods~\cite{moes1999finite}, which accommodate singular solutions to partial differential equations by enriching classical finite element spaces with singular functions.  There are a few differences though.  In the model problem above involving the surface vector Laplacian, the exact solution to the underlying smooth problem is not singular; only the discrete solution is.  Moreover, the locations of the discrete solution's singularities are mesh-dependent.   In a typical application of XFEM, such as the simulation of fracture propagation, the exact solution is singular, and the location of the singularity is independent of the mesh.  This is not to say that blow-up Whitney forms could not be used for such applications; we merely wish to emphasize that our goals are rather different.  It is also worth noting that enriched finite element spaces for differential forms of degree $k>0$ are not typically considered in the XFEM literature.

\subsection*{Gluing degrees of freedom.}  The way we choose to glue the local spaces together has an impact on the global space we obtain, so let us discuss this in more depth, focusing on the (scalar-valued) $b\mathcal{P}_1^-\Lambda^0(T)$ space for simplicity. As illustrated in Figure~\ref{fig:dof}, within a single element $T$, each degree of freedom is associated to a vertex-edge pair. So, globally, prior to gluing, we have a degree of freedom for each nested triple $P<E<T$ of a vertex $P$, edge $E$, and face $T$ of the triangulation; we call such a triple a \emph{flag}. We discuss four natural choices of gluings, which are illustrated in Figure~\ref{fig:glue}.

The gluing alluded to above is obtained by equating degrees of freedom with matching vertex-edge pairs. That is, we glue together the degrees of freedom corresponding to $P<E<T_1$ and $P<E<T_2$, where $T_1$ and $T_2$ are the two elements containing edge $E$. However, importantly, we do not equate degrees of freedom associated with flags $P<E_1<T$ and $P<E_2<T$ if $E_1 \neq E_2$.  This leads to a global finite element space whose members are single-valued along edges but multi-valued at vertices. We could even take this one step further and equate all degrees of freedom associated to a particular edge. That is, for each edge $E$, we glue $P_1<E<T_1$, $P_1<E<T_2$, $P_2<E<T_1$, and $P_2<E<T_2$, where $P_1$ and $P_2$ are the two endpoints of the edge $E$, and $T_1$ and $T_2$ are the two elements containing $E$.  Then we obtain a global finite element space whose members are single-valued \emph{and constant} along edges but multi-valued at vertices.  These two spaces have some resemblance to the Crouzeix-Raviart finite element space~\cite{crouzeix1973conforming}, whose members are single-valued at the midpoint of each edge.   

Consider now the following alternative gluing: we equate degrees of freedom $P_1<E_1<T_1$ and $P_2<E_2<T_2$ whenever $P_1=P_2$. Then our global space reduces to the piecewise affine Lagrange finite element space $\mathcal{P}_1^-\Lambda^0(\mathcal{T})$.  This can be understood with the help of~\eqref{affine}.  Yet another alternative is to equate degrees of freedom $P_1<E_1<T_1$ and $P_2<E_2<T_2$ whenever both $P_1=P_2$ and $T_1=T_2$.  Then our global space reduces to the space of discontinuous piecewise affine functions.

\begin{figure}
  \centering

  \begin{tikzpicture}[scale=0.45]
    \begin{pgfonlayer}{nodelayer}
      \node [style=none] (0) at (-3, 5) {};
      \node [style=none] (1) at (0, 1) {};
      \node [style=none] (2) at (3, 5) {};
      \node [style=none] (3) at (-4, 4) {};
      \node [style=none] (4) at (-1, 0) {};
      \node [style=none] (5) at (-6, 0) {};
      \node [style=none] (6) at (-6, -1) {};
      \node [style=none] (7) at (-1, -1) {};
      \node [style=none] (8) at (-4, -5) {};
      \node [style=none] (9) at (1, 0) {};
      \node [style=none] (10) at (4, 4) {};
      \node [style=none] (11) at (6, 0) {};
      \node [style=none] (12) at (1, -1) {};
      \node [style=none] (13) at (6, -1) {};
      \node [style=none] (14) at (4, -5) {};
      \node [style=none] (15) at (0, -2) {};
      \node [style=none] (16) at (-3, -6) {};
      \node [style=none] (17) at (3, -6) {};
      \node [circle,fill=black,inner sep=2pt] (18) at (-2, 5) {};
      \node [circle,fill=black,inner sep=2pt] (19) at (-2.25, 4) {};
      \node [circle,fill=black,inner sep=2pt] (20) at (2, 5) {};
      \node [circle,fill=black,inner sep=2pt] (21) at (2.25, 4) {};
      \node [circle,fill=black,inner sep=2pt] (22) at (0.75, 2) {};
      \node [circle,fill=black,inner sep=2pt] (23) at (-0.75, 2) {};
      \node [circle,fill=black,inner sep=2pt] (24) at (-1.75, 1) {};
      \node [circle,fill=black,inner sep=2pt] (25) at (-2, 0) {};
      \node [circle,fill=black,inner sep=2pt] (26) at (-5.5, 1) {};
      \node [circle,fill=black,inner sep=2pt] (27) at (-5, 0) {};
      \node [circle,fill=black,inner sep=2pt] (28) at (-5, -1) {};
      \node [circle,fill=black,inner sep=2pt] (29) at (-5.5, -2) {};
      \node [circle,fill=black,inner sep=2pt] (30) at (-4.5, -4) {};
      \node [circle,fill=black,inner sep=2pt] (31) at (-3.25, -4) {};
      \node [circle,fill=black,inner sep=2pt] (32) at (-2, -1) {};
      \node [circle,fill=black,inner sep=2pt] (33) at (-1.75, -2) {};
      \node [circle,fill=black,inner sep=2pt] (34) at (-0.75, -3) {};
      \node [circle,fill=black,inner sep=2pt] (35) at (0.75, -3) {};
      \node [circle,fill=black,inner sep=2pt] (36) at (-2.25, -5) {};
      \node [circle,fill=black,inner sep=2pt] (37) at (-2, -6) {};
      \node [circle,fill=black,inner sep=2pt] (38) at (2, -6) {};
      \node [circle,fill=black,inner sep=2pt] (39) at (2.25, -5) {};
      \node [circle,fill=black,inner sep=2pt] (40) at (1.75, -2) {};
      \node [circle,fill=black,inner sep=2pt] (41) at (2, -1) {};
      \node [circle,fill=black,inner sep=2pt] (42) at (3.25, -4) {};
      \node [circle,fill=black,inner sep=2pt] (43) at (4.5, -4) {};
      \node [circle,fill=black,inner sep=2pt] (44) at (5, -1) {};
      \node [circle,fill=black,inner sep=2pt] (45) at (5.5, -2) {};
      \node [circle,fill=black,inner sep=2pt] (46) at (5, 0) {};
      \node [circle,fill=black,inner sep=2pt] (47) at (5.5, 1) {};
      \node [circle,fill=black,inner sep=2pt] (48) at (3.25, 3) {};
      \node [circle,fill=black,inner sep=2pt] (49) at (4.5, 3) {};
      \node [circle,fill=black,inner sep=2pt] (50) at (1.75, 1) {};
      \node [circle,fill=black,inner sep=2pt] (51) at (2, 0) {};
      \node [circle,fill=black,inner sep=2pt] (52) at (-4.5, 3) {};
      \node [circle,fill=black,inner sep=2pt] (53) at (-3.25, 3) {};
    \end{pgfonlayer}
    \begin{pgfonlayer}{edgelayer}
      \draw[very thick] (0.center) to (2.center);
      \draw[very thick] (1.center) to (0.center);
      \draw[very thick] (1.center) to (2.center);
      \draw[very thick] (3.center) to (4.center);
      \draw[very thick] (5.center) to (4.center);
      \draw[very thick] (5.center) to (3.center);
      \draw[very thick] (6.center) to (7.center);
      \draw[very thick] (6.center) to (8.center);
      \draw[very thick] (8.center) to (7.center);
      \draw[very thick] (16.center) to (15.center);
      \draw[very thick] (16.center) to (17.center);
      \draw[very thick] (17.center) to (15.center);
      \draw[very thick] (12.center) to (14.center);
      \draw[very thick] (14.center) to (13.center);
      \draw[very thick] (12.center) to (13.center);
      \draw[very thick] (9.center) to (10.center);
      \draw[very thick] (9.center) to (11.center);
      \draw[very thick] (11.center) to (10.center);
      \draw[very thick,dotted] (19.center) to (53.center);
      \draw[very thick,dotted] (24.center) to (23.center);
      \draw[very thick,dotted] (25.center) to (32.center);
      \draw[very thick,dotted] (27.center) to (28.center);
      \draw[very thick,dotted] (31.center) to (36.center);
      \draw[very thick,dotted] (33.center) to (34.center);
      \draw[very thick,dotted] (35.center) to (40.center);
      \draw[very thick,dotted] (39.center) to (42.center);
      \draw[very thick,dotted] (44.center) to (46.center);
      \draw[very thick,dotted] (41.center) to (51.center);
      \draw[very thick,dotted] (22.center) to (50.center);
      \draw[very thick,dotted] (21.center) to (48.center);
    \end{pgfonlayer}
  \end{tikzpicture} \hspace{0.2in}
  \begin{tikzpicture}[scale=0.45]
    \begin{pgfonlayer}{nodelayer}
      \node [style=none] (0) at (-3, 5) {};
      \node [style=none] (1) at (0, 1) {};
      \node [style=none] (2) at (3, 5) {};
      \node [style=none] (3) at (-4, 4) {};
      \node [style=none] (4) at (-1, 0) {};
      \node [style=none] (5) at (-6, 0) {};
      \node [style=none] (6) at (-6, -1) {};
      \node [style=none] (7) at (-1, -1) {};
      \node [style=none] (8) at (-4, -5) {};
      \node [style=none] (9) at (1, 0) {};
      \node [style=none] (10) at (4, 4) {};
      \node [style=none] (11) at (6, 0) {};
      \node [style=none] (12) at (1, -1) {};
      \node [style=none] (13) at (6, -1) {};
      \node [style=none] (14) at (4, -5) {};
      \node [style=none] (15) at (0, -2) {};
      \node [style=none] (16) at (-3, -6) {};
      \node [style=none] (17) at (3, -6) {};
      \node [circle,fill=black,inner sep=2pt] (18) at (-2, 5) {};
      \node [circle,fill=black,inner sep=2pt] (19) at (-2.25, 4) {};
      \node [circle,fill=black,inner sep=2pt] (20) at (2, 5) {};
      \node [circle,fill=black,inner sep=2pt] (21) at (2.25, 4) {};
      \node [circle,fill=black,inner sep=2pt] (22) at (0.75, 2) {};
      \node [circle,fill=black,inner sep=2pt] (23) at (-0.75, 2) {};
      \node [circle,fill=black,inner sep=2pt] (24) at (-1.75, 1) {};
      \node [circle,fill=black,inner sep=2pt] (25) at (-2, 0) {};
      \node [circle,fill=black,inner sep=2pt] (26) at (-5.5, 1) {};
      \node [circle,fill=black,inner sep=2pt] (27) at (-5, 0) {};
      \node [circle,fill=black,inner sep=2pt] (28) at (-5, -1) {};
      \node [circle,fill=black,inner sep=2pt] (29) at (-5.5, -2) {};
      \node [circle,fill=black,inner sep=2pt] (30) at (-4.5, -4) {};
      \node [circle,fill=black,inner sep=2pt] (31) at (-3.25, -4) {};
      \node [circle,fill=black,inner sep=2pt] (32) at (-2, -1) {};
      \node [circle,fill=black,inner sep=2pt] (33) at (-1.75, -2) {};
      \node [circle,fill=black,inner sep=2pt] (34) at (-0.75, -3) {};
      \node [circle,fill=black,inner sep=2pt] (35) at (0.75, -3) {};
      \node [circle,fill=black,inner sep=2pt] (36) at (-2.25, -5) {};
      \node [circle,fill=black,inner sep=2pt] (37) at (-2, -6) {};
      \node [circle,fill=black,inner sep=2pt] (38) at (2, -6) {};
      \node [circle,fill=black,inner sep=2pt] (39) at (2.25, -5) {};
      \node [circle,fill=black,inner sep=2pt] (40) at (1.75, -2) {};
      \node [circle,fill=black,inner sep=2pt] (41) at (2, -1) {};
      \node [circle,fill=black,inner sep=2pt] (42) at (3.25, -4) {};
      \node [circle,fill=black,inner sep=2pt] (43) at (4.5, -4) {};
      \node [circle,fill=black,inner sep=2pt] (44) at (5, -1) {};
      \node [circle,fill=black,inner sep=2pt] (45) at (5.5, -2) {};
      \node [circle,fill=black,inner sep=2pt] (46) at (5, 0) {};
      \node [circle,fill=black,inner sep=2pt] (47) at (5.5, 1) {};
      \node [circle,fill=black,inner sep=2pt] (48) at (3.25, 3) {};
      \node [circle,fill=black,inner sep=2pt] (49) at (4.5, 3) {};
      \node [circle,fill=black,inner sep=2pt] (50) at (1.75, 1) {};
      \node [circle,fill=black,inner sep=2pt] (51) at (2, 0) {};
      \node [circle,fill=black,inner sep=2pt] (52) at (-4.5, 3) {};
      \node [circle,fill=black,inner sep=2pt] (53) at (-3.25, 3) {};
      \node [style=none] (54) at (-2.75, 3.5) {};
      \node [style=none] (55) at (-1.25, 1.5) {};
      \node [style=none] (56) at (1.25, 1.5) {};
      \node [style=none] (57) at (2.75, 3.5) {};
      \node [style=none] (58) at (2, -0.5) {};
      \node [style=none] (59) at (5, -0.5) {};
      \node [style=none] (60) at (-5, -0.5) {};
      \node [style=none] (61) at (-2, -0.5) {};
      \node [style=none] (62) at (-2.75, -4.5) {};
      \node [style=none] (63) at (-1.25, -2.5) {};
      \node [style=none] (64) at (1.25, -2.5) {};
      \node [style=none] (65) at (2.75, -4.5) {};
    \end{pgfonlayer}
    \begin{pgfonlayer}{edgelayer}
      \draw[very thick] (0.center) to (2.center);
      \draw[very thick] (1.center) to (0.center);
      \draw[very thick] (1.center) to (2.center);
      \draw[very thick] (3.center) to (4.center);
      \draw[very thick] (5.center) to (4.center);
      \draw[very thick] (5.center) to (3.center);
      \draw[very thick] (6.center) to (7.center);
      \draw[very thick] (6.center) to (8.center);
      \draw[very thick] (8.center) to (7.center);
      \draw[very thick] (16.center) to (15.center);
      \draw[very thick] (16.center) to (17.center);
      \draw[very thick] (17.center) to (15.center);
      \draw[very thick] (12.center) to (14.center);
      \draw[very thick] (14.center) to (13.center);
      \draw[very thick] (12.center) to (13.center);
      \draw[very thick] (9.center) to (10.center);
      \draw[very thick] (9.center) to (11.center);
      \draw[very thick] (11.center) to (10.center);
      \draw[very thick,dotted] (19.center) to (53.center);
      \draw[very thick,dotted] (24.center) to (23.center);
      \draw[very thick,dotted] (25.center) to (32.center);
      \draw[very thick,dotted] (27.center) to (28.center);
      \draw[very thick,dotted] (31.center) to (36.center);
      \draw[very thick,dotted] (33.center) to (34.center);
      \draw[very thick,dotted] (35.center) to (40.center);
      \draw[very thick,dotted] (39.center) to (42.center);
      \draw[very thick,dotted] (44.center) to (46.center);
      \draw[very thick,dotted] (41.center) to (51.center);
      \draw[very thick,dotted] (22.center) to (50.center);
      \draw[very thick,dotted] (21.center) to (48.center);
      \draw[very thick,dotted] (54.center) to (55.center);
      \draw[very thick,dotted] (56.center) to (57.center);
      \draw[very thick,dotted] (58.center) to (59.center);
      \draw[very thick,dotted] (64.center) to (65.center);
      \draw[very thick,dotted] (62.center) to (63.center);
      \draw[very thick,dotted] (60.center) to (61.center);
    \end{pgfonlayer}
  \end{tikzpicture} \\\vspace{0.2in}
  \begin{tikzpicture}[scale=0.45]
    \begin{pgfonlayer}{nodelayer}
      \node [style=none] (0) at (-3, 5) {};
      \node [style=none] (1) at (0, 1) {};
      \node [style=none] (2) at (3, 5) {};
      \node [style=none] (3) at (-4, 4) {};
      \node [style=none] (4) at (-1, 0) {};
      \node [style=none] (5) at (-6, 0) {};
      \node [style=none] (6) at (-6, -1) {};
      \node [style=none] (7) at (-1, -1) {};
      \node [style=none] (8) at (-4, -5) {};
      \node [style=none] (9) at (1, 0) {};
      \node [style=none] (10) at (4, 4) {};
      \node [style=none] (11) at (6, 0) {};
      \node [style=none] (12) at (1, -1) {};
      \node [style=none] (13) at (6, -1) {};
      \node [style=none] (14) at (4, -5) {};
      \node [style=none] (15) at (0, -2) {};
      \node [style=none] (16) at (-3, -6) {};
      \node [style=none] (17) at (3, -6) {};
      \node [circle,fill=black,inner sep=2pt] (18) at (-2, 5) {};
      \node [circle,fill=black,inner sep=2pt] (19) at (-2.25, 4) {};
      \node [circle,fill=black,inner sep=2pt] (20) at (2, 5) {};
      \node [circle,fill=black,inner sep=2pt] (21) at (2.25, 4) {};
      \node [circle,fill=black,inner sep=2pt] (22) at (0.75, 2) {};
      \node [circle,fill=black,inner sep=2pt] (23) at (-0.75, 2) {};
      \node [circle,fill=black,inner sep=2pt] (24) at (-1.75, 1) {};
      \node [circle,fill=black,inner sep=2pt] (25) at (-2, 0) {};
      \node [circle,fill=black,inner sep=2pt] (26) at (-5.5, 1) {};
      \node [circle,fill=black,inner sep=2pt] (27) at (-5, 0) {};
      \node [circle,fill=black,inner sep=2pt] (28) at (-5, -1) {};
      \node [circle,fill=black,inner sep=2pt] (29) at (-5.5, -2) {};
      \node [circle,fill=black,inner sep=2pt] (30) at (-4.5, -4) {};
      \node [circle,fill=black,inner sep=2pt] (31) at (-3.25, -4) {};
      \node [circle,fill=black,inner sep=2pt] (32) at (-2, -1) {};
      \node [circle,fill=black,inner sep=2pt] (33) at (-1.75, -2) {};
      \node [circle,fill=black,inner sep=2pt] (34) at (-0.75, -3) {};
      \node [circle,fill=black,inner sep=2pt] (35) at (0.75, -3) {};
      \node [circle,fill=black,inner sep=2pt] (36) at (-2.25, -5) {};
      \node [circle,fill=black,inner sep=2pt] (37) at (-2, -6) {};
      \node [circle,fill=black,inner sep=2pt] (38) at (2, -6) {};
      \node [circle,fill=black,inner sep=2pt] (39) at (2.25, -5) {};
      \node [circle,fill=black,inner sep=2pt] (40) at (1.75, -2) {};
      \node [circle,fill=black,inner sep=2pt] (41) at (2, -1) {};
      \node [circle,fill=black,inner sep=2pt] (42) at (3.25, -4) {};
      \node [circle,fill=black,inner sep=2pt] (43) at (4.5, -4) {};
      \node [circle,fill=black,inner sep=2pt] (44) at (5, -1) {};
      \node [circle,fill=black,inner sep=2pt] (45) at (5.5, -2) {};
      \node [circle,fill=black,inner sep=2pt] (46) at (5, 0) {};
      \node [circle,fill=black,inner sep=2pt] (47) at (5.5, 1) {};
      \node [circle,fill=black,inner sep=2pt] (48) at (3.25, 3) {};
      \node [circle,fill=black,inner sep=2pt] (49) at (4.5, 3) {};
      \node [circle,fill=black,inner sep=2pt] (50) at (1.75, 1) {};
      \node [circle,fill=black,inner sep=2pt] (51) at (2, 0) {};
      \node [circle,fill=black,inner sep=2pt] (52) at (-4.5, 3) {};
      \node [circle,fill=black,inner sep=2pt] (53) at (-3.25, 3) {};
    \end{pgfonlayer}
    \begin{pgfonlayer}{edgelayer}
      \draw[very thick] (0.center) to (2.center);
      \draw[very thick] (1.center) to (0.center);
      \draw[very thick] (1.center) to (2.center);
      \draw[very thick] (3.center) to (4.center);
      \draw[very thick] (5.center) to (4.center);
      \draw[very thick] (5.center) to (3.center);
      \draw[very thick] (6.center) to (7.center);
      \draw[very thick] (6.center) to (8.center);
      \draw[very thick] (8.center) to (7.center);
      \draw[very thick] (16.center) to (15.center);
      \draw[very thick] (16.center) to (17.center);
      \draw[very thick] (17.center) to (15.center);
      \draw[very thick] (12.center) to (14.center);
      \draw[very thick] (14.center) to (13.center);
      \draw[very thick] (12.center) to (13.center);
      \draw[very thick] (9.center) to (10.center);
      \draw[very thick] (9.center) to (11.center);
      \draw[very thick] (11.center) to (10.center);
      \draw[very thick,dotted] (18.center) to (19.center);
      \draw[very thick,dotted] (20.center) to (21.center);
      \draw[very thick,dotted] (23.center) to (22.center);
      \draw[very thick,dotted] (52.center) to (53.center);
      \draw[very thick,dotted] (26.center) to (27.center);
      \draw[very thick,dotted] (24.center) to (25.center);
      \draw[very thick,dotted] (28.center) to (29.center);
      \draw[very thick,dotted] (30.center) to (31.center);
      \draw[very thick,dotted] (32.center) to (33.center);
      \draw[very thick,dotted] (36.center) to (37.center);
      \draw[very thick,dotted] (34.center) to (35.center);
      \draw[very thick,dotted] (39.center) to (38.center);
      \draw[very thick,dotted] (42.center) to (43.center);
      \draw[very thick,dotted] (40.center) to (41.center);
      \draw[very thick,dotted] (44.center) to (45.center);
      \draw[very thick,dotted] (50.center) to (51.center);
      \draw[very thick,dotted] (48.center) to (49.center);
      \draw[very thick,dotted] (46.center) to (47.center);
      \draw[very thick,dotted] (53.center) to (19.center);
      \draw[very thick,dotted] (24.center) to (23.center);
      \draw[very thick,dotted] (22.center) to (50.center);
      \draw[very thick,dotted] (25.center) to (32.center);
      \draw[very thick,dotted] (33.center) to (34.center);
      \draw[very thick,dotted] (35.center) to (40.center);
      \draw[very thick,dotted] (51.center) to (41.center);
      \draw[very thick,dotted] (21.center) to (48.center);
      \draw[very thick,dotted] (46.center) to (44.center);
      \draw[very thick,dotted] (39.center) to (42.center);
      \draw[very thick,dotted] (36.center) to (31.center);
      \draw[very thick,dotted] (28.center) to (27.center);
    \end{pgfonlayer}
  \end{tikzpicture} \hspace{0.2in}
  \begin{tikzpicture}[scale=0.45]
    \begin{pgfonlayer}{nodelayer}
      \node [style=none] (0) at (-3, 5) {};
      \node [style=none] (1) at (0, 1) {};
      \node [style=none] (2) at (3, 5) {};
      \node [style=none] (3) at (-4, 4) {};
      \node [style=none] (4) at (-1, 0) {};
      \node [style=none] (5) at (-6, 0) {};
      \node [style=none] (6) at (-6, -1) {};
      \node [style=none] (7) at (-1, -1) {};
      \node [style=none] (8) at (-4, -5) {};
      \node [style=none] (9) at (1, 0) {};
      \node [style=none] (10) at (4, 4) {};
      \node [style=none] (11) at (6, 0) {};
      \node [style=none] (12) at (1, -1) {};
      \node [style=none] (13) at (6, -1) {};
      \node [style=none] (14) at (4, -5) {};
      \node [style=none] (15) at (0, -2) {};
      \node [style=none] (16) at (-3, -6) {};
      \node [style=none] (17) at (3, -6) {};
      \node [circle,fill=black,inner sep=2pt] (18) at (-2, 5) {};
      \node [circle,fill=black,inner sep=2pt] (19) at (-2.25, 4) {};
      \node [circle,fill=black,inner sep=2pt] (20) at (2, 5) {};
      \node [circle,fill=black,inner sep=2pt] (21) at (2.25, 4) {};
      \node [circle,fill=black,inner sep=2pt] (22) at (0.75, 2) {};
      \node [circle,fill=black,inner sep=2pt] (23) at (-0.75, 2) {};
      \node [circle,fill=black,inner sep=2pt] (24) at (-1.75, 1) {};
      \node [circle,fill=black,inner sep=2pt] (25) at (-2, 0) {};
      \node [circle,fill=black,inner sep=2pt] (26) at (-5.5, 1) {};
      \node [circle,fill=black,inner sep=2pt] (27) at (-5, 0) {};
      \node [circle,fill=black,inner sep=2pt] (28) at (-5, -1) {};
      \node [circle,fill=black,inner sep=2pt] (29) at (-5.5, -2) {};
      \node [circle,fill=black,inner sep=2pt] (30) at (-4.5, -4) {};
      \node [circle,fill=black,inner sep=2pt] (31) at (-3.25, -4) {};
      \node [circle,fill=black,inner sep=2pt] (32) at (-2, -1) {};
      \node [circle,fill=black,inner sep=2pt] (33) at (-1.75, -2) {};
      \node [circle,fill=black,inner sep=2pt] (34) at (-0.75, -3) {};
      \node [circle,fill=black,inner sep=2pt] (35) at (0.75, -3) {};
      \node [circle,fill=black,inner sep=2pt] (36) at (-2.25, -5) {};
      \node [circle,fill=black,inner sep=2pt] (37) at (-2, -6) {};
      \node [circle,fill=black,inner sep=2pt] (38) at (2, -6) {};
      \node [circle,fill=black,inner sep=2pt] (39) at (2.25, -5) {};
      \node [circle,fill=black,inner sep=2pt] (40) at (1.75, -2) {};
      \node [circle,fill=black,inner sep=2pt] (41) at (2, -1) {};
      \node [circle,fill=black,inner sep=2pt] (42) at (3.25, -4) {};
      \node [circle,fill=black,inner sep=2pt] (43) at (4.5, -4) {};
      \node [circle,fill=black,inner sep=2pt] (44) at (5, -1) {};
      \node [circle,fill=black,inner sep=2pt] (45) at (5.5, -2) {};
      \node [circle,fill=black,inner sep=2pt] (46) at (5, 0) {};
      \node [circle,fill=black,inner sep=2pt] (47) at (5.5, 1) {};
      \node [circle,fill=black,inner sep=2pt] (48) at (3.25, 3) {};
      \node [circle,fill=black,inner sep=2pt] (49) at (4.5, 3) {};
      \node [circle,fill=black,inner sep=2pt] (50) at (1.75, 1) {};
      \node [circle,fill=black,inner sep=2pt] (51) at (2, 0) {};
      \node [circle,fill=black,inner sep=2pt] (52) at (-4.5, 3) {};
      \node [circle,fill=black,inner sep=2pt] (53) at (-3.25, 3) {};
    \end{pgfonlayer}
    \begin{pgfonlayer}{edgelayer}
      \draw[very thick] (0.center) to (2.center);
      \draw[very thick] (1.center) to (0.center);
      \draw[very thick] (1.center) to (2.center);
      \draw[very thick] (3.center) to (4.center);
      \draw[very thick] (5.center) to (4.center);
      \draw[very thick] (5.center) to (3.center);
      \draw[very thick] (6.center) to (7.center);
      \draw[very thick] (6.center) to (8.center);
      \draw[very thick] (8.center) to (7.center);
      \draw[very thick] (16.center) to (15.center);
      \draw[very thick] (16.center) to (17.center);
      \draw[very thick] (17.center) to (15.center);
      \draw[very thick] (12.center) to (14.center);
      \draw[very thick] (14.center) to (13.center);
      \draw[very thick] (12.center) to (13.center);
      \draw[very thick] (9.center) to (10.center);
      \draw[very thick] (9.center) to (11.center);
      \draw[very thick] (11.center) to (10.center);
      \draw[very thick,dotted] (18.center) to (19.center);
      \draw[very thick,dotted] (20.center) to (21.center);
      \draw[very thick,dotted] (23.center) to (22.center);
      \draw[very thick,dotted] (52.center) to (53.center);
      \draw[very thick,dotted] (26.center) to (27.center);
      \draw[very thick,dotted] (24.center) to (25.center);
      \draw[very thick,dotted] (28.center) to (29.center);
      \draw[very thick,dotted] (30.center) to (31.center);
      \draw[very thick,dotted] (32.center) to (33.center);
      \draw[very thick,dotted] (36.center) to (37.center);
      \draw[very thick,dotted] (34.center) to (35.center);
      \draw[very thick,dotted] (39.center) to (38.center);
      \draw[very thick,dotted] (42.center) to (43.center);
      \draw[very thick,dotted] (40.center) to (41.center);
      \draw[very thick,dotted] (44.center) to (45.center);
      \draw[very thick,dotted] (50.center) to (51.center);
      \draw[very thick,dotted] (48.center) to (49.center);
      \draw[very thick,dotted] (46.center) to (47.center);
    \end{pgfonlayer}
  \end{tikzpicture}
  \caption{Four ways of gluing together the local spaces $b\mathcal{P}_1^-\Lambda^0(T)$.  Top left: Degrees of freedom on shared edges are equated, leading to a space of functions whose members are single-valued along edges but multi-valued at vertices.  Top right: In addition to equating degrees of freedom on shared edges, degrees of freedom on opposite endpoints of each edge are equated, leading to a space of functions whose members are single-valued \emph{and constant} along edges but multi-valued at vertices. Bottom left: Degrees of freedom at each vertex are equated, leading to the piecewise affine Lagrange finite element space.  Bottom right: Degrees of freedom at each vertex are equated only within individual triangles, leading to the space of discontinuous piecewise affine functions.}
  \label{fig:glue}
\end{figure}
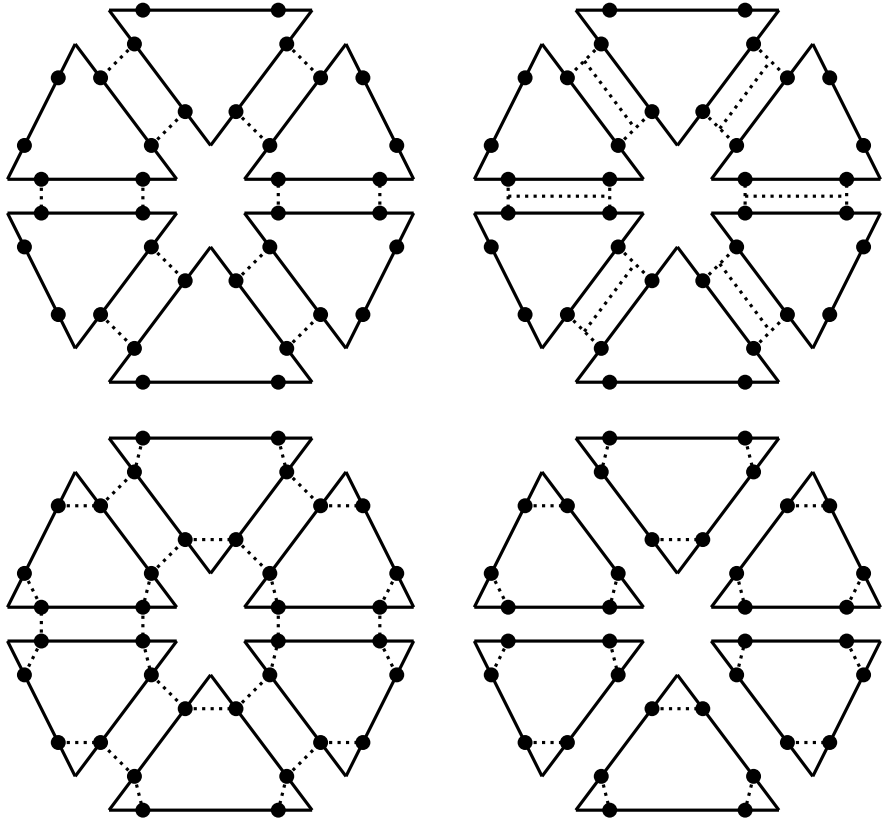

\begin{table}
  \begin{tabular}{c|c|c}
    $k$ & $F$ & $\psi_F$ \\
    \hline & & \\[-0.8em]
    0 & $012$ & $\displaystyle\frac{\lambda_0 \lambda_1}{\rr{012}\rr{12}}$ \\[1em]
    1 & $\{01\}2$ & $\displaystyle\frac{\varphi_{01}}{\rr{012}^2}$ \\[1em]
    1 & $0\{12\}$ & $\displaystyle\frac{\lambda_0 \phi_{12}}{\rr{012}\rr{12}} \left( \frac{1}{\rr{012}} + \frac{1}{\rr{12}} \right)$ \\[1em]
    2 & $\{012\}$ & $\displaystyle\frac{\phi_{012}}{\rr{012}^3}$ \\[1em]
  \end{tabular}
  \caption{Blow-up Whitney $k$-forms in dimension $n=2$.  Here, $\lambda_i$ denotes the barycentric coordinate function associated with vertex $i$, $\phi_{ij}$ denotes the classical Whitney 1-form associated with the oriented edge $(i,j)$, and $\phi_{ijk}$ denotes the classical Whitney 2-form associated with the oriented triangle $(i,j,k)$.  The symbols $\rr{12}$ and $\rr{012}$ are shorthand for $\lambda_1+\lambda_2$ and $\lambda_0+\lambda_1+\lambda_2$, respectively.  Note that $\lambda_{012}=1$; it is included in the formulas to homogenize the fractions.  In the left column, shorthand notation is used for flags.  For example, $\{01\}2$ is shorthand for the flag $(\{0,1\},\{2\})$. Only some flags are listed; the rest may be obtained by permuting the vertices of the triangle.} \label{tab:2d}
\end{table}

\begin{table}
  \begin{tabular}{c|c|c}
    $k$ & $F$ & $\psi_F$ \\
    \hline & & \\[-0.8em]
    0 & $0123$ & $\displaystyle\frac{\lambda_0 \lambda_1 \lambda_2}{\rr{0123}\rr{123}\rr{23}}$ \\[1em]
    1 & $\{01\}23$ & $\displaystyle\frac{\phi_{01} \lambda_2}{\rr{0123}^2 \rr{23}}$ \\[1em]
    1 & $0\{12\}3$ & $\displaystyle\frac{ \lambda_0 \phi_{12} }{ \rr{0123} \rr{123} } \left( \frac{1}{\rr{0123}} + \frac{1}{\rr{123}} \right)$ \\[1em]
    1 & $01\{23\}$ & $\displaystyle\frac{ \lambda_0 \lambda_1 \phi_{23} }{ \rr{0123} \rr{123} \rr{23} } \left( \frac{1}{\rr{0123}} + \frac{1}{\rr{123}} + \frac{1}{\rr{23}} \right)$ \\[1em]
    2 & $\{01\}\{23\}$ & $\displaystyle\frac{\phi_{01} \wedge \phi_{23}}{\rr{0123}^2 \rr{23}} \left( \frac{2}{\rr{0123}} + \frac{1}{\rr{23}} \right)$ \\[1em]
    2 & $\{012\}3$ & $\displaystyle\frac{\phi_{012}}{\rr{0123}^3}$ \\[1em]
    2 & $0\{123\}$ & $\displaystyle\frac{\lambda_0\phi_{123}}{\rr{0123}\rr{123}} \left( \frac{1}{\rr{0123}^2} + \frac{1}{\rr{0123}\rr{123}} + \frac{1}{\rr{123}^2} \right)$ \\[1em]
    3 & $\{0123\}$ & $\displaystyle\frac{\phi_{0123}}{\rr{0123}^4}$ \\[1em]
  \end{tabular}
  \caption{Blow-up Whitney $k$-forms in dimension $n=3$.  The notation here is similar to that in Table~\ref{tab:2d}.  For example, $\rr{123}=\lambda_1+\lambda_2+\lambda_3$, $\phi_{123}$ is the classical Whitney 2-form associated with the oriented triangle $(1,2,3)$, and $0\{12\}3$ is shorthand for the flag $(\{0\},\{1,2\},\{3\})$. Factors of $\rr{0123}=1$ are included to homogenize the fractions. Only some flags are listed; the rest may be obtained by permuting the vertices of the tetrahedron.} \label{tab:3d}
\end{table}

\subsection*{Bases for blow-up Whitney forms.}  Tables~\ref{tab:2d} and~\ref{tab:3d} list bases for the blow-up Whitney forms on an $n$-simplex $T$ in dimensions $n=2$ and $n=3$, respectively. To understand these tables, a few comments are in order. As we have alluded to above, in any dimension $n$, the degrees of freedom can be conveniently indexed by \emph{flags}: nested sequences of faces of $T$. However, it turns out to be more convenient to view a flag as an ordered partition of the vertices of $T$; from such an ordered partition, we can recover the nested sequence of faces by taking the span of the cumulative union of the sets of the partition. For example, if a triangle $T$ has vertices labeled $0,1,2$, then the flag $P_0<E_{01}<T$ is given by the ordered partition $(\{0\},\{1\},\{2\})$, the flag $P_0<T$ is given by the partition $(\{0\},\{1,2\})$, and the flag $E_{01}<T$ is given by the ordered partition $(\{0,1\},\{2\})$. Additionally, to make the notation more concise, we denote these flags with the shorthand $012$, $0\{12\}$, and $\{01\}2$, respectively.

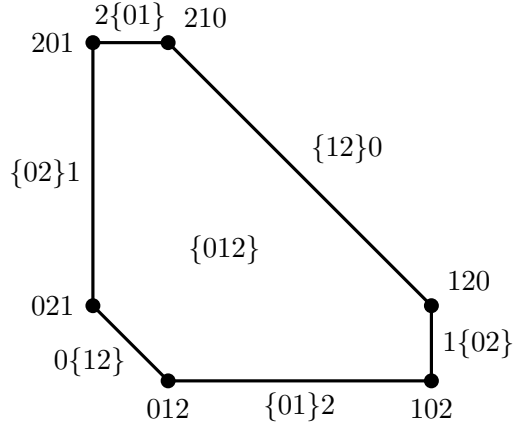
\begin{figure}
  \begin{tikzpicture}[scale=0.5]
    \begin{pgfonlayer}{nodelayer}
      \node [fill,circle,inner sep=2pt,label=left:$021$] (0) at (0, 2) {};
      \node [fill,circle,inner sep=2pt,label=below:$012$] (1) at (2, 0) {};
      \node [fill,circle,inner sep=2pt,label=below:$102$] (2) at (9, 0) {};
      \node [fill,circle,inner sep=2pt,label=above right:$120$] (3) at (9, 2) {};
      \node [fill,circle,inner sep=2pt,label=above right:$210$] (4) at (2, 9) {};
      \node [fill,circle,inner sep=2pt,label=left:$201$] (5) at (0, 9) {};
      \node [] at (3.5,3.5) {$\{012\}$};
    \end{pgfonlayer}
    \begin{pgfonlayer}{edgelayer}
      \draw[very thick] (5.center) to[edge label'=$\{02\}1$] (0.center);
      \draw[very thick] (0.center) to[edge label'=$0\{12\}$,inner sep=1pt]  (1.center);
      \draw[very thick] (1.center) to[edge label'=$\{01\}2$] (2.center);
      \draw[very thick] (2.center) to[edge label'=$1\{02\}$] (3.center);
      \draw[very thick] (3.center) to[edge label'=$\{12\}0$] (4.center);
      \draw[very thick] (4.center) to[edge label'=$2\{01\}$] (5.center);
    \end{pgfonlayer}
  \end{tikzpicture}
  \caption{The blow-up $\widetilde{T}$ of a triangle $T$ has 6 faces of dimension 1 that one can loosely think of as the 3 edges of $T$ and 3 infinitesimal arcs at the vertices of $T$. Here, the arcs are depicted as straight lines parallel to the opposite side.  This reflects the fact that we can parametrize a point on the arc by specifying a ray from the vertex through the point and seeing where it meets the opposite side.  The faces of $\widetilde{T}$ are labelled using shorthand notation for flags as described in the text.} \label{fig:blowuptri}
\end{figure}

As we discussed above, the degrees of freedom for blow-up Whitney $k$-forms correspond to integration over the $k$-dimensional faces of the blow-up $\widetilde{T}$, so, more properly, it is these faces that are indexed by flags. For example, in two dimensions, the blow-up $\widetilde{T}$ of a triangle $T$ has $6$ faces of dimension $1$ that one can loosely think of as the $3$ edges of $T$ and $3$ infinitesimal arcs at the vertices of $T$. The face corresponding to the edge $E_{01}$ is then given by the flag $E_{01}<T$, or $\{01\}2$; its endpoints are given by the flags $012$ and $102$. Meanwhile, the face corresponding to the arc at vertex $2$ corresponds to the flag $V_2<T$ or $2\{01\}$; its endpoints are given by the flags $201$ and $210$. See Figure~\ref{fig:blowuptri} for an illustration, and see Section~\ref{sec:cohomology} and Figure~\ref{fig:blowuptet} for information about how this extends to higher dimensions. 

Correspondingly, the dual basis we construct for the $6$-dimensional space of blow-up Whitney 1-forms on $T$ is given by integration over these 6 unidimensional faces of $\widetilde{T}$.  Point evaluations at the 6 endpoints of these unidimensional faces serve as the degrees of freedom for the $6$-dimensional space of blow-up Whitney 0-forms on $T$.  Integration over all of $\tl T$, which is equivalent to integration over $T$, serves as the degree of freedom for the $1$-dimensional space of blow-up Whitney 2-forms on $T$.  Table~\ref{tab:2d} lists bases for the blow-up Whitney forms on the triangle $T$ that are dual to these degrees of freedom, and Table~\ref{tab:3d} lists analogous bases in dimension $n=3$.  In other words, for a flag $F$, we list $\psi_F$, the unique blow-up Whitney form that evaluates to $1$ on the degree of freedom corresponding to the flag $F$ and $0$ on all others. (Please note that our Definition~\ref{def:whitney} of the classical Whitney $k$-forms differs from the usual one by a factor of $k!$.) Only a few of the basis forms are listed in the tables since all others can be obtained by permuting indices.

\subsection*{Organization}  
This paper is organized as follows.  In Section~\ref{sec:blowupWhitney}, we define the blow-up Whitney forms, introduce degrees of freedom for them, and point out a link with Poisson processes that allows us to quickly write down bases that are dual to those degrees of freedom.  In Section~\ref{sec:complex}, we use the Poisson process perspective to prove that the blow-up Whitney forms on a simplex $T$ form a complex, and we study its cohomology by defining the blow-up $\widetilde{T}$ of $T$ and showing that the complex of blow-up Whitney forms is isomorphic to the cellular cochain complex of $\widetilde{T}$. We conclude with some speculation about the cohomology of the complex of blow-up Whitney forms on a triangulation (as opposed to a single simplex), as well as with some preliminary results on generalizations to higher order.

\section{Blow-up Whitney forms} \label{sec:blowupWhitney}
\subsection{The quasi-cylindrical coordinate system on a simplex}
\begin{definition}
  For a set $V$, let $T=T_V$ denote the standard barycentric coordinate simplex
  \begin{equation*}
    T_V=\left\{\left(\lambda_i\right)_{i\in V}\in\bR_{\ge0}^V\mid\sum_{i\in V}\lambda_i=1\right\}.
  \end{equation*}
  
  Throughout, we will let $n=\dim T_V=\abs V-1$.
\end{definition}

Commonly, $V$ is simply taken to be the set $\{0,\dotsc,n\}$, and so the barycentric coordinates are simply $\lambda_0,\dotsc,\lambda_n$. However, we keep the notation that $V$ is a general set for a couple reasons. One reason is that we want to consider faces of $T_V$. Faces of $T_V$ are all of the form $T_W$ where $W$ is a subset of $V$, but $W$ will generally not be a set of consecutive integers even if $V$ is. For example, the edge joining vertices $0$ and $2$ will be $T_W$ for $W=\{0,2\}$. Another reason to use the general set notation is that we may want to consider a full triangulation, not just a single element. In that context, it makes sense to index \emph{all} of the vertices of the triangulation; then, for a particular element $T$, we take $V$ to be the indices of the vertices of $T$.

We use the set-theory notation $\bR^V$ rather than $\bR^{n+1}$ for similar reasons: if $V=\{0,1,2\}$ and $W=\{0,2\}$ and we talk about the projection $\bR^V\to\bR^W$, it is clear \emph{which} coordinates are being projected, whereas it is ambiguous with $\bR^3\to\bR^2$.

\begin{definition}\label{def:flag}
  For a set $V$, a \emph{flag} $F$ is an ordered partition of $V$, which we denote by
  \begin{equation*}
    F=\left(V_0,V_1,\dotsc,V_{n-k}\right).
  \end{equation*}
  As before, $n+1$ denotes $\abs V$, and note that $k$ is the \emph{complement} of the number of sets in the partition. Observe that, if $k=0$, then $F$ is just an ordering of $V$, and if $k=n$, then $F=(V)$.

  In the context of a particular flag, we will
  \begin{itemize}
  \item let $J$ denote the index set $\{0,\dotsc,n-k\}$,
  \item let $T_j$ be shorthand for $T_{V_j}$,
  \item let $n_j$ be shorthand for $\dim T_j=\abs{V_j}-1$, and
  \item let $F!$ be shorthand for $\prod_{j\in J}n_j!$.
  \end{itemize}

\end{definition}

\begin{definition}[Quasi-cylindrical coordinates]
  Fix a set $V$ and a flag $F$, with notation as above. Let the \emph{radial simplex} $R=R_F$ associated to $F$ be the simplex $T_J$. In other words, each vertex $j$ of $R_F$ corresponds to the set of vertices $V_j$ of $T$. Let $\rho_0,\dotsc,\rho_{n-k}$ be the barycentric coordinates for $R$.

  Next, define the \emph{angular space} $\Theta$ by
  \begin{equation*}
    \Theta=\Theta_F=\prod_{j\in J}T_j=\left\{\left(\theta_i\right)_{i\in V}\in\bR^V_{\ge0}\mid\sum_{i\in V_j}\theta_i=1\quad\forall j\right\}.
  \end{equation*}

  Then the \emph{quasi-cylindrical coordinate system} on $T=T_V$ is defined by
  \begin{equation*}
    \lambda_i=\rho_j\theta_i,
  \end{equation*}
  for all $j$ and all $i\in V_j$.
\end{definition}

\begin{proposition}\label{prop:cyl}
  With notation as above, the quasi-cylindrical coordinate system is an isomorphism onto its image when restricted to the interior $\mr R$ of $R$, that is, the subset where $\rho_j\neq0$ for all $j$. If we also restrict to the interior $\mr\Theta$ of $\Theta$, then we obtain an isomorphism
  \begin{equation*}
    \mr R\times\mr\Theta\to\mr T.
  \end{equation*}
\end{proposition}
\begin{proof}
  Given $\rho_j$ and $\theta_i$, we compute
  \begin{equation*}
    \sum_{i\in V}\lambda_i=\sum_{j\in J}\sum_{i\in V_j}\lambda_i=\sum_{j\in J}\rho_j\sum_{i\in V_j}\theta_i=\sum_{j\in J}\rho_j=1.
  \end{equation*}
  Conversely, given the $\lambda_i$, set $\rho_j=\sum_{i\in V_j}\lambda_i$, and set $\theta_i=\frac{\lambda_i}{\rho_j}$. Here, we must use the $\rho_j\neq0$ assumption. Then we check that
  \begin{equation*}
    \sum_{j\in J}\rho_j=\sum_{j\in J}\sum_{i\in V_j}\lambda_i=\sum_{i\in V}\lambda_i=1
  \end{equation*}
  and, for each $j$,
  \begin{equation*}
    \sum_{i\in V_j}\theta_i=\frac1{\rho_j}\sum_{i\in V_j}\lambda_i=\frac{\rho_j}{\rho_j}=1.
  \end{equation*}
  Finally, we check that $\lambda_i\neq0$ for all $i$ is equivalent to $\rho_j\neq0$ and $\theta_i\neq0$ for all $i$ and $j$.
\end{proof}

Note that the quasi-cylindrical coordinate system depends only on the partition of $F$, not the order.

\begin{example}
  If $k=0$, then recall that $F$ is just an ordering of $V$, that is, a bijection between $\{0,\dotsc,n\}$ and $V$. This bijection then induces an isomorphism between $R$ and $T$. Meanwhile, each $T_j$ is a single-point set, so $\Theta$ is a single-point set as well.
\end{example}

\begin{example}\label{eg:k1}
  If $k=1$, then all but one $V_j$ has a single element. Let $j^*$ be the special value of $j$, and let $v_-$ and $v_+$ be the two elements of $V_{j^*}$. Then, for $j\neq j^*$, $T_j$ is a single point; meanwhile, $T_{j^*}$ is an interval. Therefore, $\Theta$ is an interval. Meanwhile, $R$ is an $(n-1)$-dimensional simplex, with one vertex corresponding to the set $\{v_-,v_+\}$, and one vertex for each other vertex of $T$.
\end{example}

\begin{example}\label{eg:kn1}
  If $k=n-1$, then $R$ is an interval, and $\Theta=T_0\times T_1$.
\end{example}

\begin{example}
  If $k=n$, then $R$ is a single-point set, and $\Theta=T$.
\end{example}

We give a brief summary of the notation implied by Proposition~\ref{prop:cyl}.

\begin{notation}
  In the context of a flag $F$:
  \begin{itemize}
  \item $R=R_F$ is the simplex with barycentric coordinates $\rho_j$, with $j\in J=\{0,\dotsc,n-k\}$.
  \item $\Theta=\Theta_F=\prod_{j\in J}T_j$, with barycentric coordinates $\theta_i$ with $i\in V_j$ on each factor.
  \item For $i\in V_j$, we have
    \begin{equation*}
      \lambda_i=\rho_j\theta_i,\qquad\rho_j=\sum_{i\in V_j}\lambda_i,\qquad\theta_i=\frac{\lambda_i}{\rho_j}.
    \end{equation*}
  \end{itemize}
\end{notation}

We also introduce some additional notation.

\begin{notation}
  In the context of a flag $F$:
  \begin{itemize}
  \item We will denote points in $R=R_F$ by $\brho$.
  \item We will denote points in $\Theta=\Theta_F$ by $\btheta$.
  \item We will let $\brho^F$ be shorthand for $\prod_{j\in J}\rho_j^{\abs{V_j}}$.
  \end{itemize}
\end{notation}

\subsection{Whitney forms}
\begin{definition}
  We define an \emph{orientation} of $V$ to be an ordering of $V$ modulo even permutations. An orientation of $V$ is equivalent to an orientation of $\bR^V$, which in turn is equivalent to an orientation of $T_V$ with the outward normal.
\end{definition}

\begin{notation}
  Let $X_{\id}$ be the tautological vector field in $\bR^V$ given by
  \begin{equation*}
    X_{\id}:=\sum_{i\in V}\lambda_i\pp{\lambda_i},
  \end{equation*}
  or, equivalently, by
  \begin{equation*}
    X_{\id}[\lambda_i]=\lambda_i.
  \end{equation*}
  Let $\d\lambda_V$ be shorthand for
  \begin{equation*}
    \d\lambda_V:=\bigwedge_{i\in V}d\lambda_i,
  \end{equation*}
  with an ordering compatible with the orientation, and let
  \begin{equation*}
    \rho=\rho_V:=\sum_{i\in V}\lambda_i.
  \end{equation*}
  Note that $T_V$ is defined by $\rho_V=1$ on $\bR_{\ge0}^V$.
\end{notation}

\begin{definition}\label{def:whitney}
  We define the \emph{Whitney form} $\phi=\phi_V$ on $\bR^V$ to be the contraction
  \begin{equation*}
    \phi=\phi_V:=n!\,i_{X_{\id}}\d\lambda_V.
  \end{equation*}
  Additionally, on $\bR_{>0}^V$, we define the \emph{homogenized Whitney form} to be
  \begin{equation*}
    \omega=\omega_V:=\frac{\phi_V}{\rho_V^{\abs{V}}}.
  \end{equation*}
  Note that, restricted to $T_V$, the forms $\phi$ and $\omega$ are equal.
  
  Finally, if $W$ is a subset of $V$, then $\phi_W$ can be viewed as a form on $\bR^V$ and $\omega_W$ can be viewed as a form on $\bR_{>0}^V$ by pulling it back via the projections $\bR^V\to\bR^W$.
\end{definition}

Our definition of the Whitney forms differs from the usual one by a factor of $n!$, for the following reason, which also underlies our choice of symbol $\omega$ for the homogenized Whitney form.

\begin{proposition}\label{prop:intwhitney}
  Restricted to $T_V$, the Whitney form $\phi_V$ and the homogenized Whitney form $\omega_V$ are a constant multiple of the volume form on $T_V$, scaled so that their integral over $T_V$ is $1$.
\end{proposition}

We begin the proof with a simple lemma that we will use frequently.

\begin{lemma}\label{lem:intwhitney}
  For a set $V$ of size $n+1$, we have
  \begin{equation*}
    \d\lambda_V=\frac{\d\rho}{\rho}\wedge\frac{\phi_V}{n!},
  \end{equation*}
  where, as before, $\rho=\sum_{i\in V}\lambda_i$.
\end{lemma}

\begin{proof}
  By dimension, $\d\rho\wedge\d\lambda_V=0$. Applying $i_{X_{\id}}$, we therefore have that
  \begin{equation*}
    0=i_{X_{\id}}(d\rho)\d\lambda_V-\d\rho\wedge i_{X_{\id}}\d\lambda_V=\rho\d\lambda_V-\d\rho\wedge\frac{\phi_V}{n!}.\qedhere
  \end{equation*}
\end{proof}

\begin{proof}[Proof of Proposition~\ref{prop:intwhitney}]
  As discussed, $\phi_V$ and $\omega_V$ are equal when restricted to $T_V$; we will focus on $\phi_V$. Let $X_1,\dotsc,X_n$ be vectors tangent to $T_V$. Observing that $\frac{\d\rho}\rho(X_{\id})=1$ and $\frac{\d\rho}\rho(X_i)=0$ for $1\le i\le n$, we have
  \begin{equation}\label{eq:phivx}
    \d\lambda_V(X_{\id},X_1,\dotsc,X_n)=\frac{\phi_V}{n!}(X_1,\dotsc,X_n).
  \end{equation}
  
  Now set $\rho=1$, choose a vertex $0\in V$, and let $X_1,\dotsc,X_n$ be the displacement vectors from $0$ to the other vertices of $T_V$, chosen so that $X_1,\dotsc,X_n$ is positively oriented. Observing that the $(n+1)$ by $(n+1)$ matrix
  \begin{equation*}
    \begin{pmatrix}
      \lambda_0&\lambda_1&\lambda_2&\lambda_3&\cdots&\lambda_n\\
      -1&1&0&0&\cdots&0\\
      -1&0&1&0&\cdots&0\\
      -1&0&0&1&&0\\
      \vdots&\vdots&\vdots&&\ddots&0\\
      -1&0&0&0&\cdots&1
    \end{pmatrix}
  \end{equation*}
  has determinant $\lambda_0+\dotsb+\lambda_n=\rho=1$, we conclude that
  \begin{equation*}
    \frac{\phi_V}{n!}(X_1,\dotsc,X_n)=1.
  \end{equation*}
  In particular, this quantity does not depend on the $\lambda_i$, so we know that $\frac{\phi_V}{n!}$ is a constant multiple of the volume form on $T_V$. Recalling that the volume of an $n$-dimensional parallelepiped is $n!$ times the volume of the corresponding simplex, we conclude that $\int_{T_V}\phi_V=1$, as desired.
\end{proof}

In light of the above proposition, one may wonder why we bother with the homogenized Whitney form $\omega_V$ at all. We will shortly see that it is, in fact, the more natural way to extend the normalized volume form on $T_V$ to $\bR_{>0}^V$.

\begin{proposition}\label{prop:quasisphere}
  Consider the quasi-spherical projection $\bR_{>0}^V\to\mr T_V$ defined by $\theta_i=\frac{\lambda_i}{\rho}$, where $\rho=\sum_{i\in V}\lambda_i$. Here, the $\lambda_i$ are the coordinates of $\bR_{>0}^V$, and the $\theta_i$ are the barycentric coordinates of $\mr T_V$. Then $\omega_V$ is the pullback of the normalized volume form on $T_V$.
\end{proposition}

\begin{proof}
  Observe that we have an automorphism $\bR_{>0}\times\bR_{>0}^V\to\bR_{>0}\times\bR_{>0}^V$ with $(\sigma,\blambda)\mapsto(\rho,\btheta)$ given by
  \begin{equation*}
    \rho=\sum_{i\in V}\lambda_i,\qquad\theta_i=\frac\sigma\rho\lambda_i,
  \end{equation*}
  or conversely,
  \begin{equation*}
    \sigma=\sum_{i\in V}\theta_i,\qquad\lambda_i=\frac\rho\sigma\theta_i.
  \end{equation*}
  The quasi-spherical projection is simply the composition of this automorphism with natural inclusions and projections, given by
  \begin{equation*}
    \bR_{>0}^V\xrightarrow{\sigma=1}\bR_{>0}\times\bR_{>0}^V\to\bR_{>0}\times\bR_{>0}^V\rightarrow\bR_{>0}^V.
  \end{equation*}
  We now let $X_{\id}$ be the tautological vector field on the domain $\bR_{>0}\times\bR_{>0}^V$, that is, $X_{\id}=\sigma\pp\sigma+\sum_i\lambda_i\pp{\lambda_i}$. In other words, we have $X_{\id}[\lambda_i]=\lambda_i$ as before, and now we also have $X_{\id}[\sigma]=\sigma$. By dilation-equivariance of the automorphism or by explicit computation, we then have $X_{\id}[\theta_i]=\theta_i$ and $X_{\id}[\rho]=\rho$. In particular, we have $X_{\id}[\frac\sigma\rho]=0$. So, then,
  \begin{equation*}
    n!i_{X_{\id}}\d\theta_V=\left(\frac\sigma\rho\right)^{\abs V}n!i_{X_{\id}}\d\lambda_V.
  \end{equation*}
  We now restrict to $\sigma=1$, which is the defining equation of $T_V$ in the $\btheta$ coordinate system. Applying Proposition~\ref{prop:intwhitney} to the $\btheta$ coordinate system, we have that the left-hand side is the normalized volume form on $T_V$. Meanwhile, at $\sigma=1$, the right-hand side is $\omega_V$ by definition.
\end{proof}

We now discuss the implications for the quasi-cylindrical coordinate system, but first we establish notation.

Henceforth, we assume that we have selected an orientation of $T=T_V$, and, as is common in finite elements, we assume that we have preselected an orientation for every face as well. In particular, in the context of a flag $F$, each $T_j=T_{V_j}$ is equipped with an orientation, and thus so is $\Theta_F=\prod_{j=0}^{n-k}T_j$. With that in mind, we introduce additional shorthand.

\begin{notation}
  Given a flag $F$:
  \begin{itemize}
  \item Let $\phi_j$ be shorthand for $\phi_{V_j}$ and $\omega_j$ be shorthand for $\omega_{V_j}$.
  \item Let $\phi_F$ be shorthand for $\bigwedge_{j=0}^{n-k}\phi_j$ and $\omega_F$ be shorthand for $\bigwedge_{j=0}^{n-k}\omega_j$.
  \end{itemize}
  Note that, per our established notation, we have
  \begin{equation*}
    \phi_F=F!\bigwedge_{j\in J}i_{X_{\id}}\d\lambda_{V_j},\qquad\omega_F=\brho^{-F}\phi_F.
  \end{equation*}
\end{notation}

\begin{corollary}\label{cor:omegavol}
  Given a flag $F$, the form $\omega_F$ is the pullback of the normalized volume form on $\Theta_F$ via the quasi-cylindrical projection $\mr T\cong\mr R_F\times\mr\Theta_F\to\Theta_F$.
\end{corollary}

\begin{proof}
  We apply Proposition~\ref{prop:quasisphere} to each $\bR_{>0}^{V_j}\to\mr T_{V_j}$, obtaining that $\omega_j$ is the pullback of the normalized volume form on $T_j$ via the quasi-spherical projection on $\bR_{>0}^{V_j}$. Taking the product over $j$, we obtain a map $\bR_{>0}^V\to\mr\Theta_F$. This map is our quasi-cylindrical coordinate projection when restricted to $\mr T\subset\bR_{>0}^V$. The normalized volume form on $\Theta_F$ is the wedge product of the normalized volume forms on $T_j$, and so $\omega_F$ is its pullback.
\end{proof}

It remains to discuss the orientation of $R$. It turns out to be convenient to \emph{not} give $R$ the orientation implied by the order $\rho_0,\dotsc,\rho_{n-k}$. Instead, we orient $R$ so that the quasi-cylindrical coordinate system preserves orientation.

\begin{definition}
  Let $F$ be a flag on a set $V$ with notation as above. Assume that we have an orientation of $T$ and of each $T_j$. Then let $R$ have the orientation such that the map $\mr R\times\mr\Theta\to\mr T$ in Proposition~\ref{prop:cyl} preserves orientation.
\end{definition}

\subsection{Blow-up Whitney forms}
We are now ready to define the blow-up Whitney forms. As we have discussed, these are the same as the shadow forms of \cite{brasselet1991simplicial}, but we give a self-contained exposition emphasizing the blow-up/quasi-cylindrical perspective that we use later in the paper. For readers who wish to understand the equivalence with \cite{brasselet1991simplicial}, we recommend comparing Proposition~\ref{prop:psiexplicit} with \cite[Theorem 4.1]{brasselet1991simplicial} rather than comparing the definitions.

\begin{definition}\label{def:rs}
  Let $R$ be a simplex, and let $\brho$ be a point in the interior of $R$ with barycentric coordinates $\rho_0,\dotsc,\rho_{n-k}$. Then let $R_\brho$ be the subset of $R$ defined by
  \begin{align*}
    R_\brho&=\left\{\left(r_0,\dotsc,r_{n-k}\right)\in R\mid\frac{r_j}{\rho_{j}}\le\frac{r_{j'}}{\rho_{j'}}\quad\text{whenever}\quad j\le j'\right\}\\
           &=\left\{\left(r_0,\dotsc,r_{n-k}\right)\in R\mid\frac{r_{j-1}}{\rho_{j-1}}\le\frac{r_{j}}{\rho_j}\quad\text{for all}\quad 1\le j\le n-k\right\}\\
  \end{align*}
\end{definition}

An illustration of $R_\brho$ is given in Figure~\ref{fig:region}.

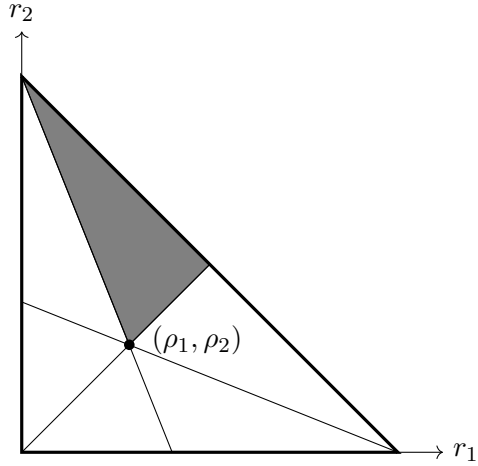
\begin{figure}
  \centering
  \begin{center}
    \begin{tikzpicture}[scale=2]
      \draw[fill=gray] (0,2.5) -- (1.25,1.25) -- (0.715,0.715) -- cycle;
      \draw[->] (0,0) -- (xyz cs:x=2.8) node[right] {$r_1$};
      \draw[->] (0,0) -- (xyz cs:y=2.8) node[above] {$r_2$};
      \draw[very thick] (0,0) node[label=$$]{$$}
      -- (2.5,0) node[anchor=north]{$$}
      -- (0,2.5) node[anchor=south]{$$}
      -- cycle;
      \draw (2.5,0) -- (0,1);
      \draw (0,2.5) -- (1,0);
      \draw (0,0) -- (1.25,1.25);
      \fill[black] (0.715,0.715) circle (1pt) node[label={[shift={(0.9,-0.4)}]$(\rho_1,\rho_2)$}] {};
    \end{tikzpicture}
  \end{center}
  \caption{If $R$ is a triangle and $\brho$ is a point in $R$, then $R_\brho$ is the shaded region depicted above.} \label{fig:region}
\end{figure}

\begin{definition}[Blow-up Whitney forms]\label{def:blowupwhitney}
  Fix a set $V$ and a flag $F$, with notation $T=T_V$, $R=T_J$, $T_j=T_{V_j}$, and $\Theta=\prod_{j\in J}T_j$ as above. As above, fixing an orientation for the $T_j$, choose an orientation for $R$ to ensure that the above isomorphism between the interiors of $T$ and $R\times\Theta$ preserves orientation. As before, let $\omega$ be the constant multiple of the volume form on $T$ whose integral over $T$ is one.

  Note that $\omega$ is an $n$-form, and that $R_\brho$ is $(n-k)$-dimensional. So then, by Fubini's theorem, for each $\brho\in\mr R$,
  \begin{equation*}
    \psi_{F,\brho}:=\int_{R_\brho}\omega
  \end{equation*}
  defines a $k$-form on $\Theta$. Putting all of these together, we obtain a $k$-form $\psi_F$ on $T$, which we call the \emph{blow-up Whitney form associated to $F$}. In detail, given a point $x$ in the interior of $T$ and vectors $X_1,\dotsc,X_k$ at $x$, we let $\brho$ be the projection of $x$ onto the $R$ factor, and we let $Y_1,\dotsc,Y_k$ be the projections of the $X_1,\dotsc,X_k$ onto the $\Theta$ factor. We then set $\psi_F(X_1,\dotsc,X_k):=\psi_{F,\brho}(Y_1,\dotsc,Y_k)$.

  In summary, to evaluate $\psi_F$ at a point $x\in\mr T$ with quasi-cylindrical coordinates $(\brho,\btheta)$, we perform the following operations to the volume form $\omega$ on $T$. Here, $\Phi\colon R\times\Theta\to T$ denotes the quasi-cylindrical coordinate map, and $\pi_\Theta\colon R\times\Theta\to\Theta$ is the projection.
  \begin{equation*}
    \begin{tikzcd}[column sep=small, row sep=tiny]
      \Lambda^n(\mr T)\arrow[r,"\Phi^*"]&\Lambda^n(\mr R\times\mr\Theta)\arrow[r,"\int_{R_\brho}"]&\Lambda^k(\mr\Theta)\arrow[r,"\pi_\Theta^*"]&\Lambda^k(\mr R\times\mr\Theta)\arrow[r,"\left(\Phi^{-1}\right)^*"]&\Lambda^k(\mr T)\arrow[r,"\rvert_x"]&\Lambda^kT^*_xT\\
      \omega\arrow[r, mapsto]&\omega\arrow[r, mapsto]&\psi_{F,\brho}\arrow[r, mapsto]&\psi_{F,\brho}\arrow[r, mapsto]&\psi_{F,\brho}\arrow[r, mapsto]&\psi_F\bigr\rvert_x
    \end{tikzcd}
  \end{equation*}
  Note that, in the second line, we treat some pullbacks as implicit, in the same sense that $\d\theta_i$ can be viewed as a one-form on $\Theta$, $R\times\Theta$, or $T$. We also emphasize that the operation depends on $x$ not only in the last step (evaluation), but also in the second step, because $R_{\brho}$ depends on $\brho$ which depends on $x$.
\end{definition}

\begin{example}
  Consider a triangle $T$ with vertices $V=\{0,1,2\}$, and let $F = (\{0\},\{1,2\})$.  Then $R \times \Theta$ is a product of two intervals, and the quasi-cylindrical coordinates satisfy
  \[
    \lambda_0 = \rho_0, \quad \lambda_1 = \rho_1 \theta_1, \quad \lambda_2 = \rho_1 \theta_2, \quad \rho_0 + \rho_1 = 1, \quad \theta_1+\theta_2=1.
  \]
  After some algebra, one can express the volume form 
  \[
    \omega = 2\frac{\lambda_0 d\lambda_1 \wedge d\lambda_2 + \lambda_1 d\lambda_2 \wedge d\lambda_0  + \lambda_2 d\lambda_0 \wedge d\lambda_1}{(\lambda_0+\lambda_1+\lambda_2)^3}
  \]
  in quasi-cylindrical coordinates as
  \[
    \omega = (\theta_1 d\theta_2 - \theta_2 d\theta_1) \wedge \frac{2\rho_1 d\rho_0}{(\rho_0+\rho_1)^2}.
  \]
  Note that $\theta_1\d\theta_2-\theta_2\d\theta_1$ is just $\d\theta_2$, but this form makes it easier to check that $\theta_1\d\theta_2-\theta_2\d\theta_1=\frac{\phi_{12}}{\rho_1^2}$, where $\varphi_{12}=\lambda_1 d\lambda_2 - \lambda_2 d\lambda_1$.
  Since $R$ is just an interval, the set $R_\brho$ is defined by $r_0\in[0,\rho_0]$. Since $\rho_1=1-\rho_0$, integrating over $R_{\brho}$ produces
  \[
    \psi_{F,\brho} = (\theta_1 d\theta_2 - \theta_2 d\theta_1) \frac{2\rho_0-\rho_0^2}{(\rho_0+\rho_1)^2}.
  \]
  One checks that $\frac{2\rho_0-\rho_0^2}{(\rho_0+\rho_1)^2} = \frac{\lambda_0}{\lambda_{012}} \left( \frac{\lambda_{12}}{\lambda_{012}} + 1 \right)$, where $\lambda_{12}=\lambda_1+\lambda_2=\rho_1$ and $\lambda_{012}=\lambda_0+\lambda_1+\lambda_2=1$, so 
  \[
    \psi_F= \frac{\lambda_0 \varphi_{12}}{\lambda_{012}\lambda_{12}} \left( \frac{1}{\lambda_{012}} + \frac{1}{\lambda_{12}} \right).
  \]
  In Section~\ref{sec:poisson} we will show that such formulas for $\psi_F$ can be arrived at more easily using a combinatorial calculation associated with Poisson processes.
\end{example}

The following proposition gives a formula for $\psi_F$ that also appears in~\cite[Theorem 4.1]{brasselet1991simplicial}.

\begin{proposition}\label{prop:psiexplicit}
  We have that
  \begin{equation*}
    \psi_F=p_F\omega_F
  \end{equation*}
  where $p_F(\brho)$ is the relative volume of the subset $R_\brho\times\Theta$ of $T$, and we recall that $\omega_F=\brho^{-F}\phi_F=\bigwedge_{j=0}^{n-k}\rho_j^{-\abs{V_j}}\phi_{V_j}$ is the normalized volume form on $\Theta$, pulled back to $T$ via the quasi-cylindrical projection $\mr T\cong\mr R\times\mr\Theta\to\Theta$.
\end{proposition}

\begin{proof}
  Let $Y$ be a vector in $\Theta$, and let $X$ be its lift to a vector tangent to the $\Theta$ factor in the product decomposition $R\times\Theta$, which, in particular, means that $\d\rho_j(X)=0$ for all $j$. Consequently, by $\lambda_i=\rho_j\theta_i$, we have that $\d\lambda_i(X)=\rho_j\d\theta_i(X)$. The important thing to note is that the coefficient of $\d\theta_i$ depends only on the $\rho$ coordinates and not on the $\theta$ coordinates. Now let $Y_1,\dotsc,Y_k$, and, respectively, $X_1,\dotsc,X_k$ be $k$ such vectors, and let $\eta$ be the $(n-k)$-form obtained by contracting $\omega$ on the right with $X_1,\dotsc,X_k$. Since $\omega$ can be expressed as a constant multiple of a wedge product of the $\d\lambda_i$, we conclude that $\eta$ likewise has a coefficient that depends only on the $\rho$ coordinates and not on the $\theta$ coordinates, when expressed in terms of the $\d\theta_i$. Fixing a $\brho\in\mr R$, to obtain $\psi_{F,\brho}(Y_1,\dotsc,Y_k)$, we then integrate $\eta$ over each $R_\brho\times\{\btheta\}$ for $\btheta\in\Theta$. We can conclude that the $k$-form $\psi_{F,\brho}$ on the $k$-dimensional space $\Theta$ has a \emph{constant} coefficient when expressed in terms of the $\d\theta_i$. But note that the volume form on each $T_j$ likewise has a constant coefficient when expressed in terms of the $\d\theta_i$. We conclude that $\psi_{F,\brho}$ is a constant multiple of $\omega_F$, the normalized volume form on $\Theta_F=\prod_jT_j$. So, we have that
  \begin{equation*}
    \psi_{F,\brho}=p\omega_F,
  \end{equation*}
  for some function $p$ that depends on $\brho$, but does not depend on the $\theta$ variables or, equivalently, does not depend on the position $\btheta\in\Theta$.

  From here, we can compute $p$ explicitly by integrating both sides over $\Theta$. On the right-hand side, since $\omega_F$ integrates to $1$ over $\Theta$, we simply have $p$. So then,
  \begin{equation*}
    p=\int_\Theta\psi_{F,\brho}=\int_\Theta\int_{R_\brho}\omega=\int_{R_\brho\times\Theta}\omega,
  \end{equation*}
  which is the relative volume of $R_\brho\times\Theta$ in $T$. Note that here we used the fact that the orientation of $R$ and hence $R_\brho$ was chosen so that the orientation of $R\times\Theta$ matches that of $T$.
\end{proof}

\subsection{Unisolvence}
We now discuss the degrees of freedom, that is, the dual basis, for the blow-up Whitney forms.

\begin{definition}\label{def:dof}
  Fix a set $V$ and a flag $F=(V_0,V_1,\dotsc,V_{n-k})$. As before, let $T=T_V$, $R=R_F=T_J$, $T_j=T_{V_j}$, and $\Theta=\prod_{j\in J}T_j$, and recall the isomorphism between the interior of $T$ and the interior of $R\times\Theta$, where $\dim R=n-k$ and $\dim\Theta=k$. We will define the \emph{degree of freedom} $\Psi_F$ associated to $F$, a functional from $k$-forms on $T$ to the real numbers, as follows.

  Let $\psi$ be a $k$-form on $T$. First, for any $\brho\in\mr R$, we let $\psi_\brho$ be the pullback of $\psi$ to $\mr\Theta$ under the inclusion $\mr\Theta\cong\{\brho\}\times\mr\Theta\hookrightarrow T$. In other words, for $Y_1,\dotsc,Y_k$ vectors at a point $\btheta\in\Theta$, we let $x\in T$ be the point in $T$ mapping to $(\brho,\btheta)$ under the isomorphism $\mr T\to\mr R\times\mr\Theta$, and we let $X_1,\dotsc,X_k$ be the vectors at $x$ tangent to the $\Theta$ factor of $\mr R\times\mr\Theta$ that project to $Y_1,\dotsc,Y_k$ under the map $\mr T\cong\mr R\times\mr\Theta\to\mr\Theta$. Then $\psi_\brho(Y_1,\dotsc,Y_k):=\psi(X_1,\dotsc,X_k)$.

  So, we now have a set of $k$-forms on the $k$-dimensional space $\Theta$ parametrized by $\brho$. We compute the following limit of the $\psi_\brho$ followed by integrating over $\Theta$.
  \begin{equation*}
    \Psi_F(\psi):=\int_\Theta\lim_{\rho_1\to0}\dotsb\lim_{\rho_{n-k}\to0}\psi_\brho.
  \end{equation*}
  Here, the limits are taken pointwise on $\mr\Theta$. In particular, fixing a point in $\mr\Theta$ means that when we take the limit $\lim_{\rho_j\to0}$, we take the limit ``by dilation'', that is, along a path where the ratios of the $\lambda_i$ for $i\in V_j$ are fixed. In addition, as implied by the sequential limits, the $\rho_j$ remain nonzero until that limit is taken: In other words, when we take the limit $\rho_{n-k}\to0$, we approach points in the interior of the face $\rho_{n-k}=0$ of $R$; with the next limit, we approach points in the interior of the subface of $\rho_{n-k}=0$ defined by $\rho_{n-k-1}=0$, and so forth. In particular, since we are on the interior of $\Theta$, the $\lambda_i$ also remain nonzero until we take $\lim_{\rho_j\to0}$ where $i\in V_j$.
  
  Note that we could also write down a more natural formula
  \begin{equation*}
    \Psi_F(\psi):=\int_\Theta\lim_{\rho_0\to0}\lim_{\rho_1\to0}\dotsb\lim_{\rho_{n-k}\to0}\psi_\brho.
  \end{equation*}
  provided we extend $\psi$ with appropriate homogeneous scaling to the orthant $\bR_{>0}^V$.
\end{definition}

\begin{remark}
  The degrees of freedom may be undefined if the limits or integral fail to converge. As we will see, they are well-defined on the blow-up Whitney forms. However, even on blow-up Whitney forms, while we have pointwise convergence of the $\psi_\brho$, we generally do not expect $L^1$ convergence, and so it matters that the limits are taken before the integral.
\end{remark}
\begin{remark}
  More generally, we will later define the blow-up $\tl T$ of $T$, and we will show in Proposition~\ref{prop:dof} that the degrees of freedom above are equivalent to integration over the $k$-dimensional faces of $\tl T$. In particular, the degrees of freedom are defined on all $k$-forms that are smooth on $\tl T$ (a weaker condition than smoothness on $T$).
\end{remark}

\begin{theorem}[Unisolvence]
  The basis $\Psi_F$ is dual to the basis $\psi_F$.
\end{theorem}

\begin{proof}
  We first show that $\Psi_F(\psi_F)=1$. With notation as above, let $\brho\in R$. Comparing the definition of the degrees of freedom and the definition of the blow-up Whitney forms, we see that $\left(\psi_F\right)_\brho$ from Definition~\ref{def:dof} is just $\psi_{F,\brho}=\int_{R_\brho}\omega$ from Definition~\ref{def:blowupwhitney}. So, we then have that
  \begin{equation*}
    \Psi_F(\psi_F)=\int_\Theta\lim_{\rho_1\to0}\dotsb\lim_{\rho_{n-k}\to0}\int_{R_\brho}\omega,
  \end{equation*}
  where we recall that $\omega$ is the volume form for $T$ scaled so that the volume of $T$ is one.

  Next, we recall Definition~\ref{def:rs} of $R_\brho$. As we send $\rho_{n-k}\to0$ while keeping the other $\rho_j$ nonzero, the restriction $\frac{r_{n-k-1}}{r_{n-k}}\le\frac{\rho_{n-k-1}}{\rho_{n-k}}$ becomes vacuously true, while the other inequalities $\frac{r_{j-1}}{r_j}\le\frac{\rho_{j-1}}{\rho_j}$ are unaffected. As we continue the iterated limits, when we reach $\rho_j\to0$, the inequality $\frac{r_{j-1}}{r_j}\le\frac{\rho_{j-1}}{\rho_j}$ becomes vacuously true, the inequalities for smaller values of $j$ are unaffected, and the inequalities for larger values of $j$ were eliminated at an earlier step. Finally, when we send $\rho_1\to0$, we vacuously satisfy $\frac{r_0}{r_1}\le\frac{\rho_0}{\rho_1}$. Therefore, regarding this limiting procedure as acting on the set $R_\brho$, we wind up with $R_\brho$ converging to the whole set $R$. We thus conclude that
  \begin{equation*}
    \Psi_F(\psi_F)=\int_\Theta\int_R\omega=\int_{R\times\Theta}\omega=\int_T\omega=1,
  \end{equation*}
  where we recall that the orientation of $R$ was selected so that the orientation of $R\times\Theta$ matches the orientation of $T$.

  Now, we show that $\Psi_{F'}(\psi_F)=0$ when $F'\neq F$. We will first consider the situation where $F$ and $F'$ have the same unordered partition of $V$, but the order is different. In this case, up to orientation, we can identify $R$ and $R'$, as well as $\Theta$ and $\Theta'$, since we just permute the vertices and product factors, respectively. With this identification, the computation is exactly the same as above, but we do the iterated limit in the wrong order. So then, there must exist a $j$ where we send $\rho_{j-1}\to0$ before we send $\rho_j\to0$. In this case, the inequality $\frac{r_{j-1}}{r_j}\le\frac{\rho_{j-1}}{\rho_j}$ forces $r_{j-1}$ to be $0$, at least on the open set where $r_j\neq0$. This then forces $R_\brho$ to converge to a set with area zero in $R$, so $\int_{R_\brho}\omega$ converges to zero.

  We now consider the more involved case of showing that $\Psi_{F'}(\psi_F)=0$ when $F$ and $F'$ have different unordered partitions. This part of the proof relies on an inequality whose proof we postpone to Lemmas~\ref{lem:unisolvence1}--\ref{lem:unisolvence3}. As in Definition~\ref{def:dof}, we let $X_1',\dotsc,X_k'$ be vectors tangent to the $\Theta'$ factor of the decomposition $\mr T\cong\mr R'\times\mr\Theta'$, and we let $Y_1',\dotsc,Y_k'$ be the corresponding vectors in $\Theta'$. So then, using Proposition~\ref{prop:psiexplicit} and Lemma~\ref{lem:unisolvence3}, we have
  \begin{multline*}
    \abs{\psi_F(X_1',\dotsc,X_k')}\le\brho^{-F}\abs{\phi_F(X_1',\dotsc,X_k')}\\
    \le \frac{F!}{F'!}\brho^{-F}\abs{\phi_{F'}(X_1',\dotsc,X_k')}=\frac{F!}{F'!}\frac{{\brho'}^{F'}}{\brho^F}\abs{\omega_{F'}(Y_1',\dotsc,Y_k')},
  \end{multline*}
  where we recall that
  \begin{itemize}
  \item $\brho^{F}$ is shorthand for $\prod_j\rho_j^{\abs{V_j}}$,
  \item $F!$ is shorthand for $\prod_jn_j!$ where $n_j=\dim T_j=\abs{V_j}-1$,
  \item $\phi_F$ is shorthand for $\bigwedge_j\phi_{V_j}$, and
  \item $\omega_F$ is shorthand for $\bigwedge_j\omega_j=\brho^{-F}\phi_F$, which we recall is the volume form on $\Theta$ normalized to have integral one.
  \end{itemize}

  Thus, as, top-level forms on $\Theta'$, we have
  \begin{equation*}
    \abs{\left(\psi_F\right)_{\brho'}}\le\frac{F!}{F'!}\frac{{\brho'}^{F'}}{\brho^F}\abs{\omega_{F'}}.
  \end{equation*}
  So then it suffices to show that $\frac{{\brho'}^{F'}}{\brho^F}$ converges to zero as we take the sequence of limits. Specifically, we claim that
  \begin{equation}\label{eq:limfracrho}
    \lim_{\rho_1'\to0}\dotsb\lim_{\rho_{n-k}'\to0}\frac{\prod_{j=0}^{n-k}{\rho'_j}^{\abs{V'_j}}}{\prod_{j=0}^{n-k}\rho_j^{\abs{V_j}}}=0,
  \end{equation}
  provided, as assumed in this case, that the $V_j'$ are not simply a permutation of the $V_j$.

  When we take the limit as $\rho'_{n-k}\to0$, the numerator vanishes to order $\abs{V_{n-k}'}$. Now consider the denominator. If $V_j\subseteq V'_{n-k}$, then $\rho_j\to0$, so the corresponding factor vanishes to order $\abs{V_j}$. Otherwise, $V_j$ contains an $i$ not in $V'_{n-k}$, so $\rho_j$ does not go to zero. So, to count the order of vanishing of the denominator, we count the number of $i\in V'_{n-k}$ such that $i\in V_j\subseteq V'_{n-k}$. So, the order of vanishing of the denominator is at most $\abs{V_{n-k}'}$, with equality if and only if all $V_j$ are either contained in or disjoint from $V'_{n-k}$. Unless this condition holds, the limit is zero.

  Assume then for the sake of contradiction that the sequence of limits \eqref{eq:limfracrho} is nonzero. Then, by the same reasoning, for $1\le j'\le n-k$, all $V_j$ are contained in or disjoint from $V'_{j'}$. (Note that this implies that the same holds for $j'=0$ as well, since $V'_0$ is the complement of the union of the $V'_{j'}$ for $j'\ge1$.) This contradicts the assumptions that $F$ and $F'$ partition $V$ into the same number of sets, but not the same exact sets.
\end{proof}

We now prove the postponed lemmas, starting with a definition needed only for the lemmas.

\begin{definition}
  Let $F$ be a flag with $n+1-k$ subsets $(V_0,\dotsc,V_{n-k})$ as above. Let $W$ be an $(n+1-k)$-element subset of $V$. We say that $W$ is \emph{distinguished} by $F$ if every $V_j$ contains exactly one element of $W$.
\end{definition}

\begin{lemma}\label{lem:unisolvence1}
  Let $F$ and $F'$ be two flags, each with $n+1-k$ subsets. Then
  \begin{equation*}
    \left(\bigwedge_{j=0}^{n-k}\d\rho_j'\right)(X_{\id,0},\dotsc,X_{\id,n-k})=\sum\pm\lambda_W,
  \end{equation*}
  where
  \begin{equation*}
    \rho_j'=\sum_{i\in V_j'}\lambda_i,\qquad X_{\id,j}=\sum_{i\in V_j}\lambda_i\pp{\lambda_i},\qquad \lambda_W=\prod_{i\in W}\lambda_i,
  \end{equation*}
  and the sum is taken over all $W$ that are distinguished by both $F$ and $F'$.
\end{lemma}

\begin{proof}
  Since the $V_j'$ are a partition of $V$, we have that
  \begin{equation*}
    \bigwedge_{j=0}^{n-k}\d\rho'_j=\sum\pm\d\lambda_W,
  \end{equation*}
  where the sum is over all $W$ that are distinguished by $F'$, and, as before, $\d\lambda_W:=\bigwedge_{i\in W}\d\lambda_i$. Computing the same way for the multivector, we have
  \begin{equation*}
    \bigwedge_{j=0}^{n-k}X_{\id,j}=\sum\pm\lambda_W\pp{\lambda_W},
  \end{equation*}
  where the sum is over all $W$ that are distinguished by $F$, and $\pp{\lambda_W}=\bigwedge_{i\in W}\pp{\lambda_i}$.

  The claim follows since the $\d\lambda_W$ are dual to the $\pp{\lambda_W}$.
\end{proof}

\begin{lemma}\label{lem:unisolvence2}
  With notation as above, on $\bR_{>0}^V$, we have
  \begin{equation*}
    \abs{\left(\bigwedge_{j=0}^{n-k}\frac{\d\rho_j'}{\rho_j'}\right)(X_{\id,0},\dotsc,X_{\id,n-k})}\le1
  \end{equation*}
  with equality if and only if the flags $F$ and $F'$ have the same partition, possibly in a different order.
\end{lemma}

\begin{proof}
  The inequality follows from the preceding lemma, along with the fact that
  \begin{equation*}
    \prod_{j=0}^{n-k}\rho'_j=\sum\lambda_W,
  \end{equation*}
  where the sum is taken over all $W$ distinguished by $F'$.

  It is easy to check that equality holds if $F=F'$, since $\d\rho_j(X_{\id,j})=\rho_j$ and $\d\rho_j(X_{\id,l})=0$ for $l\neq j$. Next, if we permute the subsets of the partition, then $\left(\bigwedge_{j=0}^{n-k}\d\rho_j'\right)(X_{\id,0},\dotsc,X_{\id,n-k})$ is unchanged except up to sign, and $\prod_{j=0}^{n-k}\rho_j$ is unchanged. Thus, equality holds if $F$ and $F'$ have the same partition, possibly in a different order.

  Finally, if $F$ and $F'$ have different partitions, then, using the fact that both partitions have $n+1-k$ subsets, there exists a $W$ that is distinguished by $F$ but not $F'$, and vice versa. Therefore, in the preceding lemma, when we sum over $W$ that are distinguished by both $F$ and $F'$, we have strictly fewer terms than when we sum over $W$ that are distinguished by one of the flags. Recalling the positivity assumption in $\bR_{>0}^V$, we conclude that the inequalities are strict.
\end{proof}

\begin{lemma}\label{lem:unisolvence3}
  Let $X_1',\dotsc,X_k'$ be vectors tangent to the $\Theta'$ factor in the decomposition $\mr T\cong\mr R'\times\mr\Theta'$. Then
  \begin{equation*}
    \abs{\frac{\phi_F}{F!}(X_1',\dotsc,X_k')}\le\abs{\frac{\phi_{F'}}{F'!}(X_1',\dotsc,X_k')},
  \end{equation*}
  where, as before, $\phi_F$ is shorthand for $\bigwedge_j\phi_{V_j}$, and similarly for $\phi_{F'}$.
\end{lemma}

\begin{proof}
  Let $X_{\id,0},\dotsc,X_{\id,n-k}$ be the tautological vector fields in the preceding lemmas, ad let $X'_{\id,0},\dotsc,X'_{\id,n-k}$ be the corresponding tautological vector fields for the sets $V'_j$ of the flag $F'$. By definition, we have
  \begin{equation*}
    i_{X_{\id,j}}\d\lambda_{V_j}=\frac{\phi_{V_j}}{n_j!}.
  \end{equation*}
  Additionally, by disjointness of the $V_j$, we have $i_{X_j}\d\lambda_{V_l}=0$ for $l\neq j$. We conclude that then
  \begin{equation*}
    i_{X_{\id,0}}\dotsm i_{X_{\id,n-k}}\d\lambda_V=\pm\frac{\phi_F}{F!}.
  \end{equation*}
  Proceeding similarly to the proof of Lemma~\ref{lem:intwhitney}, we then conclude that
  \begin{equation*}
    \left(\bigwedge_j\frac{\d\rho_j'}{\rho_j'}\right)\wedge\frac{\phi_F}{F!}=\pm\left(\left(\bigwedge_j\frac{\d\rho_j'}{\rho_j'}\right)(X_{\id,0},\dotsc,X_{\id,n-k})\right)\d\lambda_V.
  \end{equation*}
  The preceding lemma tells us that, as top-level forms, the right-hand side has magnitude at most $\d\lambda_V$.
  
  So then, we observe that $\frac{\d\rho_j'}{\rho_j'}$ evaluates to $1$ on $X_{\id,j}'$, evaluates to $0$ on $X_{\id,l}$ for $l\neq j$, and evaluates to $0$ on each of the vectors $X_1',\dotsc,X_k'$ because they are tangent to the $\Theta'$ factor. Thus, evaluating the previous equation on $X'_{\id,0},\dotsc,X'_{\id,n-k},X'_1,\dotsc,X_k'$ and applying the preceding lemma, we obtain
  \begin{equation*}
    \begin{split}
      \abs{\frac{\phi_F}{F!}(X_1',\dotsc,X_k')}&\le\abs{\d\lambda_V(X'_{\id,0},\dotsc,X'_{\id,n-k},X'_1,\dotsc,X_k')},
    \end{split}
  \end{equation*}
  with equality when $F=F'$.
\end{proof}

\subsection{Arrival times of Poisson process ensembles} \label{sec:poisson}
Recall from Proposition~\ref{prop:psiexplicit} that the blow-up Whitney forms can be expressed as
\begin{equation*}
  \psi_F=p_F\brho^{-F}\phi_F=p_F\omega_F,
\end{equation*}
where $p_F$ is the relative volume of $R_\brho\times\Theta$, $\phi_F$ is the wedge product of the Whitney forms for each $V_j$, and $\omega_F$ is the homogenization $\brho^{-F}\phi_F$, which is also the volume form on $\Theta$. Surprisingly, $p_F$ can be interpreted as a probability of a particular order of arrival times of an ensemble of Poisson processes. Poisson processes can be found in many probability textbooks; for the purposes of this paper, a reference that we recommend is \cite[Section~11.1.2]{pishronik2013probability}.

We first recall that if $l$ is the first arrival time of a Poisson process with rate $1$, then $l$ is distributed as $e^{-l}\d l$. Now if we have an ensemble of such processes, indexed by $i\in V$ with $\abs V=n+1$, then $r:=\sum_il_i$ is the $(n+1)$st arrival time of a Poisson process with rate $1$, and the values $\frac{l_i}r$ define a uniformly distributed point in the standard $n$-simplex, independent of $r$. Indeed, applying Lemma~\ref{lem:intwhitney}, we have
\begin{equation}\label{eq:poissonuniform}
  \bigwedge_{i\in V} e^{-l_i}\d l_i=e^{-r}\d l_V=e^{-r}\frac{\d r}r\wedge i_{X_{\id}}\d l_V=\frac{r^n}{n!}e^{-r}\d r\wedge n!i_{X_{\id}}\frac{\d l_V}{r^{\abs V}}.
\end{equation}
Per \cite{pishronik2013probability}, $\frac{r^n}{n!}e^{-r}\d r$ is the probability distribution of the arrival time of the $(n+1)$st particle, and per Proposition~\ref{prop:intwhitney} and scaling, $n!i_{X_{\id}}\frac{\d l_V}{r^{\abs V}}$ is the uniform distribution on the simplex $\sum_il_i=r$ for each $r$.

In particular, if $\chi$ is a random variable depending on the simplex coordinates $\frac{l_i}r$ but not on $r$, then its expected value with respect to the uniform distribution on the simplex is the same as its expected value with respect to the Poisson ensemble distribution. We state this fact more precisely.

\begin{proposition}\label{prop:RequivT}
  If $\chi$ is a dilation-invariant scalar-valued function on $\bR_{>0}^V$, then
  \begin{equation*}
    \int_{\bR_{>0}^V}\chi e^{-r}\d l_V=\int_{T}\chi\omega,
  \end{equation*}
  where $r=\sum_{i\in V} l_i$ and $\omega$ is the volume form on $T$ rescaled to have integral one.
\end{proposition}

\begin{proof}
  Note that $\omega=n!i_{X_{\id}}\frac{\d l_V}{r^{\abs V}}$ is invariant with respect to pullback under the dilation transformation. By the dilation invariance of $\chi$, we therefore have that $\int_{rT}\chi\omega$ does not depend on $r$. Therefore, by Equation~\eqref{eq:poissonuniform},
  \begin{multline*}
    \int_{\bR_{>0}^V}\chi e^{-r}\d l_V=\int_{\bR_{>0}}\left(\frac{r^n}{n!}e^{-r}\d r\left(\int_{rT}\chi\omega\right)\right)\\
    =\left(\int_{\bR_{>0}}\frac{r^n}{n!}e^{-r}\d r\right)\left(\int_{T}\chi\omega\right)=\int_{T}\chi\omega.\qedhere
  \end{multline*}
\end{proof}

We are now ready to interpret the $p_F$ function from Proposition~\ref{prop:psiexplicit} as a probability.

\begin{proposition}
  Let $F$ be a flag, and let $\rho_0,\dotsc,\rho_{n-k}$ be positive numbers. Consider an ensemble of $n+1-k$ Poisson processes, where the $j$th process has rate $\rho_j$. Let $t_j$ be the arrival time of the $\abs{V_j}$th particle of the $j$th Poisson process. Then the relative volume $p_F$ from Proposition~\ref{prop:psiexplicit} is the probability that
  \begin{equation*}
    t_0\le\dotsb\le t_{n-k}.
  \end{equation*}
\end{proposition}
Before we present the proof, we remark that, earlier, the $\rho_j$ were constrained to sum to one, but here there is no such constraint. On the other hand, the probability that $t_0\le\dotsb\le t_{n-k}$ is clearly invariant with respect to dilating $\brho$. So, properly, we are extending $p_F\colon T\to\bR$ from Proposition~\ref{prop:psiexplicit} to the unique dilation invariant function on $\bR_{>0}^V$.

\begin{proof}
  Let $r_j=\rho_jt_j$, so then $t_0\le\dotsb\le t_{n-k}$ is equivalent to
  \begin{equation*}
    \frac{r_0}{\rho_0}\le\dotsb\le\frac{r_{n-k}}{\rho_{n-k}},
  \end{equation*}
  which are exactly the inequalities defining $R_\brho$ in Definition~\ref{def:rs}.

  Since $t_j$ is the arrival time of the $\abs{V_j}$th particle of a Poisson process with rate $\rho_j$, we know that $r_j$ is distributed as the arrival time of the $\abs{V_j}$th particle of a Poisson process with rate $1$. So then, we can also view $r_j=\sum_{i\in V_j}l_i$, where $l_i$ are the \emph{first} arrival times of Poisson processes with rate $1$. (As per \cite{pishronik2013probability}, waiting for the $\abs{V_j}$th particle is equivalent to waiting for the first particle $\abs{V_j}$ consecutive times.)

  So now we can, additionally, view $R_\brho$ as a dilation-invariant subset of the orthant $\mathbb R_{>0}^V$ with coordinates $l_i$. Letting $\chi$ be its characteristic function, by Proposition~\ref{prop:RequivT} we conclude that
  \begin{equation*}
    \int_{R_\brho}e^{-r}\d l_V=\int_{R_\brho\cap T}\omega,
  \end{equation*}
  where, on the left, we view $R_\brho$ as a subset of $\mathbb R_{>0}^V$ as just discussed, and, on the right, we view $R_\brho$ as a subset of $T$ as per Definition~\ref{def:rs}. The right-hand side therefore gives the relative volume of $R_\brho$ (properly its preimage $R_\brho\times\Theta$) in $T$ as per Proposition~\ref{prop:psiexplicit}.

  Meanwhile, the left-hand side is the probability that we land in $R_\brho$ with respect to the probability distribution $e^{-r}\d l_V$. As we have discussed, landing in $R_\brho$ is equivalent to $t_0\le\dotsb\le t_{n-k}$. As we have also discussed, the probability distribution $e^{-r}\d l_V$ gives the arrival times $l_i$ of Poisson processes with rate $1$, which is equivalent to the $r_j$ being $\abs{V_j}$th arrival times of Poisson processes with rate $1$, which, in turn, is equivalent to the $t_j$ being the $\abs{V_j}$th arrival times of Poisson processes with rates $\rho_j$.
\end{proof}

We can apply the above proposition to quickly compute $p_F$ and hence $\psi_F$ using some basic combinatorics and probability.

\begin{example}\label{eg:poissonp}
  Let $F$ be the flag $01\{23\}$. Then $J$ is a three-element set, but to avoid confusion with $V=\{0,1,2,3\}$ we will use $J=\{a,b,c\}$ instead of $J=\{0,1,2\}$. So we have $\rho_a=\lambda_0$, $\rho_b=\lambda_1$, and $\rho_c=\lambda_2+\lambda_3$. So we ask for the probability that $t_a\le t_b\le t_c$, where $t_a$ is the first arrival time of particle $a$, $t_b$ is the first arrival time of particle $b$, and $t_c$ is the \emph{second} arrival time of particle $c$. We can ignore ``excess'' particles, that is, particles of type $a$ after the first, particles of of type $b$ after the first, and particles of type $c$ after the second. In doing so, there are only four particles we care about, and there are three possible sequences they can arrive in so that we have $t_a\le t_b\le t_c$, namely
  \begin{equation*}
    abcc,\qquad acbc,\qquad cabc.
  \end{equation*}

  Now, let us assess the probability of each of these sequences, starting with $abcc$. At any time, the relative probability that the next particle we receive is of type $j$ is $\rho_j$. So, the chance that we receive $a$ first is $\frac{\rho_a}{\rho_a+\rho_b+\rho_c}$. After that, we will ignore any further particles of type $a$, so the chance that $b$ is next is really the chance that $b$ is next among only the options $b$ and $c$, which is $\frac{\rho_b}{\rho_b+\rho_c}$. After that, we will also ignore any further particles of type $b$, so the chance that $c$ is next is $\frac{\rho_c}{\rho_c}=1$, and likewise for the last $c$ particle.

  Reasoning similarly, for $acbc$, the probability that $a$ is first is $\frac{\rho_a}{\rho_a+\rho_b+\rho_c}$, and the probability that $c$ is next (ignoring $a$) is $\frac{\rho_c}{\rho_b+\rho_c}$. But, unlike the previous calculation, we don't ignore $c$ at this point, since we are waiting for two $c$ particles in total. So the chance that $b$ is next is $\frac{\rho_b}{\rho_b+\rho_c}$. As before, since we now ignore both $a$ and $b$, the next particle is $c$ with probability $\frac{\rho_c}{\rho_c}=1$.

  Finally, for $cabc$, the chance that $c$ arrives first is $\frac{\rho_c}{\rho_a+\rho_b+\rho_c}$. The chance that $a$ arrives next is $\frac{\rho_a}{\rho_a+\rho_b+\rho_c}$. We now ignore $a$, so the chance that $b$ arrives next is $\frac{\rho_b}{\rho_b+\rho_c}$. And the last $c$ arrives with probability $1$.

  So, then, putting everything together, we will compute the total probability. To keep the equations manageable, we will use shorthand $\rho_{abc}=\rho_a+\rho_b+\rho_c$ and $\rho_{bc}=\rho_b+\rho_c$.
  \begin{equation*}
    p_F=\frac{\rho_a}{\rho_{abc}}\frac{\rho_b}{\rho_{bc}}+\frac{\rho_a}{\rho_{abc}}\frac{\rho_c}{\rho_{bc}}\frac{\rho_b}{\rho_{bc}}+\frac{\rho_c}{\rho_{abc}}\frac{\rho_a}{\rho_{abc}}\frac{\rho_b}{\rho_{bc}}=\frac{\rho_a\rho_b\rho_c}{\rho_{abc}\rho_{bc}}\left(\frac1{\rho_{abc}}+\frac1{\rho_{bc}}+\frac1{\rho_c}\right).
  \end{equation*}
  So, then, multiplying by $\omega_F=\frac{\lambda_0}{\rho_a}\frac{\lambda_1}{\rho_b}\frac{\phi_{23}}{\rho_c^2}$ (the first two factors are just $1$), we obtain
  \begin{equation*}
    \psi_F=\frac{\lambda_0\lambda_1\phi_{23}}{\rho_{abc}\rho_{bc}\rho_c}\left(\frac1{\rho_{abc}}+\frac1{\rho_{bc}}+\frac1{\rho_c}\right),
  \end{equation*}
  where $\rho_{abc}=\lambda_{0123}=\lambda_0+\lambda_1+\lambda_2+\lambda_3$, $\rho_{bc}=\lambda_{123}=\lambda_1+\lambda_2+\lambda_3$, and $\rho_c=\lambda_{23}=\lambda_2+\lambda_3$.
\end{example}

\begin{example}
  We will now do the same computation where $F$ is the flag $\{01\}\{23\}$, using the notation $J=\{a,b\}$, so $\rho_a=\lambda_0+\lambda_1$, and $\rho_b=\lambda_2+\lambda_3$. Now we want to receive two particles of both types, with the second arrival times $t_a$ and $t_b$ satisfying $t_a\le t_b$. There are once again three possibilities of arrival sequences of the particles we care about:
  \begin{equation*}
    aabb,\qquad abab,\qquad baab.
  \end{equation*}
  Using shorthand $\rho_{ab}=\rho_a+\rho_b$, the probability is
  \begin{equation*}
    p_F=\frac{\rho_a}{\rho_{ab}}\frac{\rho_a}{\rho_{ab}}+\frac{\rho_a}{\rho_{ab}}\frac{\rho_b}{\rho_{ab}}\frac{\rho_a}{\rho_{ab}}+\frac{\rho_b}{\rho_{ab}}\frac{\rho_a}{\rho_{ab}}\frac{\rho_a}{\rho_{ab}}=\frac{\rho_a^2\rho_b}{\rho_{ab}^2}\left(\frac2{\rho_{ab}}+\frac1{\rho_b}\right).
  \end{equation*}
  Multiplying by $\frac{\phi_{01}}{\rho_a^2}\wedge\frac{\phi_{23}}{\rho_b^2}$, we obtain
  \begin{equation*}
    \psi_F=\frac{\phi_{01}\wedge\phi_{23}}{\rho_{ab}^2\rho_b}\left(\frac2{\rho_{ab}}+\frac1{\rho_b}\right),
  \end{equation*}
  where $\rho_{ab}=\lambda_{0123}=\lambda_0+\lambda_1+\lambda_2+\lambda_3$ and $\rho_b=\lambda_{23}=\lambda_2+\lambda_3$.
\end{example}

\begin{example}
  We now compute for the flag $F=0\{123\}$, once again using $J=\{a,b\}$, this time with $\rho_a=\lambda_0$ and $\rho_b=\lambda_1+\lambda_2+\lambda_3$. Now we need one particle of type $a$ and three particles of type $b$, so the possible arrival sequences of the particles we care about, with $t_a\le t_b$, are
  \begin{equation*}
    abbb,\qquad babb,\qquad bbab.
  \end{equation*}
  Using shorthand $\rho_{ab}=\rho_a+\rho_b$, the probability is
  \begin{equation*}
    p_F=\frac{\rho_a}{\rho_{ab}}+\frac{\rho_b}{\rho_{ab}}\frac{\rho_a}{\rho_{ab}}+\frac{\rho_b}{\rho_{ab}}\frac{\rho_b}{\rho_{ab}}\frac{\rho_a}{\rho_{ab}}=\frac{\rho_a\rho_b^2}{\rho_{ab}}\left(\frac1{\rho_{ab}^2}+\frac1{\rho_{ab}\rho_b}+\frac1{\rho_b^2}\right).
  \end{equation*}
  Multiplying by $\frac{\lambda_0}{\rho_a}\frac{\phi_{123}}{\rho_b^3}$, we obtain
  \begin{equation*}
    \psi_F=\frac{\lambda_0\phi_{123}}{\rho_{ab}\rho_b}\left(\frac1{\rho_{ab}^2}+\frac1{\rho_{ab}\rho_b}+\frac1{\rho_b^2}\right),
  \end{equation*}
  where $\rho_{ab}=\lambda_{0123}=\lambda_0+\lambda_1+\lambda_2+\lambda_3$ and $\rho_b=\lambda_{123}=\lambda_1+\lambda_2+\lambda_3$.
\end{example}

The Poisson process framework can yield intuitive proofs of statements about blow-up Whitney forms. We begin with the claim that the space of blow-up Whitney forms contains the ordinary Whitney forms, which is shown in~\cite[Proposition 11.1]{brasselet1991simplicial} by direct algebraic calculation.

\begin{proposition} \label{prop:contains}
  The space of blow-up Whitney forms contains the usual Whitney forms.
\end{proposition}

\begin{proof}
  Let $W$ be a subset of $V$ of size $k+1$. We aim to show that $\phi_W$ can be expressed as a sum of blow-up Whitney forms. To that end, consider the set of flags $F$ for which $V_0=W$, and $V_1,\dotsc,V_{n-k}$ each have one element; there are $(n-k)!$ such flags. We claim that $\phi_W$ is the sum of $\psi_F$ over all of these flags.

  To that end, in light of Proposition~\ref{prop:psiexplicit}, observe that, for $1\le j\le n-k$, since $V_j$ is a singleton set, we have $\rho_j^{-\abs{V_j}}\phi_{V_j}=\lambda_{i_j}^{-1}\lambda_{i_j}=1$, where $i_j$ denotes the sole element of $V_j$. As such, for each flag, we have $\psi_F=p_F\rho_0^{-\abs W}\phi_W$. So then, it remains to show that the $p_F$ sum to $\rho_0^{\abs W}$ (with the normalization $\rho=\sum_{i\in V}\lambda_i=1$; otherwise $\left(\rho_0/\rho\right)^{\abs W}$).

  Since these flags all have the same unordered partition, we are working with an ensemble of Poisson processes, where we wait for the $\abs{W}$-th particle of a Poisson process with rate $\rho_0$, and we wait for the first particle of $n-k$ Poisson processes each with rate $\lambda_i$ for $i\notin W$. For a particular flag $F$ in the above set, we want the $\abs{W}$-th particle of Poisson process $0$ to arrive first, and then we want the first particles of the other Poisson processes to arrive in a particular order. However, since we sum over all of these orders, we conclude that the sum of the $p_F$ is simply the probability that the $\abs{W}$-th particle of Poisson process $0$ arrives first, that is, before we receive any particles from any of the other $n-k$ Poisson processes.

  Letting $\rho=\sum_{i\in V}\lambda_i$, which we may, optionally, normalize to one, the probability that the first particle that we receive is from the Poisson process $0$ is $\frac{\rho_0}\rho$. After receiving it, we wait for the next particle, which likewise has probability $\frac{\rho_0}\rho$ of coming from Poisson process $0$. Continuing on, we find that the probability that the first $\abs W$ particles that we received all came from Poisson process $0$ is $\bigl(\frac{\rho_0}\rho\bigr)^{\abs W}$, as desired.
\end{proof}

Another result that follows easily from the Poisson process perspective is the relationship between the blow-up Whitney forms on simplices of different dimension. We recall that $\psi_F=p_F\omega_F$.

\begin{proposition}\label{prop:reducedim}
  Let $F=(V_0,\dotsc,V_{n-k})$ be a flag on $V$. Removing the last subset of the partition, we obtain a flag $F':=(V_0,\dotsc,V_{n-k-1})$ on the smaller set $V':=V\setminus V_{n-k}$. Then
  \begin{equation*}
    p_{F'}=\lim_{\rho_{n-k}\to0}p_F.
  \end{equation*}
\end{proposition}

\begin{proof}
  Recall that $p_F$ is the probability that $t_0\le\dotsb\le t_{n-k}$, where $t_j$ is the arrival time of the $\abs{V_j}$th particle of a Poisson process with rate $\rho_j$. So, as the rate $\rho_{n-k}$ goes to zero, the arrival time $t_{n-k}$ goes to infinity. More precisely, for any $T$, the probability that $t_{n-k}\ge T$ goes to $1$. So then the probability that $t_0\le\dotsb\le t_{n-k}$ approaches the probability that $t_0\le\dotsb\le t_{n-k-1}$, which is precisely $p_{F'}$.
\end{proof}

The Poisson framework also yields integral formulas similar to those that appear in \cite[p.~1028]{brasselet1991simplicial} and \cite[Lemma 1]{bendiffalah1995shadow}. The main idea is that $t_0\le\dotsb\le t_{n-k}$ is equivalent to $\frac{t_{j-1}}{t_j}\le1$ for $1\le j\le n-k$, so we can express the probability $p_F$ as an integral over the $t_j$ coordinates.

\begin{proposition}
  We have the following integral formulas for $p_F$.
  \begin{align*}
    p_F&=\int_0^\infty\int_0^1\dotsi\int_0^1\frac{\br^F}{F!}e^{-r}\tfrac1{t_0}\d(\tfrac{t_0}{t_1})\wedge\dotsb\wedge\d(\tfrac{t_{n-k-1}}{t_{n-k}})\wedge\d t_{n-k}\\
       &=\int_0^\infty\frac{r^ne^{-r}}{n!}\int_0^{\frac{\rho_{n-k-1}}{\rho_{n-k}}}\dotsi\int_0^{\frac{\rho_0}{\rho_1}}\frac{n!}{F!}\left(\frac{\br}{r}\right)^F\frac{r_{n-k}}{r_0}\d(\tfrac{r_0}{r_1})\wedge\dotsb\wedge\d(\tfrac{r_{n-k-1}}{r_{n-k}})\wedge\d r\\
       &=\int_{-\infty}^\infty\frac{r^{n+1}e^{-r}}{n!}\int_{-\infty}^{\ln\frac{\rho_{n-k-1}}{\rho_{n-k}}}\dotsi\int_{-\infty}^{\ln\frac{\rho_0}{\rho_1}}\frac{n!}{F!}\left(\frac{\br}{r}\right)^F\d\ln\tfrac{r_0}{r_1}\wedge\dotsb\wedge\d\ln\tfrac{r_{n-k-1}}{r_{n-k}}\wedge\d\ln r.
  \end{align*}
  Here, as before, $n_j=\abs{V_j}-1$, $r_j=\rho_jt_j$, $r=\sum r_j$, and $\br^F$ and $F!$ are shorthand for $\prod r_j^{\abs{V_j}}$ and $\prod n_j!$, respectively.
  
  Note, in the first line, that the bounds of integration are constants and the integrand varies with $\brho$, whereas in the second and third lines the bounds depend on $\brho$ but the integrand does not. In all cases, the bounds do not depend on the integration variables, so we are integrating over an (infinite) rectangle.

  Note also that, in the second and third lines, the inner integrand is dilation-invariant, and the outer integral integrates to one, so, similarly to Proposition~\ref{prop:RequivT}, we can reduce to an integral over the subset $R_\brho$ of the simplex defined by $r=1$.
\end{proposition}

\begin{proof}
  Recall for Poisson arrival times that the $t_j$ are independent with probability distributions
  \begin{equation*}
    \rho_j^{\abs{V_j}}\frac{t_j^{n_j}}{n_j!}e^{-\rho_jt_j}\d t_j=\frac{r_j^{\abs{V_j}}}{n_j!}e^{-r_j}\d\ln t_j.
  \end{equation*}
  Thus, the full probability distribution is
  \begin{equation*}
    \frac{\br^F}{F!}e^{-r}\d\ln t_0\wedge\dotsb\wedge\d\ln t_{n-k}.
  \end{equation*}
  Next, by wedging right to left, we can verify that
  \begin{multline*}
    (\d\ln t_0-\d\ln t_1)\wedge\dotsb\wedge(\d\ln t_{n-k-1}-\d\ln t_{n-k})\wedge\d\ln t_{n-k}\\
    =\d\ln t_0\wedge\dotsb\wedge\d\ln t_{n-k}.
  \end{multline*}
  Since $t_0\le\dotsb\le t_{n-k}$ is equivalent to $\ln\frac{t_{j-1}}{t_j}\le0$, we conclude that
  \begin{equation*}
    p_F=\int_{-\infty}^{\infty}\int_{-\infty}^0\dotsi\int_{-\infty}^0\frac{\br^F}{F!}e^{-r}\d\ln\tfrac{t_0}{t_1}\wedge\dotsb\wedge\d\ln\tfrac{t_{n-k-1}}{t_{n-k}}\wedge\d\ln t_{n-k}.
  \end{equation*}
  Changing variables, we can also write
  \begin{equation*}
    p_F=\int_0^\infty\int_0^1\dotsi\int_0^1\frac{\br^F}{F!}e^{-r}\tfrac1{t_0}\d(\tfrac{t_0}{t_1})\wedge\dotsb\wedge\d(\tfrac{t_{n-k-1}}{t_{n-k}})\wedge\d t_{n-k},
  \end{equation*}
  which is the first equation in the proposition. Changing variables again, we can write
  \begin{equation*}
    p_F=\int_0^\infty\int_0^{\frac{\rho_{n-k-1}}{\rho_{n-k}}}\dotsi\int_0^{\frac{\rho_0}{\rho_1}}\frac{\br^F}{F!}e^{-r}\tfrac1{r_0}\d(\tfrac{r_0}{r_1})\wedge\dotsb\wedge\d(\tfrac{r_{n-k-1}}{r_{n-k}})\wedge\d r_{n-k}.
  \end{equation*}
  Observe that
  \begin{equation*}
    r=r_{n-k}(1+\tfrac{r_{n-k-1}}{r_{n-k}}+\tfrac{r_{n-k-1}}{r_{n-k}}\tfrac{r_{n-k-2}}{r_{n-k-1}}+\dotsb+\tfrac{r_{n-k-1}}{r_{n-k}}\dotsm\tfrac{r_0}{r_1}).
  \end{equation*}
  We thus have that
  \begin{equation*}
    \d(\tfrac{r_0}{r_1})\wedge\dotsb\wedge\d(\tfrac{r_{n-k-1}}{r_{n-k}})\wedge\d(\tfrac{r_{n-k}}r)=0,
  \end{equation*}
  because it is an $(n-k+1)$-form in the $n-k$ variables $\frac{r_{j-1}}{r_j}$. We can conclude that
  \begin{equation*}
    \d(\tfrac{r_0}{r_1})\wedge\dotsb\wedge\d(\tfrac{r_{n-k-1}}{r_{n-k}})\wedge\d r_{n-k}=\tfrac{r_{n-k}}r\d(\tfrac{r_0}{r_1})\wedge\dotsb\wedge\d(\tfrac{r_{n-k-1}}{r_{n-k}})\wedge\d r,
  \end{equation*}
  and so
  \begin{equation*}
    p_F=\int_0^\infty\int_0^{\frac{\rho_{n-k-1}}{\rho_{n-k}}}\dotsi\int_0^{\frac{\rho_0}{\rho_1}}\frac{\br^F}{F!}\frac{e^{-r}}r\frac{r_{n-k}}{r_0}\d(\tfrac{r_0}{r_1})\wedge\dotsb\wedge\d(\tfrac{r_{n-k-1}}{r_{n-k}})\wedge\d r,
  \end{equation*}
  which yields the second and third equations in the proposition.
\end{proof}

\section{The complex of blow-up Whitney forms and its cohomology} \label{sec:complex}

\subsection{The exterior derivative}
In this subsection, we show that the blow-up Whitney forms form a complex with respect to the exterior derivative. 

\begin{theorem}\label{thm:d}
  \begin{equation*}
    d\psi_F=\pm\sum_{j=1}^{n-k}\psi_{F_j},
  \end{equation*}
  where $F_j$ is the flag constructed from $F$ by replacing the two partition elements $V_{j-1}$ and $V_j$ with their union.
\end{theorem}

This result appears in~\cite[Corollary 2.2]{brasselet1991simplicial}. Here, we give a perspective on the proof from the vantage point of Poisson processes. We want to emphasize that, although the Poisson process perspective is new, the proof follows roughly the same key ideas as in \cite{brasselet1991simplicial}, except that we work on the orthant squared $\bR^V_{\ge0}\times\bR^V_{>0}$ whereas they work on the simplex squared $\mr T\times T$. In particular, the arrival time subset $A_F$ in Definition~\ref{def:af} corresponds to the incidence variety $D_\sigma$ in \cite{brasselet1991simplicial} (with $\sigma$ corresponding to $F$), and Proposition~\ref{prop:weakpsi} corresponds to \cite[Proposition 2.1]{brasselet1991simplicial}.

We first establish and recall some notation. First, note that we have actually been using two copies of $\mathbb R_{\ge0}^V$. One of them, with coordinates $l_i$, we used in the section on Poisson processes, and there we treated the $\lambda_i$ as parameters giving the rates of the Poisson processes. But, of course, the $\lambda_i$ are also coordinates of a copy of $\mathbb R_{\ge0}^V$, and the blow-up Whitney forms are written in terms of the $\lambda_i$. As before, we let $\phi_{V_j}$ be Whitney forms in terms of the $\lambda_i$, and we let $\omega_{V_j}=\frac{\phi_{V_j}}{\rho_j^{\abs{V_j}}}$, and we have $\omega_F=\bigwedge_j\omega_{V_j}$. Recall that the blow-up Whitney forms are $\psi_F=p_F\omega_F$, where, as before, $p_F$ is the probability of a particular order of certain arrival times of Poisson processes. Note that the expression $\omega_{V_j}$ is defined on all of $\bR_{>0}^V$ and is dilation-invariant (via pullback). Therefore, instead of taking the perspective that $\psi_F$ is a form on $T$, we take the perspective that it is a dilation-invariant form on $\bR_{>0}^V$.

With that all in mind, we will now work on the space $\bR_{\ge0}^V\times\bR_{>0}^V$, where the first copy has coordinates $l_i$, and the second copy has coordinates $\lambda_i$. When we need to disambiguate, we will write $\bR_{\ge0}^{V,\bl}$, and $\bR_{>0}^{V,\blambda}$. As before, $r_j=\sum_{i\in V_j}l_i$, and $\rho_j=\sum_{i\in V_j}\lambda_i$, and likewise $r$ and $\rho$ are those sums over all of $V$. As before, we will let $t_j=\frac{r_j}{\rho_j}$ for $j\in J$, but now we view $t_j$ as a scalar-valued function on $\bR_{\ge0}^V\times\bR_{>0}^V$. We will also let $s_i=\frac{l_i}{\lambda_i}$ for $i\in V$; this value can likewise be interpreted as an arrival time, but we use a different letter to avoid confusion. Note that the denominator in these expressions is the reason we restrict the $\blambda$ space to the strictly positive orthant.

In this setting, we have a simple weak characterization of the blow-up Whitney forms $\psi_F$. We begin with a preliminary definition.

\begin{definition}\label{def:af}
  For a flag $F$, let $A_F$ be the \emph{arrival time subset} of $\bR_{\ge0}^V\times\bR_{>0}^V$, defined by the inequalities $s_i\le s_{i'}$ for $i\in V_j$, $i'\in V_{j'}$, whenever $j\le j'$. Note in particular that $s_i=s_{i'}$ if $i$ and $i'$ are in the same subset $V_j=V_{j'}$. Moreover, this common value is $t_j$, so we have $t_0\le\dotsb\le t_{n-k}$ as before.
\end{definition}

Observe that, due to the presence of these equalities, the space $A_F$ is $(2(n+1)-k)$-dimensional.

\begin{proposition}\label{prop:weakpsi}
  Let $F$ be a flag and let $\alpha$ be a suitable (for instance, smooth and compactly supported) test $(n+1-k)$-form on $\bR_{>0}^{V,\blambda}$, which we can view as a form on $\bR_{\ge0}^V\times\bR_{>0}^V$ that only depends on $\blambda$. Then
  \begin{equation*}
    \int_{A_F}e^{-r}\d l_V\wedge\alpha=\int_{\bR_{>0}^{V,\blambda}}\psi_F\wedge\alpha,
  \end{equation*}
  presuming a suitable choice of orientation of $A_F$ (without which the equation may need a negative sign).
\end{proposition}

\begin{proof}
  We first observe that the $n+1$ coordinates $\lambda_i$, along with the $n+1-k$ functions $r_j$, form a coordinate system on the the $(2(n+1)-k)$-dimensional space $A_F$. Indeed, given the $\lambda_i$ and the $r_j$, we can compute the $\rho_j$ and hence the $t_j=\frac{r_j}{\rho_j}$. Given that we are on $A_F$, the $s_i$ are equal to the $t_j$ for $i\in V_j$, and so we can recover the $l_i=\lambda_i s_i$. In fact, this yields a coordinate system on the larger set where the $s_i$ are equal to the $t_j$ for $i\in V_j$, and restricting this coordinate system to $A_F$ amounts to restricting $\br$ to $R_\brho$ for each value of $\blambda$, because we recall that $\br\in R_\brho$ is equivalent to $t_0\le\dotsb\le t_{n-k}$.

  Our next task is to convert the $(n+1)$-form $e^{-r}\d l_V$ to this coordinate system. We will let $X_{\id}$ be the tautological vector field on the whole space (both the $l_i$ and $\lambda_i$ coordinates), and we observe that $i_{X_{\id}}\d l_i=s_i i_{X_{\id}}\d\lambda_i$.

  We now focus on a particular value of $j$, recalling that all of the $s_i$ for $i\in V_j$ are equal to a single value, which we denote by $t_j$. We compute as in Lemma~\ref{lem:intwhitney} and using the above fact that $i_{X_{\id}}\d l_i=s_i i_{X_{\id}}\d\lambda_i$. We obtain
  \begin{multline*}
    e^{-r_j}\d l_{V_j}=e^{-r_j}\frac{dr_j}{r_j}\wedge i_{X_{\id}}(\d l_{V_j})
    =e^{-r_j}\frac{dr_j}{r_j}\wedge t_j^{\abs{V_j}}i_{X_{\id}}\d\lambda_{V_j}\\
    =e^{-r_j}\frac{r_j^{n_j}}{n_j!}\d r_j\wedge\rho_j^{-\abs{V_j}}\phi_{V_j},
  \end{multline*}
  where, as before $n_j=\abs{V_j}-1$. We recognize the first factor as the probability distribution of the arrival time $r_j$ of the $\abs{V_j}$th particle under a Poisson process with rate $1$, and we recognize the second factor as $\omega_j$, the dilation-invariant Whitney form on $T_{V_j}$. So then, wedging everything together over $j$, we have
  \begin{equation*}
    e^{-r}\d l_V=\pm e^{-r}\frac{\br^{F}}{F!}\d r_J\wedge\omega_F,
  \end{equation*}
  where, as before, $\br^F$, $F!$, $\d r_J$, and $\omega_F$ are shorthands for the respective products.

  So now we integrate by Fubini's theorem. Recall that, for fixed $\blambda$, because the $r_j$ are distributed as $\abs{V_j}$th arrival times of Poisson processes with rate $1$, the $t_j=\frac{r_j}{\rho_j}$ are distributed as $\abs{V_j}$th arrival times of Poisson processes with rate $\rho_j$. Also, recall that, for fixed $\blambda$, choosing $\br$ to land us in $A_F$ is equivalent to $\br\in R_\brho$ (viewing $R_\brho$ as a dilation-invariant subset of $\bR_{>0}^J$), which we know occurs with probability $p_F$ under this probability distribution. Hence, assuming a proper choice of orientation of $A_F$ and recalling that $\alpha$ only depends on $\blambda$, we have
  \begin{multline*}
    \int_{A_F}e^{-r}\d l_V\wedge\alpha=\int_{\bR_{>0}^{V,\blambda}}\left(\int_{R_\brho}e^{-r}\frac{\br^{F}}{F!}\d r_J\right)\omega_F\wedge\alpha\\
    =\int_{\bR_{>0}^{V,\blambda}}p_F\omega_F\wedge\alpha=\int_{\bR_{>0}^{V,\blambda}}\psi_F\wedge\alpha,
  \end{multline*}
  as desired.
\end{proof}

\begin{proof}[Proof of Theorem~\ref{thm:d}]
  We proceed by taking the weak exterior derivative and applying Stokes's theorem. We lower the degree of $\alpha$ by one, taking it to be a test $(n-k)$-form on $\bR_{>0}^{V,\blambda}$, and as before we require compact support away from the boundary of the orthant. So then, by the preceding proposition, we have
  \begin{multline*}
    \int_{\bR_{>0}^{V,\blambda}}\d\psi_F\wedge\alpha=\pm\int_{\bR_{>0}^{V,\blambda}}\psi_F\wedge\d\alpha=\pm\int_{A_F}e^{-r}\d l_V\wedge\d\alpha\\
    =\pm\int_{A_F}\d\left(e^{-r}\d l_V\wedge\alpha\right)=\pm\int_{\partial A_F}e^{-r}\d l_V\wedge\alpha.
  \end{multline*}
  Here we used the fact that $d(e^{-r}\d l_V)=0$ because $e^{-r}$ only depends on $\bl$, so we are effectively taking the derivative of a top-level form on $\bR_{\ge0}^{V,\bl}$.

  So all that remains is to relate the boundary of $A_F$ to the $A_{F_j}$. The set $A_F$ is defined by the inequalities
  \begin{equation*}
    t_0\le\dotsb\le t_{n-k},
  \end{equation*}
  along with the condition that for $i\in V_j$ we have $s_i$ all equal to $t_j$. So then the boundary of $A_F$ is comprised of components which are given by turning one of the defining inequalities $t_{j-1}\le t_j$ into an equality $t_{j-1}=t_j$ (while maintaining all other inequalities and equalities). We see that doing so precisely gives us $A_{F_j}$, the arrival time set of the flag $F_j$ obtained by replacing $V_{j-1}$ and $V_j$ with their union.

  Additionally, we must consider the boundary component $t_0=0$, but note that this implies that $r_0=0$, which means that some of the $l_i$ are zero, which means some of the $\d l_i$ are zero, which means that $e^{-r}\d l_V=0$. We must also consider the ``boundary component'' as $t_{n-k}\to\infty$, but note that this requires either $r_{n-k}\to\infty$, which means $e^{-r}\to0$, or it requires $\rho_{n-k}\to0$, in which case $\alpha$ vanishes due to the assumption of compactness away from the boundary of the orthant.

  In conclusion, we have
  \begin{equation*}
    \int_{\bR_{>0}^{V,\blambda}}\d\psi_F\wedge\alpha=\sum_{j=1}^{n-k}\pm\int_{A_{F_j}}e^{-r}\d l_V\wedge\alpha=\sum_{j=1}^{n-k}\pm\int_{\bR_{>0}^{V,\blambda}}\psi_{F_j}\wedge\alpha,
  \end{equation*}
  for all $\alpha$, as desired.
\end{proof}

\subsection{Blowing up and cohomology} \label{sec:cohomology}
Via integration, the standard Whitney forms on a simplex are dual to the faces of the simplex, which means that the complex of Whitney forms is isomorphic to the complex of simplicial cochains. Consequently, the cohomology of these complexes are the same. In particular, a single simplex is homeomorphic to a ball, so the cohomology vanishes except for the constants in degree zero.

The goal of this section is to similarly conclude that the cohomology of the blow-up Whitney form complex vanishes except for the constants in degree zero. A priori, however, the task seems foolish: the blow-up Whitney forms are only smooth on the interior of $T$, so evaluation on faces is not defined, unless we follow the specific limiting procedure discussed in the degrees of freedom section. Our solution to this problem is to \emph{blow up} the simplex, obtaining a \emph{blow-up space} $\tl T$. Combinatorially, the blow-up $\tl T$ is equivalent to a well-understood polytope called the permutahedron; see Figure~\ref{fig:blowuptet}. But its main advantage for us comes from analysis: The blow-up simplex comes with a map $\pi\colon\tl T\to T$ that is a diffeomorphism on the interior, so we can pullback the blow-up Whitney forms $\psi_F$ on $T$ to get forms $\tl\psi_F$ on $\tl T$. Notably, whereas the $\psi_F$ are only smooth on the interior of $T$, the blow-up simplex ``desingularizes'' the boundary of $T$, and the forms $\tl\psi_F$ end up being smooth on $\tl T$ up to and including the boundary. Moreover, integrating the $\tl\psi_F$ over the faces of $\tl T$ ends up being equivalent to the limiting procedure discussed in the degrees of freedom section.

\begin{figure}
  \centering
  \includegraphics[scale=0.4,clip=true,trim=0in 0in 0in 0in]{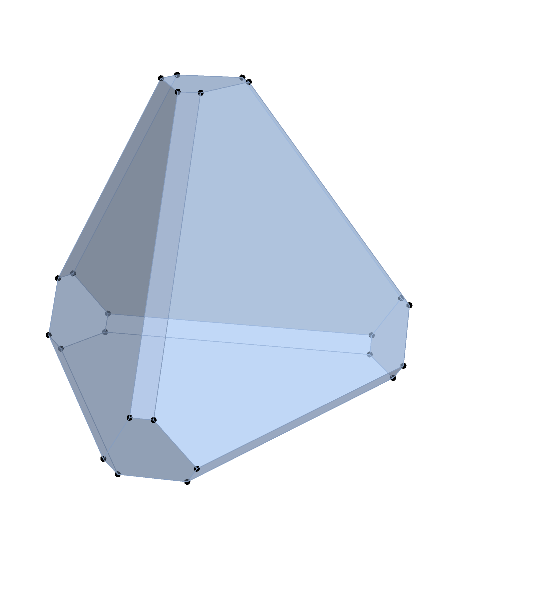}
  \caption{The blow-up of a tetrahedron is combinatorially equivalent to a permutahedron of order 4.  The 24 vertices of the permutahedron correspond to the 24 flags $0123, 0132, 0213, \dots,3210$ (where, for instance, $0123$ is shorthand for $(\{0\},\{1\},\{2\},\{3\})$).}
  \label{fig:blowuptet}
\end{figure}

In short, by lifting to the blow-up $\tl T$, the blow-up Whitney forms become smooth including on the boundary, and they are dual to the faces via integration. So, we get the same story for blow-up Whitney forms that we had for the regular Whitney forms, just with $\tl T$ in place of $T$ and cellular cohomology in place of simplicial cohomology.

We proceed with some intuition, followed by the definition, and then followed by examples.

\begin{intuition}[via barycentric subdivision]
  The blow-up simplex $\tl T$ can be thought of as the configuration space of barycentric subdivisions of $T$, where degenerate subdivisions are allowed. Normally, one performs barycentric subdivision by choosing a point $x$ in the interior of $T$, and then constructing rays, half-planes, etc., that join the faces of $T$ and $x$. But what if we allow degenerate subdivisions, where $x$ is on the boundary of $T$, for example, at a vertex? Then the barycentric subdivision is generally no longer determined by $x$. On the other hand, for each face $K$ of $T$, a barycentric subdivision of $T$ also induces a barycentric subdivision of $K$, and hence gives a point $x^K$ in $K$. The collection of points $x^K$ determines the barycentric subdivision even in the degenerate case, but there are relations between the $x^K$ that must be accounted for.
\end{intuition}

\begin{intuition}[via Poisson processes]
  As in Section~\ref{sec:poisson}, consider a collection of Poisson processes with rates $\lambda_i$. Then, for example, the probability that we receive particle $0$ before particle $2$ is $\frac{\lambda_0}{\lambda_0+\lambda_2}$, and likewise the probability that we receive particle $0$ first among particles $0$, $1$, and $2$, is $\frac{\lambda_0}{\lambda_0+\lambda_1+\lambda_2}$. But what if we allow some rates to be infinitesimal relative to others? For example, perhaps source $0$ emits a particle every second on average, whereas source $1$ emits three particles every million years on average, and source $2$ emits two particles every million years on average. Then, effectively,
  \begin{equation}\label{eq:degeneraterates}
    \frac{\lambda_0}{\lambda_0+\lambda_1+\lambda_2}=1,\qquad\frac{\lambda_1}{\lambda_0+\lambda_1+\lambda_2}=0,\qquad\frac{\lambda_2}{\lambda_0+\lambda_1+\lambda_2}=0.
  \end{equation}
  On the other hand, if we consider only particles $1$ and $2$, it is clear that
  \begin{equation*}
    \frac{\lambda_1}{\lambda_1+\lambda_2}=\frac35.
  \end{equation*}
  However, as written above, $\frac{\lambda_1}{\lambda_1+\lambda_2}$ does not make sense because Equation~\eqref{eq:degeneraterates} implies $\lambda_1=\lambda_2=0$. The solution is to treat all of these fractions as independent variables, and then use constraint equations to enforce the relationships between them.
\end{intuition}

\begin{definition}\label{def:blowupbary}
  Given a simplex $T$, let $K\le T$ denote a face of $T$, and let $V_K$ denote the vertices of $K$. Consider then the high-dimensional first orthant
  \begin{equation*}
    \prod_{K\le T}\bR_{\ge0}^{V_K},
  \end{equation*}
  Denote the coordinates of this space by $\lambda_i^K$, where $K\le T$ and $i\in V_K$.

  Let $\tilde T$ be the subset of this space carved out by the equations
  \begin{equation*}
    \sum_{i\in V_K}\lambda_i^K=1
  \end{equation*}
  for all $K$ and  
  \begin{equation*}
    \lambda_i^H=\rho_K^H\lambda_i^K
  \end{equation*}
  for all $K\le H$ and $i\in V_K$, where $\rho_K^H=\sum_{i\in V_K}\lambda_i^H$. Note that the value of $\rho_K^H$ also follows by summing the above equation over $i\in V_K$, and that the equations are vacuous when $H=K$ since $\rho_K^K=1$.

  There is a natural projection $\pi\colon\tilde T\to T$ where we take the barycentric coordinates of the image point in $T$ to simply be $\lambda_i=\lambda_i^T$.
\end{definition}

\begin{example}
  Consider a two-dimensional triangle $T$ with vertices $\{0,1,2\}$. Let $P_i$ denote the zero-dimensional face corresponding to that vertex, and let $E_i$ denote the opposite edge. Letting $\lambda_i$ denote $\lambda_i^T$, the coordinates we are working with are
  \begin{gather*}
    (\lambda_0,\lambda_1,\lambda_2),\\
    (\lambda_1^{E_0},\lambda_2^{E_0}),(\lambda_0^{E_1},\lambda_2^{E_1}),(\lambda_0^{E_2},\lambda_1^{E_2}),\\
    (\lambda_0^{P_0}),(\lambda_1^{P_1}),(\lambda_2^{P_2}).
  \end{gather*}
  Since barycentric coordinates sum to one, we have that $\lambda_i^{P_i}=1$, so we do not really need to think about these coordinates. Consequently, taking $K=P_i$ in the constraint $\lambda_i^H=\rho_K^H\lambda_i^K$, we just have $\rho_K^H=\lambda_i^H$ and $\lambda_i^K=1$, so the constraint is vacuously satisfied. So, the only nontrivial constraints come from $E_i\le T$. For example, taking $K=E_0\le H=T$, we get
  \begin{align*}
    \lambda_1&=(\lambda_1+\lambda_2)\lambda_1^{E_0},&
                                                      \lambda_2&=(\lambda_1+\lambda_2)\lambda_2^{E_0}.
  \end{align*}
  So, we see that, provided that $\lambda_1+\lambda_2\neq0$, we just have
  \begin{align*}
    \lambda_1^{E_0}&=\frac{\lambda_1}{\lambda_1+\lambda_2},&\lambda_2^{E_0}&=\frac{\lambda_2}{\lambda_1+\lambda_2},
  \end{align*}
  consistent with the fact that any interior point of $T$ has only one preimage in $\tilde T$, so all coordinates should be determined by $\lambda_0$, $\lambda_1$, and $\lambda_2$. More precisely, $\lambda_1^{E_0}$ and $\lambda_2^{E_0}$ are the coordinates of the unique point on edge $E_0$ that is on the ray from $P_0$ through the point $x$ with coordinates $(\lambda_0,\lambda_1,\lambda_2)$. On the other hand, if $\lambda_1=\lambda_2=0$, namely, at vertex $P_0$, then the constraint equations are vacuously satisfied, so $\lambda_1^{E_0}$ and $\lambda_2^{E_0}$ are free to represent any point on the edge $E_0$, corresponding to the fact that if $x=P_0$, then the ray through $P_0$ and $x$ is undetermined.
\end{example}

\begin{example}
  Now consider a three-dimensional tetrahedron $T$ with vertices $\{0,1,2,3\}$. As before, we will use the notation $P_i$ for vertices and $F_i$ for the opposite faces, but now we must also use $E_{ij}$ for the edge joining $i$ and $j$. As before, we have constraint equations that give, for example, for $F_0\le T$, that
  \begin{equation*}
    \lambda_i^{F_0}=\frac{\lambda_i}{\lambda_1+\lambda_2+\lambda_3}
  \end{equation*}
  for $i\in\{1,2,3\}$ when the denominator is nonzero, and otherwise the constraint is vacuously satisfied. Likewise, we have, for example, for $E_{12}\le T$, that
  \begin{equation*}
    \lambda_i^{E_{12}}=\frac{\lambda_i}{\lambda_1+\lambda_2}
  \end{equation*}
  for $i\in\{1,2\}$ when the denominator is nonzero.

  But, now, we also have constraints coming from, for instance, $E_{12}\le F_0$, which yield
  \begin{equation}\label{eq:constraintEF}
    \lambda_i^{E_{12}}=\frac{\lambda_i^{F_0}}{\lambda_1^{F_0}+\lambda_2^{F_0}}.
  \end{equation}
  for $i\in\{1,2\}$, again provided the denominator is nonzero. On the interior of $T$, this equation is automatically satisfied: it is simply the statement that
  \begin{equation*}
    \frac{\lambda_i}{\lambda_1+\lambda_2}=\frac{\frac{\lambda_i}{\lambda_1+\lambda_2+\lambda_3}}{\frac{\lambda_1}{\lambda_1+\lambda_2+\lambda_3}+\frac{\lambda_2}{\lambda_1+\lambda_2+\lambda_3}}.
  \end{equation*}
  But what if $\lambda_1=\lambda_2=\lambda_3=0$, so we are at vertex $P_0$? Then, the constraint equations $\lambda_i=(\lambda_1+\lambda_2+\lambda_3)\lambda_i^{F_0}$ and $\lambda_i=(\lambda_1+\lambda_2)\lambda_i^{E_{12}}$ are vacuously satisfied, so the $\lambda_i^{F_0}$ and $\lambda_i^{E_{12}}$ are free to roam. So now Equation~\eqref{eq:constraintEF} imposes a nontrivial constraint, saying that the choice of $\lambda_i^{E_{12}}$ is determined by the $\lambda_i^{F_0}$. Unless, of course, we also have $\lambda_1^{F_0}=\lambda_2^{F_0}=0$, in which case we are free to choose $\lambda_i^{E_{12}}$.

  As before, the $\lambda_i^{F_0}$ can be thought of as representing a ray from $P_0$ by specifying where it hits the opposite face $F_0$. Likewise, the $\lambda_i^{E_{12}}$ can be thought of as representing a half-plane emanating from the edge $E_{03}$ by specifying where it hits the opposite edge $E_{12}$. One can check that the conditions imply that the ray is contained in the half-plane. The discussion above can be reinterpreted as saying that, most of the time, the half-plane is determined by the ray, unless $\lambda_1^{F_0}=\lambda_2^{F_0}=0$, so $\lambda_3^{F_0}=1$, so the ray from $P_0$ goes along edge $E_{03}$, in which case the half-plane is unconstrained.
\end{example}

We now make explicit the above claims.
\begin{proposition}
  The map $\pi\colon\tl T\to T$ is a diffeomorphism between the interiors of $\tl T$ and $T$.
\end{proposition}

\begin{proof}
  Given a point on the interior of $T$ with barycentric coordinates $\lambda_i$, we can invert $\pi$ by setting
  \begin{equation*}
    \lambda_i^K=\frac{\lambda_i}{\rho_K},
  \end{equation*}
  where $\rho_K=\rho_K^T=\sum_{i\in V_K}\lambda_i$. Observe that this map is smooth on the interior of $T$, since $\rho_K$ does not vanish there. The constraint $\sum_{i\in V_K}\lambda_i^K=1$ follows from the definition of $\lambda_i^K$. The constraint $\lambda_i^H=\rho_K^H\lambda_i^K$ reads
  \begin{equation*}
    \frac{\lambda_i}{\rho_H}=\rho_K^H\frac{\lambda_i}{\rho_K},
  \end{equation*}
  which is true because
  \begin{equation*}
    \rho_K^H=\sum_{i\in V_K}\lambda_i^H=\sum_{i\in V_K}\frac{\lambda_i}{\rho_H}=\frac{\rho_K}{\rho_H}.\qedhere
  \end{equation*}
\end{proof}

\begin{proposition}
  The blow-up Whitney forms $\psi_F$ can be pulled back to smooth forms $\tl\psi_F$ on $\tl T$.
\end{proposition}

\begin{proof}
  Let $F$ be a flag. Recall that, at a point $\blambda$, the coefficient $p_F$ of $\psi_F$ is the volume of the subset of the interior of $T$ defined by the equations
  \begin{equation*}
    \frac{r_{j-1}}{\rho_{j-1}}\le\frac{r_j}{\rho_j}.
  \end{equation*}
  Rearranging, adding one, and inverting, we obtain the equivalent inequality
  \begin{equation*}
    \frac{r_j}{r_{j-1}+r_j}\ge\frac{\rho_j}{\rho_{j-1}+\rho_j}=\rho^{H_j}_{K_j},
  \end{equation*}
  where $K_j$ has vertices $V_{j-1}$ and $H_j$ has vertices $V_{j-1}\cup V_j$. We conclude that $p_F$ is smooth in the $\rho^{H_j}_{K_j}$. The $\rho^{H_j}_{K_j}$ are not smooth on $T$, but they are smooth on $\tl T$, being just sums $\sum_{i\in V_K}\lambda_i^H$ of the coordinates of the ambient space. Hence, $p_F$ is smooth on $\tl T$.

  As for $\omega_F$, we show that $\omega_{V_j}$ is smooth on the blow-up. Let $K$ be the face with vertices $V_j$, so per our notation we have $V_K=V_j$ and $\rho_K=\rho_j$. As before, let $n_j=\dim K=\abs{V_j}-1$. Then the Whitney form in terms of the blow-up coordinates $\lambda_i^K$ is given by $n_j! i_{X_{\id}}(\d\lambda^K_K)$, where we extend to the orthant $\bR_{\ge0}^{V_K}$ and set $\d\lambda^K_K=\bigwedge_{i\in V_K}\d\lambda_i^K$. So, this form is smooth on the blow-up, and we will show that it is $\omega_{V_K}$ on the interior.

  To do so, we must use care when extending to the orthant. Essentially, we follow the proof of Proposition~\ref{prop:quasisphere}, with $V_K$ in place of $V$, $\rho_K$ in place of $\rho$, and $\lambda_i^K$ in place of $\theta_i$. So set $\sigma=\sum_{i\in V_K}\lambda_i^K$; in the end, we care about the set $\sigma=1$. But, while working on the extended space, in place of the equation $\lambda_i=\rho_K\lambda_i^K$ we will use the equation $\sigma\lambda_i=\rho_K\lambda_i^K$. Observe that $X_{\id}[\sigma]=\sigma$ and $X_{\id}[\rho_K]=\rho_K$, so $X_{\id}[\frac{\sigma}{\rho_K}]=0$. So then, using this fact along with $\lambda_i^K=\frac{\sigma}{\rho_K}\lambda_i$, we conclude that the above form is
  \begin{equation*}
    n_j!i_{X_{\id}}(\d\lambda^K_K)=n_j!\left(\tfrac\sigma{\rho_K}\right)^{\abs{V_K}}i_{X_{\id}}\d\lambda_K,
  \end{equation*}
  where $\d\lambda_K=\bigwedge_{i\in V_K}\d\lambda_i$. Restricting to $\sigma=1$ and recalling that $V_K=V_j$ and $\rho_K=\rho_j$, we see that we indeed have our previously-defined form $\omega_{V_j}$.

  So, both $p_F$ and $\omega_F$ are smooth on $\tl T$, and hence so is $\psi_F=p_F\omega_F$.
\end{proof}

\begin{proposition}\label{prop:faces}
  Each flag $F$ yields a $k$-dimensional face of $\tl T$. The subfaces of this face correspond to flags $F'$ that are formed by further subdividing the partition sets of $F$.

  Specifically, for a flag $F$, the face of $\tl T$ corresponding to the flag $F$ is given by the equations $\lambda_{i'}^{E_{ii'}}=0$ for every edge $E_{ii'}$ joining vertices $i$ and $i'$ with $i\in V_j$, $i'\in V_{j'}$ and $j<j'$.

  Moreover, this face is isomorphic to $\tl\Theta_F:=\prod_j\tl T_j$, where, as before $T_j$ denotes the simplex with vertex set $V_j$, and $\tl T_j$ is its blow up. In particular, the interior of the face is isomorphic to the interior of $\Theta_F$.
\end{proposition}

\begin{proof}
  We begin with the claimed incidence relation, which is easy to show. Observe that we have a constraint for every pair of vertices that are in different sets of the partition. Therefore, if $F'$ is formed by subdividing a flag $F$, the set of constraints grows, and so the face corresponding to $F'$ is a subset of the face corresponding to $F$.

  For the remaining claims, we must show that the equations $\lambda_{i'}^{E_{ii'}}=0$, along with a point in the $k$-dimensional space $\tl\Theta_F$, determine a unique point in $\tl T$. Before we compute, we give some intuition, via the Poisson process framework. As before, we think of the $\lambda_i$ as rates of Poisson processes, but now the rates in later $V_j$ are ``infinitesimal'' when compared to the rates in earlier $V_j$. Nonetheless, within a $V_j$, we can compare the rates to each other. The quantity $\lambda_i^K$ is, as before, the probability that particle $i$ is received before any other particles in $V_K$. Previously, this would just be $\frac{\lambda_i}{\sum_{i'\in V_K}\lambda_{i'}}$, but now we consider rates infinitesimal relative to $\lambda_i$ to be zero, and we also consider the case that $\lambda_i$ is itself infinitesimal with respect to another $\lambda_{i'}$ for $i'\in V_K$, in which case the probability is zero.

  So, let $K$ be an arbitrary face. Let $j$ be the smallest value such that $V_j$ contains a vertex of $K$. Let $L$ be the face of $K$ given by the vertices $V_K\cap V_j$, and let $L'$ be the (potentially empty) face of $K$ given by the remaining vertices $V_K\setminus V_j$, so $V_K=V_L\sqcup V_{L'}$, with $V_L$ nonempty. Since $L\le T_j$, we have the value of $\lambda_i^L$ from the given point in $\tl\Theta_F:=\prod_j\tl T_j$. So, we set $\lambda_i^K:=\lambda_i^L$ for $i\in V_L$. Meanwhile, for $i'\in V_{L'}$, we set $\lambda_{i'}^K:=0$. We easily verify that $\lambda_i^K:=\lambda_i^L$ is consistent notation in the case where $K=L$ due to $K\le T_j$ for some $j$.

  We must show that this choice of $\lambda_i^K$ defines a valid point in $\tl T$ that satisfies the equations, and that this choice is unique. Working first on uniqueness, note that for $i\in V_L$ and $i'\in V_{L'}$, we have, by construction, that $i$ is in an earlier partition set compared to $i'$, so we have the constraint $\lambda_{i'}^{E_{ii'}}=0$. Since $K$ contains $E_{ii'}$, the definition of $\tl T$ forces $\lambda_{i'}^K=\rho_{E_{ii'}}^K\lambda_{i'}^{E_{ii'}}=0$. We conclude then that the constraints force $\lambda_{i'}^K=0$ for all $i'\in V_{L'}$. Note that the argument relied on $V_L$ being nonempty, but that it is possible for $V_{L'}$ to be empty, in which case the statement is vacuous. So then, using this fact and the definition of $\tl T$, we have that $1=\sum_{i\in V_K}\lambda_i^K=\sum_{i\in V_L}\lambda_i^K=\rho_L^K$. So then, the definition of $\tl T$ forces, for $i\in V_L$, that $\lambda_i^K=\rho_L^K\lambda_i^L=\lambda_i^L$.

  We now check that our choice of $\lambda_i^K$ satisfies all of the constraints. First, in the case where $K=E_{ii'}$ with $i\in V_j$, $i'\in V_{j'}$, and $j<j'$, we have that $L$ is the vertex $i$ and $L'$ is the vertex $i'$, and so we set $\lambda_{i'}^{E_{ii'}}=\lambda_{i'}^K=0$ as required.

  Next, we verify the constraints in the definition of $\tl T$. First, we have that $\sum_{i\in V_K}\lambda_i^K=\sum_{i\in V_L}\lambda_i^L=1$, using the fact that $\lambda_{i'}^K=0$ for $i'\notin V_L$ and the fact that the $\lambda_i^L$ come from a valid point of $\tl T_j$ and hence must sum to one.

  So now it remains to verify the constraint that $\lambda_i^H=\rho_K^H\lambda_i^K$ for all $i\in V_K$ when $K\le H$. So now, let $j$ be the smallest value such that $V_j$ contains a vertex of $H$. Then either $V_j$ also contains a vertex of $K$, or it does not. If not, then $\lambda_i^H=0$ for all $i\in V_K$, so $\rho_K^H=0$, and the constraint reads $0=0$. If so, then $j$ is also the smallest value such that $V_j$ contains a vertex of $K$. So, as before, let $L$ have vertices $V_K\cap V_j$ and $L'$ have vertices $V_K\setminus V_j$, and now let $M$ have vertices $V_H\cap V_j$ and $M'$ have vertices $V_H\setminus V_j$. For $i'\in V_{L'}\subseteq V_{M'}$, by construction we have $\lambda_{i'}^H=\lambda_{i'}^K=0$, so the constraint is satisfied. Meanwhile, for $i\in V_L\subseteq V_M$, because $L\le M\le T_j$ and we have a valid point of $\tl T_j$, we have that $\lambda_i^M=\rho_L^M\lambda_i^L$. By construction, we have $\lambda_i^H=\lambda_i^M$ and $\lambda_i^K=\lambda_i^L$. In particular, $\rho_K^H=\sum_{i\in V_K}\lambda_i^H=\sum_{i\in V_L}\lambda_i^H=\sum_{i\in V_L}\lambda_i^M=\rho_L^M$. So, substituting into $\lambda_i^M=\rho_L^M\lambda_i^L$ we obtain $\lambda_i^H=\rho_K^H\lambda_i^K$.
\end{proof}

\begin{proposition}\label{prop:dof}
  Let $F$ be a flag, let $\tl\psi$ be a $k$-form on $\tl T$ that is smooth up to and including the boundary, and let $\psi$ be the corresponding $k$-form on $T$, smooth on the interior but generally discontinuous at the boundary. Then evaluating the degree of freedom $\Psi_F(\psi)$ as defined in Definition~\ref{def:dof} is equivalent to integrating $\tl\psi$ over the face of $\tl T$ corresponding to $F$.
\end{proposition}

\begin{proof}
  We first note the equivalence of the $\theta_i$ from the quasi-cylindrical coordinate system and the $\lambda_i^{T_j}$ from the definition of the blow-up. Indeed, the defining equations of the $\theta_i$ are $\lambda_i=\rho_j\theta_i$ for all $i\in V_j$, where $\rho_j=\sum_{i\in V_j}\lambda_i$. As shown above, the isomorphism on the interior between $T$ and $\tl T$ is given by $\lambda_i=\lambda_i^T$. Finally, the defining equations of $\tl T$ yield $\lambda_i^T=\rho_{T_j}^T\lambda_i^{T_j}$, where $\rho_{T_j}^T=\sum_{i\in V_j}\lambda_i^T$.

  So, the $\Theta_F$ discussed in the quasi-cylindrical coordinate section and the $\Theta_F$ discussed in the blow-up simplex section are the same. In particular, integrating over the face of $\tl T$ corresponding to $F$ is equivalent to integrating over $\tl\Theta_F$, which is equivalent to integrating over its interior, which is the same as integrating over the interior of $\Theta_F$.

  It remains to analyze the limit procedure in Definition~\ref{def:dof}. When we take the limit as $\rho_{n-k}\to0$ while leaving $\lambda_i$ for $i\notin V_{n-k}$ nonzero, the effect is to send $\lambda_{i'}^{E_{ii'}}=\frac{\lambda_{i'}}{\lambda_i+\lambda_{i'}}$ to zero whenever $i\notin V_{n-k}$ and $i'\in V_{n-k}$. Similarly, when we take the limit as $\rho_{j'}\to0$ while keeping $\lambda_i$ nonzero for $i\in V_j$ with $j<j'$, the effect is to send $\lambda_{i'}^{E_{ii'}}$ to zero whenever $i\in V_j$ with $j<j'$ and $i'\in V_{j'}$. So, the limiting procedure precisely lands us in the face of $\tl T$ corresponding to $F$. Moreover, the limiting procedure in Definition~\ref{def:dof} occurs pointwise for a fixed $\btheta\in\mr\Theta_F$; consequently, we land at the point in the interior of the face of $\tl T$ corresponding to this value of $\btheta$. Per the above isomorphism between the interior of the face of $\tl T$ corresponding to $F$ and $\mr\Theta_F$, we conclude that every point in the interior of the face is attained by this limit procedure.
\end{proof}

\begin{theorem}
  The cohomology of the complex of blow-up Whitney forms is zero except in degree zero, where it is generated by the constants.
\end{theorem}

\begin{proof}
  We view $\tl T$ as a cell complex, where the cells are the faces of $\tl T$ corresponding to the flags $F$, as discussed above. Viewing blow-up Whitney forms as smooth $k$-forms on $\tl T$, the integration pairing on $k$-cells yields a map between the space of blow-up Whitney $k$-forms and the space of cellular $k$-cochains. By unisolvence and the preceding proposition, this map is an isomorphism. Moreover, this map commutes with $d$, since Stokes's theorem tells us that $d$ is dual to the boundary operator $\partial$, which is the definition of $d$ on cellular cochains. Therefore, we have an isomorphism of cochain complexes, so we have an isomorphism of cohomology. Since $\tl T$ is topologically a ball, its cellular cohomology vanishes, except in degree zero, where it is generated by the constants.
\end{proof}

\subsection{Speculation on global cohomology of triangulations}
\subsubsection{Informal discussion}\label{subsubsec:informal}
In the preceding section, we showed that the complex of blow-up Whitney forms has the correct cohomology on a single simplex. For ordinary Whitney forms, we know that the complex has the right cohomology not only on a single simplex but also on a simplicial triangulation with many elements. Is the same true of blow-up Whitney forms? Answering this question is beyond the scope of this paper, but it turns out that even \emph{asking} the right question is subtle in the context of blow-up Whitney forms.

Constructing finite element spaces involves two ingredients. The first ingredient is understanding the space on a single element; it is this first ingredient that we have understood for blow-up Whitney forms. The second ingredient is imposing continuity conditions between elements, or, equivalently, understanding how to properly identify degrees of freedom on adjacent elements. For ordinary Whitney forms, doing so is easy: the degrees of freedom correspond to faces (of any dimension), so if two elements share a face, we identify those degrees of freedom. Equivalently, we insist that if elements $T$ and $T'$ share a face $K$, then restriction of $\psi$ on $T$ to $K$ should be equal to the restriction of $\psi$ on $T'$ to $K$. So, then, the complex of Whitney forms is once again isomorphic to the complex of simplicial cochains, this time globally, and so the cohomology of the Whitney complex is equal to the simplicial cohomology of the full space.

So then, how should we relate degrees of freedom for blow-up Whitney forms? Consider the case of scalar fields in two dimensions. If we identify all the degrees of freedom associated to a vertex, we will, in particular, identify degrees of freedom within a single simplex. As a result, we will obtain the usual Lagrange elements, and our theory will end up reducing to the ordinary Whitney complex.

So then, as discussed in the introduction, we could identify degrees of freedom corresponding to flags $P<E<T$ and $P<E<T'$, where $T$ and $T'$ are adjacent elements that share edge $E$; see Figure~\ref{fig:glue}, top left. (Recall from the introduction that each flag given by $(V_0,V_1,\dots,V_{n-k})$ is identified with the increasing sequence of faces $T_{V_0} < T_{V_0 \cup V_1} < \dots < T_{V_0 \cup V_1 \cup \dots \cup V_{n-k}}$.) We see that there are no identifications within a triangle, so we still have the blow-up space on each triangle. Moreover, being in the kernel of $d$ means being constant on each triangle, and these identifications are enough to enforce constancy on each connected component of the triangulation, so we have the right zeroth cohomology.

What about first cohomology? Intuitively, if we join blown up simplices along shared edges, we will still have ``holes'' at the vertices, so we expect to get the wrong first cohomology.  See Figure~\ref{fig:holes} for an illustration. To solve this problem, we need not only to identify degrees of freedom corresponding to flags $E<T$ and $E<T'$, but also declare that, for each vertex $P$, the degrees of freedom associated to $P<T$ must sum to zero over the triangles $T$ containing $P$. Intuitively, we are saying that the integral of the one-form around the ``hole'' is zero. Note that the gradient of a scalar field will always have zero integrals over these closed loops, so we still have a complex. But now we prevent these holes from contributing to the cohomology.

\begin{figure}
  \begin{tikzpicture}[scale=0.3]
    \begin{pgfonlayer}{nodelayer}
      \node [style=none] (0) at (0, 0) {};
      \node [style=none] (1) at (1, 0) {};
      \node [style=none] (2) at (1, 1) {};
      \node [style=none] (3) at (0, 1) {};
      \node [style=none] (4) at (0, 5) {};
      \node [style=none] (5) at (0, 6) {};
      \node [style=none] (6) at (1, 6) {};
      \node [style=none] (7) at (1, 7) {};
      \node [style=none] (8) at (0, 7) {};
      \node [style=none] (9) at (5, 0) {};
      \node [style=none] (10) at (6, 0) {};
      \node [style=none] (11) at (7, 0) {};
      \node [style=none] (12) at (7, 1) {};
      \node [style=none] (13) at (6, 1) {};
      \node [style=none] (14) at (5, 5) {};
      \node [style=none] (15) at (6, 5) {};
      \node [style=none] (16) at (7, 6) {};
      \node [style=none] (17) at (6, 6) {};
      \node [style=none] (18) at (5, 6) {};
      \node [style=none] (19) at (6, 7) {};
      \node [style=none] (20) at (7, 7) {};
      \node [style=none] (21) at (11, 0) {};
      \node [style=none] (22) at (12, 0) {};
      \node [style=none] (23) at (13, 0) {};
      \node [style=none] (24) at (13, 1) {};
      \node [style=none] (25) at (12, 1) {};
      \node [style=none] (26) at (17, 0) {};
      \node [style=none] (27) at (18, 1) {};
      \node [style=none] (28) at (18, 0) {};
      \node [style=none] (29) at (11, 5) {};
      \node [style=none] (30) at (12, 5) {};
      \node [style=none] (31) at (12, 6) {};
      \node [style=none] (32) at (11, 6) {};
      \node [style=none] (33) at (12, 7) {};
      \node [style=none] (34) at (13, 7) {};
      \node [style=none] (35) at (13, 6) {};
      \node [style=none] (36) at (17, 5) {};
      \node [style=none] (37) at (18, 5) {};
      \node [style=none] (38) at (18, 6) {};
      \node [style=none] (39) at (17, 6) {};
      \node [style=none] (40) at (18, 7) {};
      \node [style=none] (41) at (0, 11) {};
      \node [style=none] (42) at (1, 12) {};
      \node [style=none] (43) at (1, 13) {};
      \node [style=none] (44) at (0, 13) {};
      \node [style=none] (45) at (0, 12) {};
      \node [style=none] (46) at (0, 17) {};
      \node [style=none] (47) at (0, 18) {};
      \node [style=none] (48) at (1, 18) {};
      \node [style=none] (49) at (6, 11) {};
      \node [style=none] (50) at (5, 11) {};
      \node [style=none] (51) at (5, 12) {};
      \node [style=none] (52) at (6, 12) {};
      \node [style=none] (53) at (6, 13) {};
      \node [style=none] (54) at (7, 13) {};
      \node [style=none] (55) at (7, 12) {};
      \node [style=none] (56) at (11, 11) {};
      \node [style=none] (57) at (12, 11) {};
      \node [style=none] (58) at (12, 12) {};
      \node [style=none] (59) at (11, 12) {};
      \node [style=none] (60) at (13, 12) {};
      \node [style=none] (61) at (13, 13) {};
      \node [style=none] (62) at (12, 13) {};
      \node [style=none] (63) at (17, 11) {};
      \node [style=none] (64) at (18, 11) {};
      \node [style=none] (65) at (18, 12) {};
      \node [style=none] (66) at (17, 12) {};
      \node [style=none] (67) at (18, 13) {};
      \node [style=none] (68) at (5, 18) {};
      \node [style=none] (69) at (5, 17) {};
      \node [style=none] (70) at (6, 17) {};
      \node [style=none] (71) at (6, 18) {};
      \node [style=none] (72) at (7, 18) {};
      \node [style=none] (73) at (11, 17) {};
      \node [style=none] (74) at (11, 18) {};
      \node [style=none] (75) at (12, 18) {};
      \node [style=none] (76) at (12, 17) {};
      \node [style=none] (77) at (13, 18) {};
      \node [style=none] (78) at (17, 18) {};
      \node [style=none] (79) at (17, 17) {};
      \node [style=none] (80) at (18, 17) {};
      \node [style=none] (81) at (18, 18) {};
    \end{pgfonlayer}
    \begin{pgfonlayer}{edgelayer}
      \draw[thick] (3.center) to (2.center);
      \draw[thick] (2.center) to (14.center);
      \draw[thick] (14.center) to (18.center);
      \draw[thick] (18.center) to (6.center);
      \draw[thick] (6.center) to (4.center);
      \draw[thick] (4.center) to (3.center);
      \draw[thick] (2.center) to (1.center);
      \draw[thick] (1.center) to (9.center);
      \draw[thick] (9.center) to (13.center);
      \draw[thick] (13.center) to (15.center);
      \draw[thick] (15.center) to (14.center);
      \draw[thick] (13.center) to (12.center);
      \draw[thick] (12.center) to (29.center);
      \draw[thick] (29.center) to (32.center);
      \draw[thick] (32.center) to (16.center);
      \draw[thick] (16.center) to (15.center);
      \draw[thick] (12.center) to (11.center);
      \draw[thick] (11.center) to (21.center);
      \draw[thick] (21.center) to (25.center);
      \draw[thick] (25.center) to (30.center);
      \draw[thick] (30.center) to (29.center);
      \draw[thick] (25.center) to (24.center);
      \draw[thick] (24.center) to (36.center);
      \draw[thick] (36.center) to (39.center);
      \draw[thick] (39.center) to (35.center);
      \draw[thick] (35.center) to (30.center);
      \draw[thick] (24.center) to (23.center);
      \draw[thick] (23.center) to (26.center);
      \draw[thick] (26.center) to (27.center);
      \draw[thick] (27.center) to (37.center);
      \draw[thick] (37.center) to (36.center);
      \draw[thick] (8.center) to (7.center);
      \draw[thick] (7.center) to (50.center);
      \draw[thick] (50.center) to (51.center);
      \draw[thick] (51.center) to (42.center);
      \draw[thick] (42.center) to (41.center);
      \draw[thick] (41.center) to (8.center);
      \draw[thick] (7.center) to (6.center);
      \draw[thick] (18.center) to (19.center);
      \draw[thick] (19.center) to (49.center);
      \draw[thick] (49.center) to (50.center);
      \draw[thick] (49.center) to (55.center);
      \draw[thick] (55.center) to (59.center);
      \draw[thick] (59.center) to (56.center);
      \draw[thick] (56.center) to (20.center);
      \draw[thick] (20.center) to (19.center);
      \draw[thick] (20.center) to (16.center);
      \draw[thick] (32.center) to (33.center);
      \draw[thick] (33.center) to (57.center);
      \draw[thick] (57.center) to (56.center);
      \draw[thick] (33.center) to (34.center);
      \draw[thick] (34.center) to (63.center);
      \draw[thick] (63.center) to (66.center);
      \draw[thick] (66.center) to (60.center);
      \draw[thick] (60.center) to (57.center);
      \draw[thick] (34.center) to (35.center);
      \draw[thick] (39.center) to (40.center);
      \draw[thick] (40.center) to (64.center);
      \draw[thick] (64.center) to (63.center);
      \draw[thick] (44.center) to (43.center);
      \draw[thick] (43.center) to (69.center);
      \draw[thick] (69.center) to (68.center);
      \draw[thick] (68.center) to (48.center);
      \draw[thick] (48.center) to (46.center);
      \draw[thick] (46.center) to (44.center);
      \draw[thick] (43.center) to (42.center);
      \draw[thick] (51.center) to (53.center);
      \draw[thick] (53.center) to (70.center);
      \draw[thick] (70.center) to (69.center);
      \draw[thick] (53.center) to (54.center);
      \draw[thick] (54.center) to (73.center);
      \draw[thick] (73.center) to (74.center);
      \draw[thick] (74.center) to (72.center);
      \draw[thick] (72.center) to (70.center);
      \draw[thick] (54.center) to (55.center);
      \draw[thick] (59.center) to (62.center);
      \draw[thick] (62.center) to (76.center);
      \draw[thick] (76.center) to (73.center);
      \draw[thick] (62.center) to (61.center);
      \draw[thick] (61.center) to (79.center);
      \draw[thick] (79.center) to (78.center);
      \draw[thick] (78.center) to (77.center);
      \draw[thick] (77.center) to (76.center);
      \draw[thick] (61.center) to (60.center);
      \draw[thick] (66.center) to (67.center);
      \draw[thick] (67.center) to (80.center);
      \draw[thick] (80.center) to (79.center);
    \end{pgfonlayer}
  \end{tikzpicture}
  \caption{If we glue together blown up triangles (see Figure~\ref{fig:blowuptri}) along shared edges, ``holes'' appear at the vertices.} \label{fig:holes}
\end{figure}
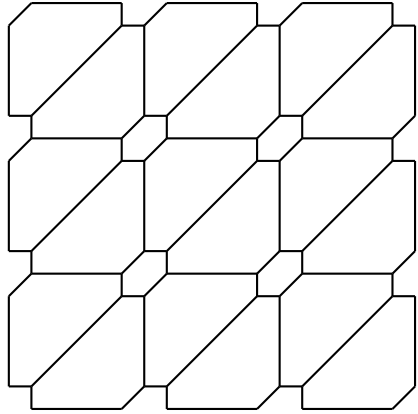

While so far, the identifications appear ad hoc, they do fit into a general framework. Around each vertex $P$, we have a topological circle $S^1$, and we insisted that integrals over it be zero. But, around each edge $E$, we have a topological zero-sphere $S^0$, a two-point set, where the two points have opposite sign. So, our identification for scalar fields can be equivalently viewed as insisting that the integral of the scalar field over this two-point set is zero. Our identification for one-forms at edges can also be viewed this way, but since we are integrating a one-form over a zero-dimensional set, we must view it in the context of Fubini's theorem, where the result of integration is itself a one-form, which we insist must be zero. Of course, going through the definitions, this condition just yields the usual requirement that the restrictions to the edge from both sides are equal.

\subsubsection{General framework for continuity conditions}
These observations suggest the following generalization. For each $d$-dimensional face $K$ with $d<n$, we obtain an $(n-d-1)$-dimensional sphere going around it, which we can denote $S_K$. Note that, within each element $T$, we can give coordinates to the portion of the sphere in $T$ by using $\lambda_i^{K^c}$, where $K^c$ is the face opposite to $K$, that is $V_{K^c}$ is the complement of $V_K$ in $V_T$. Thus, in the language of simplicial complexes, this $(n-d-1)$-dimensional sphere $S_K$ is the link of $K$, denoted $\Lk(K)$. Notably, the fact that the link is topologically a sphere is an explicit condition in the definition of piecewise linear (PL) manifolds; subtle counterexamples arise if this condition is not met.

In any case, if an element $T$ contains the face $K$, then considering the two-subset flag where $V_0=K$ and $V_1=K^c$, we have shown in Proposition~\ref{prop:faces} that $\tl T$ has an $(n-1)$-dimensional boundary face isomorphic to $\tl K\times\tl K^c$. We can call the disjoint union of these $\tl K\times\tl S_K$, because $\tl S_K$ is the blow-up of $S_K$. Of course, for the purposes of integration, we can ignore measure zero sets, so for these purposes we can think of this ``neighborhood'' of $K$ as $K\times S_K$. 

So, then, to construct the global complex of blow-up Whitney forms on a triangulation, we start with the discontinuous space by taking the direct sum over the elements of the triangulation of the complexes on each element, and then we impose the following continuity constraints. For each $d$-dimensional face $K$ with $d<n$, we can consider $\tl K\times \tl S_K$. Integrating a $k$-form over $\tl S_K$, we obtain a $k-(n-d-1)$-form on $\tl K$, and we insist that this $(k+d+1-n)$-form on $\tl K$ be zero. Explicitly, let $\psi$ be a discontinuous $k$-form, that is, an element of the direct sum of the blow-up complexes on each element. On each element $T$ containing $K$, we look at the boundary component of $\tl T$ that is isomorphic to $\tl K\times\tl K^c$. Assuming an orientation of the manifold and hence $T$ and $\tl T$, we have the boundary orientation of $\tl K\times\tl K^c$; assuming additionally a globally defined orientation of $K$, we obtain an orientation of $\tl K^c$ as well. We integrate $\psi$ over $\tl K^c$, obtaining a $(k+d+1-n)$-form on $\tl K$. Then we sum over all $T$ containing $K$. Our continuity condition is that the sum is zero.

\subsubsection{Continuity conditions in terms of degrees of freedom}
We expect that the above continuity condition is also what we will need for the higher-order generalizations, but for the blow-up Whitney forms, we can be more explicit in terms of the degrees of freedom. Fix a face $K$ of the triangulation of dimension $d<n$, and fix a flag $F$ on $V_K$ with $n-k$ subsets (one fewer than usual). Note that this requires that $n-k\le d+1$. For each element $T$ containing $K$, construct a flag $F_T$ by appending $V_T\setminus V_K$ to the flag $F$; this flag now has our usual $n-k+1$ subsets. Note that $V_T\setminus V_K$ is the vertex set of what we called $K^c$ above, and note that this operation on flags is the reverse of the dimension-reduction operation in Proposition~\ref{prop:reducedim}.

So then, with appropriate signs/orientations, we insist that the sum over $T$ of the degrees of freedom $\Psi_{F_T}$ is zero. If, as above, we hold fixed the orientations of all faces of the triangulation, then we can take the sign to be the sign of the ordering of $V_T$ given by a positive ordering of $V_K$ followed by a positive ordering of $V_T\setminus V_K$. We then impose this condition for every face $K$ of the triangulation and every flag $F$ on $V_K$.

We will now give explicit examples for how to implement these continuity conditions, matching the informal discussion of \ref{subsubsec:informal}.

\begin{example}[Scalar fields]
  Recall that the continuity conditions are vacuous unless $n-k\le d+1$. For scalar fields, $k=0$, so we only get continuity conditions for faces $K$ of dimension $d=n-1$. So then $S_K$ is the $0$-sphere, and so, as in the two-dimensional discussion above, the continuity condition is equivalent to simply identifying the degrees of freedom for matching flags $K_0<K_1<\dotsb<K_{n-1}=K<T$ and $K_0<K_1<\dotsb<K_{n-1}=K<T'$ of the two adjacent elements $T$ and $T'$ on either side of hyperface $K$. In summary, we have
  \begin{itemize}
  \item One degree of freedom per incidence $K_0<K_1<\dotsb<K_{n-1}$ in the triangulation, where each $K_j$ has dimension $j$.
  \end{itemize}
\end{example}

\begin{example}[Two dimensions]
  For the purposes of this discussion, we will use the letters $P$, $E$, and $T$ to refer to faces of dimension $0$, $1$, and $2$, respectively, of the original global triangulation. As discussed above, for scalar fields, we have one degree of freedom per incidence $P<E$.

  Moving on to one-forms, we have $k=1$, so we get continuity conditions for flags on $K$ with $n-k=1$ subset. So, the flag is determined by $K$, and there are two distinct possibilities, one where $K=E$, and one where $K=P$.

  The case $K=E$ corresponds to the long edges in Figure~\ref{fig:holes}. So then $S_K$ is the $0$-sphere, and the continuity condition amounts to identifying degrees of freedom $E<T$ and $E<T'$ for the two elements on either side of $E$.

  The more interesting case $K=P$ corresponds to the short edges around a hole in Figure~\ref{fig:holes}. So then $S_K$ is the $1$-sphere, informally, the circle around the hole. Each of these edges is given by $P<T$ where $T$ is an element containing the vertex $P$, and the condition is that the sum of the degrees of freedom on each of these edges is zero, that is, $\sum_{T>P}\Psi_{P<T}=0$. As a result of this condition, for the global complex, we have \emph{one less} degree of freedom than the number of edges around a vertex. In terms of implementation, to each vertex $P$, assign a \emph{base element} $T_P$. We then have a degree of freedom per every incidence $P<T$ where $T\neq T_P$, and then we compute $\Psi_{P<T_P}$ using the condition $\sum_{T>P}\Psi_{P<T}=0$.

  To summarize:
  \begin{itemize}
  \item For zero-forms, one degree of freedom per incidence of the form $P<E$.
  \item For one-forms:
    \begin{itemize}
    \item One degree of freedom per edge $E$, and
    \item One degree of freedom per incidence $P<T$ for $T\neq T_P$, where
      \begin{itemize}
      \item for each vertex $P$, $T_P$ is an arbitrarily assigned \emph{base element} containing $P$, and
      \item when computing the shape function on $T_P$, we set
        \begin{equation*}
          \Psi_{P<T_P}=-\sum_{\substack{T>P\\T\neq T_P}}\Psi_{P<T}.
        \end{equation*}
      \end{itemize}
    \end{itemize}
  \item For two-forms, one degree of freedom per element $T$.
  \end{itemize}
\end{example}

\begin{example}[Three dimensions]
  For the purposes of this discussion, we will use the letters $P$, $E$, $F$, and $T$, to refer to faces of dimension $0$, $1$, $2$, and $3$, respectively, of the original global triangulation, so, for scalar fields, we have one degree of freedom per incidence $P<E<F$.

  There are now three types of one-form degrees of freedom, corresponding to the three types of edges in the blow-up tetrahedron. Referring to Figure~\ref{fig:blowuptet}, let us refer to the three types of edges in the blow-up of a tetrahedron as big edges, small edges, and tiny edges. We now have $n-k=2$, so we are looking at flags on $K$ with $2$ subsets.

  The big edges correspond to flags of the form $E<F=K$. In this case, $S_K$ is the $0$-sphere, and the continuity condition amounts to identifying degrees of freedom $E<F<T$ and $E<F<T'$ for the two elements on either side of $F$. The small edges are similar. They correspond to flags of the form $P<F=K$, and once again we identify degrees of freedom $P<F<T$ and $P<F<T'$.

  The tiny edges are more interesting. They correspond to flags of the form $P<E=K$, so now $S_K$ is the $1$-sphere, and the condition is that $\sum_{T>E}\Psi_{P<E<T}=0$, summing over all elements containing the edge $E$. Informally, if we perform the gluing illustrated in Figure~\ref{fig:holes} in three dimensions using blow-up tetrahedra (Figure~\ref{fig:blowuptet}), we will have cylindrical holes around each edge. Each end of this cylinder is a loop, and we require that degrees of freedom on tiny edges sum to zero around each loop. As above, in terms of implementation, we assign a base element $T_E$ to each edge $E$, and then we leave out the degrees of freedom of the form $P<E<T_E$, instead computing $\Psi_{P<E<T_E}$ from the continuity condition.

  Finally, let us consider two-forms. There are three types, corresponding to the three kinds of faces in Figure~\ref{fig:blowuptet}, which we will refer to as big hexagons, rectangles, and small hexagons. Here $n-k=1$, so the flags are the single face $K$.

  The big hexagons correspond to $K=F$, and there we simply identify degrees of freedom $F<T$ and $F<T'$ correpsonding to the two elements on either side of $F$.

  The rectangles correspond to $K=E$, and the condition is $\sum_{T>E}\Psi_{E<T}=0$, summing over all elements containing the edge $E$. In terms of the cylindrical holes, this condition states that the degrees of freedom of the rectangular faces that make up each cylinder sum to zero. We can implement the global degrees of freedom by leaving out $E<T_E$ as before.

  The small hexagons correspond to $K=P$. Now $S_K$ is a $2$-sphere. Informally, if we collapse the cylindrical holes discussed above, we still are left with spherical holes around each vertex. As a result of this collapse, the small hexagons become small triangles, and these small triangles make up these spheres around each vertex. The condition that $\sum_{T>P}\Psi_{P<T}=0$ states that, for each sphere, the degrees of freedom on these small hexagons/triangles sum to zero.

  In summary:
  \begin{itemize}
  \item For zero-forms, one degree of freedom per incidence $P<E<F$.
  \item For one-forms:
    \begin{itemize}
    \item One degree of freedom per incidence $E<F$,
    \item One degree of freedom per incidence $P<F$, and
    \item One degree of freedom per incidence $P<E<T$ for $T\neq T_E$, where
      \begin{itemize}
      \item for each edge $E$, $T_E$ is an arbitrarily assigned base element containing $E$, and
      \item when computing the shape function on $T_E$, we set
        \begin{equation*}
          \Psi_{P<E<T_E}=-\sum_{\substack{T>E\\T\neq T_E}}\Psi_{P<E<T}.
        \end{equation*}
      \end{itemize}
    \end{itemize}
  \item For two-forms:
    \begin{itemize}
    \item One degree of freedom per face $F$,
    \item One degree of freedom per incidence $E<T$ for $T\neq T_E$, where
      \begin{itemize}
      \item as before, for each edge $E$, $T_E$ is an arbitrarily assigned base element containing $E$, and
      \item when computing the shape function on $T_E$, we set
        \begin{equation*}
          \Psi_{E<T_E}=-\sum_{\substack{T>E\\T\neq T_E}}\Psi_{E<T}.
        \end{equation*}
      \end{itemize}
    \item One degree of freedom per incidence $P<T$ for $T\neq T_P$, where
      \begin{itemize}
      \item for each vertex $P$, $T_P$ is an arbitrarily assigned base element containing $P$, and
      \item when computing the shape function on $T_P$, we set
        \begin{equation*}
          \Psi_{P<T_P}=-\sum_{\substack{T>P\\T\neq T_P}}\Psi_{P<T}.
        \end{equation*}
      \end{itemize}
    \end{itemize}
  \item For three-forms, one degree of freedom per element $T$.
  \end{itemize}
\end{example}

\begin{remark}
  In the examples above, it appears that we do something different when $K$ has codimension one, since there is no need to choose a base element. But we could choose a base element if we wanted to. For example, for two-forms in three dimensions with face $F$ between elements $T$ and $T'$, we can choose $T'$ to be the base element. Since $T$ is the only non-base element containing $F$, there is just one degree of freedom, which is assigned to the incidence $F<T$. Then we set $\Psi_{F<T'}=-\Psi_{F<T}$. The sign is appropriate because $F$ is positively oriented with respect to one of $T$ and $T'$ and negatively oriented with respect to the other. We can see that the end result is equivalent to simply assigning a degree of freedom to $F$, as before.
\end{remark}

\begin{conjecture}
  Consider a PL manifold with a given triangulation. With the aforementioned continuity conditions, the space of blow-up Whitney forms on this triangulation is a differential complex whose cohomology matches the simplicial cohomology of the manifold.
\end{conjecture}

\subsection{Towards higher order complexes}\label{subsec:higherorder}
In finite element exterior calculus \cite{arnold2006finite}, the Whitney form complex $\cP_1^-\Lambda^k$ generalizes to complexes of differential forms with coefficients that are higher degree polynomials, namely the $\cP_r^-\Lambda^k$ and $\cP_{r-k}\Lambda^k$ complexes. Like the complex of Whitney forms, these complexes have the right cohomology. Also like Whitney forms, the degrees of freedom, that is, the basis elements of the dual space, are associated to faces of the triangulation. However, for these higher order spaces, there are generally multiple degrees of freedom per face, made explicit via a duality relationship between the $\cP$ and $\cP^-$ spaces.

We would like to similarly generalize the $b\cP_1^-$ complex of blow-up Whitney forms to higher degree. As discussed in the introduction, we expect that our two new perspectives, namely the Poisson process ensemble and the blow-up simplex $\tl T$, to be key. As with blow-up Whitney forms, we expect that the basis functions of these higher order spaces to be expressed as probabilities of sequences of arrival times. Likewise, we expect that the degrees of freedom for blow-up finite element exterior calculus will be associated to faces of the blow-up simplex $\tl T$, in contrast to ordinary finite element exterior calculus, where they are associated to faces of $T$. We expect that these degrees of freedom will be glued together as speculated in the previous section.

As a first step towards this goal, we focus on the scalar case $k=0$ and propose a candidate for $b\cP_r\Lambda^0$, the space of blow-up scalar fields of ``degree'' $r$, which specializes to the shadow form scalar space $b\cP_1\Lambda^0=b\cP_1^-\Lambda^0$ of this paper when $r=1$.

\begin{definition}
  As before, let $V$ be a set, and consider a collection of $\abs V$ Poisson particle sources, each with rate $\lambda_i$. A \emph{degree $r$ arrival experiment} is performed as follows.

  Receive the first $r$ particles, and record the number of particles received from each Poisson source. Then, silence/ignore the sources from which at least one particle was received. Then receive another $r$ particles. Repeat until all sources have been silenced.

  We call a possible outcome of such an experiment an \emph{arrival sequence}.
\end{definition}

\begin{definition}
  The probability of a particular arrival sequence is a function of the $\lambda_i$. We let $b\cP_r\Lambda^0$ be the span of these functions.
\end{definition}

\begin{conjecture}
  These probability functions are linearly independent and hence form a basis for $b\cP_r\Lambda^0$.
\end{conjecture}

\begin{definition}
  To each arrival sequence, we associate a flag by recording the sets of sources silenced at each step.
\end{definition}

The example of $b\cP_3\Lambda^0$ in two dimensions (with $V=\{0,1,2\}$) is presented in Table~\ref{tab:r3}. Similarly to $\cP_3\Lambda^0$, there is one degree of freedom per vertex, two per edge, and one per face. However, for $\cP_3\Lambda^0$, these are vertices, edges, and faces of $T$, whereas for $b\cP_3\Lambda^0$, these are vertices, edges, and faces of $\tl T$. As a result, there are $6+2\cdot6+1=19$ degrees of freedom total.

\begin{table}
  \begin{tabular}{c|c|c}
    flag / face of $\tl T$ & arrival sequence & probability / basis function\\
    \hline & & \\[-0.8em]
    $012$ & $000\mid111\mid222$ & $\displaystyle\left(\frac{\lambda_0^3}{\lambda_{012}^3}\right)\left(\frac{\lambda_1^3}{\lambda_{12}^3}\right)\left(\frac{\lambda_2^3}{\lambda_2^3}\right)$ \\[1em]
    $\{01\}2$ & $001\mid222$ & $\displaystyle3\left(\frac{\lambda_0^2\lambda_1}{\lambda_{012}^3}\right)\left(\frac{\lambda_2^3}{\lambda_2^3}\right)$ \\[1em]
    $\{01\}2$ & $011\mid222$ & $\displaystyle3\left(\frac{\lambda_0\lambda_1^2}{\lambda_{012}^3}\right)\left(\frac{\lambda_2^3}{\lambda_2^3}\right)$ \\[1em]
    $0\{12\}$ & $000\mid112$ & $\displaystyle3\left(\frac{\lambda_0^3}{\lambda_{012}^3}\right)\left(\frac{\lambda_1^2\lambda_2}{\lambda_{12}^3}\right)$ \\[1em]
    $0\{12\}$ & $000\mid122$ & $\displaystyle3\left(\frac{\lambda_0^3}{\lambda_{012}^3}\right)\left(\frac{\lambda_1\lambda_2^2}{\lambda_{12}^3}\right)$ \\[1em]
    $\{012\}$ & $012$     & $\displaystyle6\left(\frac{\lambda_0\lambda_1\lambda_2}{\lambda_{012}^3}\right)$ \\[1em]
  \end{tabular}
  \caption{A proposed space $b\cP_3\Lambda^0$ of blow-up scalars of degree $r=3$ in dimension $n=2$. We list the basis functions associated to four of the flags on $\{0,1,2\}$; the other flags can be obtained by permuting the indices. As before $\rr{12}$ and $\rr{012}$ are shorthand for $\lambda_1+\lambda_2$ and $\lambda_0+\lambda_1+\lambda_2$, respectively. We may set $\lambda_{012}=1$ if we want to interpret the Poisson rates $\lambda_i$ as barycentric coordinates associated with the vertices of the triangle. Note also that the last factor in the first three rows is equal to $1$, but we include it for completeness.}\label{tab:r3}
\end{table}

It is very easy to show that the proposed higher order space of blow-up scalar fields contains the usual polynomial scalar fields.

\begin{proposition}
  The higher-order blow-up space $b\cP_r\Lambda^0$ contains the usual space $\cP_r\Lambda^0$ of polynomial scalar fields.
\end{proposition}

\begin{proof}
  Group the arrival sequences according to the first set of $r$ particles received. The sum of the corresponding probabilities / basis functions is the probability that a particular set of $r$ particles was received first. For given integers $r_i$ summing to $r$, the probability that the first set of $r$ particles received had $r_i$ particles from each source $i$ is
  \begin{equation*}
    r!\prod_{i\in V}\frac1{r_i!}\lambda_i^{r_i}.
  \end{equation*}
  Up to a constant factor, these functions are precisely the standard monomial/Bernstein basis for the space $\cP_r\Lambda^0$ of polynomials of degree at most $r$ on the simplex $T$, or, equivalently, the space of homogeneous polynomials of degree $r$ on $\bR^V$.
\end{proof}

Also, we have a blow-up analogue of the geometric decomposition of \cite{arnold2009geometric}. Namely, the flag $F$ associated to an arrival sequence is the unique minimal face of $\tl T$ on which the associated probability does not vanish.

\begin{proposition}
  Consider an arrival sequence with associated probability $p$ and flag $F$. Consider a face of $\tl T$ with associated flag $F'$. Then $p$ vanishes on the face $F'$ unless the face $F$ is a (not necessarily proper) subface of $F'$.
\end{proposition}

Note that we are abusing notation by interchangeably using $F$ for the face of $\tl T$ and for its associated flag. So, to be proper, per Proposition~\ref{prop:faces}, we should say that $p$ vanishes on the face unless the ordered partition $F$ is a subdivision of the ordered partition $F'$ (including the trivial subdivision where $F=F'$).

\begin{proof}
  For $i_1,i_2\in V$, we will say $i_1<_F i_2$ if $i_1\in V_{j_1}$, $i_2\in V_{j_2}$, and $j_1< j_2$. In other words, an ordered partition of $V$ defines a weak ordering on $V$, where vertices in subsets of the partition further to the left are smaller than vertices in subsets of the partition further to the right. We take similar notation for flag $F'$.

  We recall that, in terms of Poisson rates, the statement that we are on face $F'$ can be interpreted as saying that $\lambda_{i_1}\gg\lambda_{i_2}$ whenever $i_1<_{F'}i_2$, that is, rates later in the flag are infinitesimal relative to earlier rates; see the intuition discussion in the proof of Proposition~\ref{prop:faces}. If so, then we must silence/ignore source $i_1$ before we can receive a particle from source $i_2$, which means that $i_1<_Fi_2$ per the definition of the flag associated to an arrival sequence. In other words, presuming that $p$ does not vanish on $F'$, we have that $i_1<_{F'}i_2$ implies $i_1<_Fi_2$. Noting that it remains possible for $i_1\sim_{F'}i_2$ and $i_1<_Fi_2$, we see that the flag $F$ is a subdivision of the flag $F'$, as desired.
\end{proof}

We conjecture that these higher-order blow-up spaces for scalar fields can be generalized to differential forms in a way that preserves the key properties of the finite element exterior calculus spaces.

\begin{conjecture}
  On a simplex $T$, the shadow forms / blow-up Whitney form complex $b\cP_1^-\Lambda^k$ and the higher-order blow-up scalar fields $b\cP_r\Lambda^0$ can be generalized to a differential complex $b\cP_r^-\Lambda^k$ that is exact except at $k=0$.

  Moreover, these spaces can be \emph{blow-up geometrically decomposed} in the sense of \cite{arnold2009geometric}, that is, for $r\ge1$ we can construct bases where each $k$-form in the basis has a unique minimal face on which it does not vanish, but now this condition is about faces of the blow-up $\tl T$, not $T$.

  Moreover, for each flag/face $F$ of $\tl T$, consider the space $b\cP_r^-\Lambda^k(F)$ of the restriction of the full space to that face, and consider its subspace $b\mr{\cP}_r^-\Lambda^k(F)$ of forms that vanish on $\partial F$. Then the trace-vanishing blow-up spaces $b\mr{\cP}_r^-\Lambda^k(F)$ are isomorphic to the corresponding classical spaces $\mr{\cP}_r^-\Lambda^k(F)$. Here, the $\mr{\cP}_r^-\Lambda^k(F)$ are defined using the isomorphism $\mr F\cong\prod_j\mr T_j$ (see Proposition~\ref{prop:faces}), so the classical FEEC spaces $\mr{\cP}_r^-\Lambda^k(T_j)$ yield $\mr{\cP}_r^-\Lambda^k(F)$ via the tensor product construction \cite{arnold2015tensor}.
\end{conjecture}

One important note about the above conjecture is that, whereas the faces of $T$ are all simplices, the faces of $\tl T$, or, more precisely, their interiors, are \emph{products} of simplices. Products of simplices include products of intervals, namely cubes. So, as a welcome side effect, this conjecture may also yield to a unified perspective that encompasses finite element exterior calculus not only on simplicial meshes but also on \emph{cubical} meshes.

\section*{Acknowledgments}
Evan Gawlik was supported by NSF grants DMS-2012427 and DMS-2411208 and the Simons Foundation award MPS-TSM-00002615.  Yakov Berchenko-Kogan was supported by NSF grant DMS-2411209. We would also like to thank the anonymous referees for their feedback.

\printbibliography

\end{document}